\documentclass{article}

\usepackage[english]{babel}
\usepackage{amsmath,amsfonts,amssymb,amsthm,enumerate}
\usepackage{enumerate}

\usepackage{imakeidx}
\makeindex

\usepackage{symlist}

\usepackage{hyperref}

\usepackage{vmargin}
         \setmarginsrb{8mm}{15mm}{8mm}{9mm}%
         {0mm}{0mm}{0mm}{8mm}

\newtheorem{Lemma}{Lemma}[section]
\newtheorem{Theorem}[Lemma]{Theorem}
\newtheorem{Proposition}[Lemma]{Proposition}
\newtheorem{Corollary}[Lemma]{Corollary}

\newtheorem{question}[Lemma]{Question}

\newtheorem{claim}[Lemma]{Claim}

\theoremstyle{definition}
\newtheorem{Definition}[Lemma]{Definition}
\newtheorem{Example}[Lemma]{Example}
\newtheorem{Remark}[Lemma]{Remark}
\newtheorem{Convention}[Lemma]{Convention}

\numberwithin{equation}{section}
            
\input xy
\xyoption{all}

\newcommand{\F}{\mathcal F}
\newcommand{\V}{{\mathcal V}}
\newcommand{\W}{{\mathcal W}}
\newcommand{\sF}{{\mathcal F}}
\newcommand{\sH}{{\mathcal H}}
\newcommand{\La}{{\mathcal L}}
\newcommand{\sC}{{\mathcal C}}
\newcommand{\AG}{{\mathbf{AG}}}
\newcommand{\CT}{{\mathbf{CTop}}}
\newcommand{\Top}{{\mathbf{Top}}}

\newcommand{\ent}{{\mathrm{ent}}}
\newcommand{\CAG}{{\mathbf{CAG}}}
\newcommand{\MS}{{\mathbf{MesSp}}}
\newcommand{\Set}{{\mathbf{Set}}}

\newcommand{\Fr}{{\mathbf{Frm}}}

\newcommand{\h}{\mathbf h}
\renewcommand{\H}{\mathbf H}
\newcommand{\hf}{h_{\mathrm{fin\text{-}top}}}
\newcommand{\Hf}{H_{\mathrm{fin\text{-}top}}}
\newcommand{\fc}{{\mathrm{fin\text{-}\cov}}}
\newcommand{\cov}{\mathfrak{cov}}
\newcommand{\N}{\mathbb N}
\newcommand{\Z}{\mathbb Z}

\newcommand{\R}{\mathbb R}
\newcommand{\T}{\mathbb T}

\newcommand{\SL}{{\mathfrak{L}}}
\newcommand{\PSL}{\mathfrak{PL}}
\newcommand{\Se}{{\mathfrak{S}}}
\newcommand{\Sesu}{{\mathfrak{S}^\dag}}
\newcommand{\Mo}{{\mathfrak{M}}}
\newcommand{\X}{{\mathfrak{X}}}
\def\ent{\mathrm{ent}}

\newcommand{\sM}{\mathcal{H}}

\newcommand{\U}{{\mathcal U}}
\newcommand{\f}{\phi}

\newcommand{\ms}[1]{\mathcal{#1}}
\newcommand{\sub}{\mathfrak{sub}}
\def\cov{\mathfrak{cov}}
\def\pet{\mathfrak{pet}}
\def\atr{\mathfrak{im}}
\def\str{\mathfrak{cim}}

\newlength{\bibitemsep}
\newlength{\bibparskip}
\let\oldthebibliography\thebibliography
\renewcommand\thebibliography[1]{%
  \oldthebibliography{#1}%
  \setlength{\parskip}{\bibitemsep}%
  \setlength{\itemsep}{\bibparskip}%
}

\title{Entropy on normed semigroups\\(Towards a unifying approach to entropy)}

\author{Dikran Dikranjan \and Anna Giordano Bruno}
\date{}

\begin{document}

\maketitle

\abstract{
We present a unifying approach to the study of entropies in Mathematics, such as measure entropy, various forms of topological entropy, several notions of algebraic entropy, and two forms of set-theoretic entropy. We take into account only discrete dynamical systems, that is, pairs $(X,\phi)$, where $X$ is the underlying space (e.g., a probability space, a compact topological space, a group, a set) and $\phi:X\to X$ is a transformation of $X$ (e.g., a measure preserving transformation, a continuous selfmap, a group homomorphism, a selfmap).
We see entropies as functions $h:\mathfrak X\to \R_+$, associating to each flow $(X,\phi)$ of a category $\mathfrak \X$ either a non-negative real number or $\infty$.

First, we introduce the notion of semigroup entropy $h_\Se:\Se\to\R_+$, which is a numerical invariant attached to endomorphisms of the category $\Se$ of normed semigroups. Then, for a functor $F:\mathfrak X\to\Se$ from any specific category $\mathfrak X$ to $\Se$, we define the functorial entropy $\h_F:\mathfrak X\to\R_+$ as the composition $h_\Se\circ F$, that is, $\h_F(\phi) = h_\Se(F\phi)$ for any endomorphism $\phi: X \to X$ in $\mathfrak X$. Clearly, $\h_F$ inherits many of  the properties of $h_\Se$, depending also on the functor $F$. Motivated by this aspect, we study in detail the properties of $h_\Se$. 

Such a general scheme, using elementary category theory, permits one to obtain many relevant known entropies as functorial entropies $\h_F$, for appropriately chosen categories $\mathfrak X$ and functors $F:\mathfrak X\to\Se$. All of the above mentioned entropies are functorial. Furthermore, we exploit our scheme to elaborate a common approach to establish the properties shared by those entropies that we find as functorial entropies, pointing out their common nature. We give also a detailed description of the limits of our approach, namely entropies which cannot be covered. 

Finally, we discuss and deeply analyze the relations between pairs of entropies through the looking glass of our unifying approach. To this end we first formalize the notion of Bridge Theorem between two entropies $h_1:\mathfrak X_1\to \R_+$ and $h_2:\mathfrak X_2\to \R_+$ with respect to a functor $\varepsilon:\mathfrak X_1\to\mathfrak X_2$, taking inspiration from the known relation between the topological and the algebraic entropy via the Pontryagin duality functor. Then, for pairs of functorial entropies we use the above scheme to introduce the notion and the related scheme of Strong Bridge Theorem.
It allows us to shelter various relations between pairs of entropies under the same umbrella (e.g., the above mentioned connection of the topological and the algebraic entropy, as well as their relation with the set-theoretic entropy).}

\bigskip
\noindent \emph{Keywords}: entropy, semigroup, semilattice, normed semigroup, semigroup entropy, bridge theorem, algebraic entropy, topological entropy, measure entropy, frame entropy, set-theoretic entropy, duality.

\smallskip
\noindent\emph{2010 Mathematics Subject Classification}: 16B50, 20M15, 20K30, 20F65, 22D05, 22D35, 22D40, 28D20, 37A35, 37B40, 54C70, 54H11, 55U30.

\bigskip
\noindent The authors are supported by the project PRID Topological, Categorical and Dynamical Methods in Algebra (DMIF - University of Udine).
This work was partially supported by the ``National Group for Algebraic and Geometric Structures, and their Applications'' (GNSAGA - INdAM).

\bigskip
\bigskip
\bigskip
\noindent Dikran Dikranjan\\
Dipartimento di Scienze Matematiche, Informatiche e Fisiche\\ Universit\`a degli Studi di Udine\\
Via delle Scienze 206, 33100 Udine (Italy)\\
E-mail: dikran.dikranjan@uniud.it

\bigskip
\noindent Anna Giordano Bruno\\
Dipartimento di Scienze Matematiche, Informatiche e Fisiche\\ Universit\`a degli Studi di Udine\\
Via delle Scienze 206, 33100 Udine (Italy)\\
E-mail: anna.giordanobruno@uniud.it

\newpage

\tableofcontents


\newpage

\section{Introduction}

\subsection{Historical background}

The first notion of entropy in mathematics was the measure entropy $h_{mes}$ for measure preserving transformations of probability spaces, studied by Kolmogorov \cite{K} and Sinai \cite{Sinai} in ergodic theory. Analogously, Adler, Konheim and McAndrew \cite{AKM} introduced the topological entropy $h_{top}$ for continuous selfmaps of compact spaces by using open covers. Later, Hofer \cite{Hof} proposed a quite natural extension $\hf$ of the topological entropy $h_{top}$ to continuous selfmaps of arbitrary topological spaces, by simply replacing the open covers by finite open covers. 

Another notion of topological entropy $h_B$ for uniformly continuous selfmaps of metric spaces was given by Bowen \cite{B} and Dinaburg \cite{Din};  it coincides with $h_{top}$ on compact metric spaces. Later on, Hood \cite{hood} extended Bowen-Dinaburg's entropy to uniformly continuous selfmaps of uniform spaces. This notion of entropy is sometimes called {uniform entropy}, and it coincides with the topological entropy in the compact case, when the given compact topological space is endowed with the unique uniformity compatible with the topology (see \cite{DK,DSV} for more detail). In particular, this topological entropy can be studied for continuous endomorphisms of topological groups.

Indeed, the topological entropy was thoroughly studied for continuous endomorphisms of compact groups, starting from the work of Yuzvinski \cite{Y}, where the so-called Addition Theorem\index{Addition Theorem} was proved, and also the so-called Yuzvinski's formula\index{Yuzvinski's formula} relating the topological entropy with the Mahler measure. Moreover, a Uniqueness Theorem\index{Uniqueness Theorem} is available in this case due to Stoyanov \cite{S}. Recently, the topological entropy was studied for totally disconnected locally compact groups in \cite{DG-tdlc,GBV}, and for totally bounded abelian groups in \cite{Dik+Manolo}.

Entropy was taken to Algebraic Dynamics by Adler, Konheim and McAndrew \cite{AKM} and by Weiss \cite{W}; they studied the algebraic entropy $\ent$ for endomorphisms of torsion abelian groups, which was further investigated in \cite{DGSZ}, where in particular an Addition Theorem and a Uniqueness Theorem were provided. Then Peters \cite{P} defined its extension $h_{alg}$ to automorphisms of abelian groups; finally in \cite{DG0,DG-islam,DG-bridge,DG} the algebraic entropy $h_{alg}$ for group endomorphisms was introduced in general, developing all its fundamental properties, with the Addition Theorem playing a crucial role among them. In particular, the relation of the algebraic entropy with Lehmer's problem from number theory was pointed out, by means of the so-called algebraic Yuzvinski's formula (see \cite{GV,GV2}).

Peters \cite{Pet1} gave a further generalization of his notion of entropy $h_{alg}$ for continuous automorphisms of locally compact abelian groups, which was recently extended by Virili \cite{V} to continuous endomorphisms; the commutativity can be removed as noted in \cite{DG-islam}. The recent paper \cite{GBST} is dedicated to the study of the algebraic entropy for a class of locally compact not necessarily abelian groups  where the Addition Theorem is available.

 As a dual notion of the algebraic entropy $\ent$, the adjoint algebraic entropy $\ent^\star$ for group endomorphisms was investigated in \cite{DGS} (see also \cite{GK,SZ1}), and its topological version in \cite{G}. A notion of algebraic entropy for module endomorphisms was introduced in \cite{SZ}, namely, for $i$ an invariant of a module category the algebraic $i$-entropy $\ent_i$ for module endomorphisms. This entropy was deeply investigated in case $i$ is a length function, when the Addition Theorem holds (see \cite{SVV,SV1}).  The adjoint version of the algebraic $i$-entropy was studied in \cite{Vi}. In particular, the algebraic dimension entropy $\ent_{\dim}$ for discrete vector spaces was thoroughly analyzed in \cite{GBS}, and carried to locally linearly compact vector spaces in \cite{CGBalg}. A topological dimension entropy $\ent_{\dim}^\star$ was studied for locally linearly compact vector spaces in \cite{CGBtop}.

Finally, one can find in \cite{AZD} and \cite{DG-islam} two mutually ``dual'' notions of entropy for selfmaps of sets, namely the covariant set-theoretic entropy $\mathfrak h$ and the contravariant set-theoretic entropy $\mathfrak h^*$.

\smallskip
We briefly mention here that nowadays notions of entropy are studied also in much more general contexts, namely for (semi)group actions, for example, the measure entropy and the topological entropy for amenable group actions (e.g., see \cite{Kie,Ki,LSW,Oll,OW,ST}; analogues of the Addition Theorem and the Yuzvinski's formula can be found in \cite{CT,De,Li,LSW,S}). The algebraic entropy for amenable group actions was recently considered in \cite{LL,V2,V3}.
For amenable semigroup actions, the measure entropy and the topological entropy were introduced in \cite{CCK} and the algebraic entropy in \cite{DFGB}.
For a survey on entropy in the very general case of sofic groups see \cite{Weisss} (see also \cite{Bo0,Bo,KH}).
In another direction, entropy was extended to actions of finitely generated semigroups using regular systems in \cite{HS} (see also \cite{Bis,GLW} and \cite{BDGBS1,BDGBS}).

\subsection{The general scheme}\label{list}

We list here all the entropies that we take into account:
\begin{enumerate}[-]
\item the covariant set-theoretic entropy $\mathfrak h$ and the contravariant set-theoretic entropy $\mathfrak h^*$ (see \S\ref{set-sec});
\item the topological entropy $h_{top}$ for continuous selfmaps of compact spaces (see \S\ref{htop-sec}) or for continuous endomorphisms of locally compact groups (see \S\ref{NewSec2});
\item the topological entropy $\hf$ for continuous selfmaps of topological spaces  (see \S\ref{fr-sec} and \S\ref{lintop-sec});
\item the frame entropy $h_{fr}$ for endomorphisms of frames (see \S\ref{fr-sec});
\item the measure entropy $h_{mes}$ for measure preserving transformations of probability spaces (see \S\ref{mes-sec});
\item the algebraic entropy $\ent$ for endomorphisms of torsion abelian groups (see \S\ref{sec:h_alg});
\item the algebraic entropy $h_{alg}$ for group endomorphisms (see \S\ref{sec:h_alg}) and for continuous endomorphisms of locally compact groups (see \S\ref{NewSec2});
\item the algebraic $i$-entropy $\ent_i$ for endomorphisms of modules (see \S\ref{i-sec});
\item the adjoint algebraic entropy $\ent^\star$ for group endomorphisms (see \S\ref{adj-sec});
\item the algebraic dimension entropy $\ent_{\dim}$ for discrete vector spaces (see \S\ref{V-sec});
\item the topological dimension entropy $\ent^\star_{\dim}$ for linearly compact vector spaces (see \S\ref{V-sec}).
\end{enumerate}

Each of the above listed entropies has its specific definition, usually given by a limit computed on some ``trajectories'' and then by taking the supremum of these quantities (we will see their definitions explicitly in \S\ref{known-sec}). Their basic properties are very similar, but the proofs in the literature take into account the particular features of the specific case each time. Since it appears that all these definitions and basic properties share a lot of common features, the aim of our approach is to ``unify'' them under a general scheme, pointing out their common nature. To this end we need some easy tools from category theory. 

\medskip
Let $\mathfrak X$ be a category. A  \emph{flow}\index{flow} of $\mathfrak X$ is a pair $(X,\phi)$, where $X$ is an object and $\phi: X\to X$ is an endomorphism in $\mathfrak X$.
A morphism between two flows $(X,\phi)$ and $(Y,\psi)$ of $\mathfrak X$ is a morphism $\alpha: X \to Y$ in $\mathfrak X$ such that  the diagram
$$\xymatrix{X\ar[r]^{\phi} \ar[d]_{\alpha}&X\ar[d]^{\alpha}\\Y\ar[r]_{\psi}& Y.}$$
commutes. This defines the category \newsym{category of flows of $\mathfrak X$}{$\mathbf{Flow}_{\mathfrak X}$} of flows of $\mathfrak X$. 

If $\alpha$ is an epimorphism we say that $(Y,\psi)$ is a \emph{factor}\index{factor flow} of $(X,\phi)$, if $\alpha$ is a monomorphism then $(X,\phi)$ is a \emph{subflow}\index{subflow} of $(Y,\psi)$, and if $\alpha$ is an isomorphism we say that the flows $(X,\phi)$ and $(Y,\psi)$ are \emph{isomorphic}\index{isomorphic flows} and the morphisms $\phi$ and $\psi$ are \emph{conjugate}\index{conjugate morphisms}. 

\begin{Example}
Given a commutative ring $R$ and letting $\mathfrak X$ be the category \newsym{category of modules over the ring $R$}{$\mathbf{Mod}_R$} of $R$-modules, the category $\mathbf{Flow}_{\mathfrak X}$ is equivalent to the category $\mathbf{Mod}_{R[X]}$ of modules over the ring of polynomials $R[X]$ with coefficients in $R$.
\end{Example} 

Let \newsym{set of non-negative reals}{$\R_{\geq 0}$}$= \{r\in \R: r\geq 0\}$ and \newsym{set of positive reals}{$\R_+$}$= \R_{\geq 0}\cup \{\infty\}$. To classify flows of a category $\mathfrak X$ up to isomorphisms one can use  $\R_+$-valued invariants, that is, functions 
\begin{equation}\label{dag}
h: \mathbf{Flow}_{\mathfrak X}\to \R_+,
\end{equation}
which take the same values on isomorphic flows. We generally refer to such invariants as \emph{entropies}\index{entropy} or \emph{entropy functions}\index{entropy function} of $\mathfrak X$. For simplicity and with some abuse of notation, we adopt the following 

\begin{Convention}
If $\mathfrak X$ is a category and $h$ an entropy of $\mathfrak X$ we write (with some abuse of notation) $$h: {\mathfrak X}\to \R_+$$ in place of $h: \mathbf{Flow}_{\mathfrak X}\to \R_+$ as in \eqref{dag}.
\end{Convention}

\medskip
In order to pursue our unifying aim, in \S\ref{hs-sec} we introduce a general notion of \emph{semigroup entropy} $$h_{\Se*}: \Se^* \to \R_+$$
on the category $\Se^*$ of normed semigroups and semigroup homomorphisms. 

Most of the time we are involved with the non-full subcategory $\Se$ of $\Se^*$ consisting of normed semigroups and contractive semigroup homomorphisms; in this case we denote the semigroup entropy by $$h_{\Se}: \Se \to \R_+.$$

The next step towards the main aim is done in \S\ref{f-sec}, where for a category $\mathfrak X$ and a functor $F:\mathfrak X\to\Se$ (respectively, $F:\mathfrak X\to\Se^*$), the \emph{functorial entropy} $\h_F$ of $\mathfrak X$ is defined to be $$\h_F=h_{\Se} \circ F:\mathfrak X\to \R_+$$
(respectively, $\h_F=h_{\Se^*} \circ F:\mathfrak X\to \R_+$),
as described by the following diagram.
\begin{equation}\label{schemeq}
\xymatrix@R=6pt@C=37pt
{\mathfrak{X}\ar[dd]_{F}\ar[rrd]^{h=\h_F}	&	&	\\
& &	\R_+		\\
\Se\ar[rru]_{h_\Se}						&	&
}
\end{equation}

\medskip
Finally, in \S\ref{known-sec} and in \S\ref{NewSec2} all specific entropies listed above are obtained in our scheme as functorial entropies. 

\smallskip
For every specific functor $F$ that we consider, an order is available on the normed semigroup in the target of $F$, and the norm of the semigroup is monotone with respect to this order; this circumstance is exploited in the section dealing with the Bridge Theorem, in order to introduce the Strong Bridge Theorem, where the functorial entropies are involved. 

\smallskip
All the entropies listed at the beginning of this section can be obtained as functorial entropies $\h_F$ with respect to functors $F$ with target $\Se$, except the two considered in \S\ref{NewSec2} and the contravariant set-theoretic entropy $\mathfrak h^*$, for which the larger category $\Se^*$ is needed.
Furthermore, beyond the category $\Se$, we often make use of several subcategories of $\Se$ (see the diagram in \eqref{manycat}). 

In fact, when the target of the specific functor $F$ is $\Se$, the norm is subadditive in all specific cases considered in \S\ref{known-sec}, then one ends up in the full subcategory $\Se^\dag$ of $\Se$ consisting of all normed semigroups with subadditive norm.
Actually, most of the time the target of $F$ is the subcategory $\mathfrak L^\dag$ of $\Se^\dag$ of normed semilattices with subadditive norm, that is the bottom in the diagram of categories \eqref{manycat}.
In addition, with respect to the order naturally available in semilattices, the semigroup homomorphisms are monotone. 

To describe the remaining few cases, let $\Mo_p^\dag$ denote the subcategory of $\Se^\dag$ consisting of all normed preordered monoids with subadditive norm, and let $\PSL^\dag$ be its full subcategory of all normed presemilattices with subadditive norm. 
The target of the functor is $\PSL^\dag$ for the topological entropies, and $\Mo_p^\dag$ for the algebraic entropy considered on the category of all discrete groups. 

Clearly, we make use of this additional wealth of useful properties of the semigroup and of its norm in order to exploit better the semigroup entropy, and consequently the functorial entropy. 
Anyway, we thickly underline that this additional structure, which is present in most of the specific cases, is not needed for the general setting. 

\bigskip
The category $\Se$ and all its above mentioned subcategories are studied in \S\ref{Se-sec}. Here, we list all the basic properties of the semigroup entropy, clearly inspired by those of the known entropies, that we prove in \S\ref{hs-sec}:
\begin{enumerate}[-]
\item Monotonicity for factors;
\item Invariance under conjugation;
\item Invariance under inversion;
\item Logarithmic Law;
\item Monotonicity for subflows;
\item Continuity for direct limits;
\item Vanishing on quasi-periodic flows;
\item Weak Addition Theorem.
\end{enumerate}
It is in \S\ref{f-sec} that we show that the basic properties of $\h_F$ can be deduced from those of $h_\Se$, by taking into account the properties of the specific functor $F$.

\medskip
To obtain all specific entropies in our scheme in \S\ref{known-sec} and \S\ref{NewSec2}, we deal with the categories:
\begin{enumerate}[-]
\item \newsym{category of probability measure spaces and measure preserving transformations}{$\mathbf{Mes}$} of probability measure spaces and measure preserving transformations (for $h_{mes}$);
\item \newsym{category of topological spaces and continuous maps}{$\mathbf{Top}$} of topological spaces and continuous maps (for $\hf$);
\item \newsym{category of $T_0$ topological spaces and continuous maps}{$\mathbf{Top}_0$} of $T_0$ topological spaces and continuous maps (for $\hf$);
\item \newsym{category of $T_1$ topological spaces and continuous maps}{$\mathbf{Top}_1$} of $T_1$ topological spaces and continuous maps;
\item \newsym{category of compact topological spaces and continuous maps}{$\mathbf{CTop}$} of compact topological spaces and continuous maps (for $h_{top}$);
\item \newsym{category of compact Hausdorff spaces and continuous maps}{$\mathbf{CTop}_2$} of compact Hausdorff spaces and continuous maps (for $h_{top}$);
\item \newsym{category of linearly topologized precompact groups and continuous homomorphisms}{$\mathbf{LPG}$} of linearly topologized precompact groups and continuous homomorphisms (for $\hf$);
\item \newsym{category of totally disconnected compact groups and continuous homomorphisms}{$\mathbf{TdCG}$} of totally disconnected compact groups and continuous homomorphisms (for $h_{top}$);
\item \newsym{category of frames and frame homomorphisms}{$\Fr$} of frames and frame homomorphisms (for $h_{fr}$);
\item \newsym{category of groups and group homomorphisms}{$\mathbf{Grp}$} of groups and group homomorphisms (for $h_{alg}$ and $\ent^\star$);
\item \newsym{category of locally compact groups and continuous homomorphisms}{$\mathbf{LCG}$} of locally compact groups and continuous homomorphisms (for $h_{alg}$ and $h_{top}$);
\item \newsym{category of locally compact abelian groups and continuous homomorphisms}{$\mathbf{LCA}$} of locally compact abelian groups and continuous homomorphisms (for $h_{alg}$ and $h_{top}$);
\item \newsym{category of compact abelian groups and continuous homomorphisms}{$\mathbf{CAG}$} of compact abelian groups and continuous homomorphisms;
\item \newsym{category of totally disconnected compact abelian groups and continuous homomorphisms}{$\mathbf{TdCAG}$} of totally disconnected compact abelian groups and continuous homomorphisms;
\item \newsym{category of abelian groups and group homomorphisms}{$\mathbf{AG}$} of abelian groups and group homomorphisms (for $\ent$ and $h_{alg}$);
\item \newsym{category of torsion abelian groups and group homomorphisms}{$\mathbf{TAG}$} of torsion abelian groups and group homomorphisms (for $\ent$);
\item \newsym{category of right modules over a ring $R$ and $R$-module homomorphisms}{$\mathbf{Mod}_R$} of right modules over a ring $R$ and $R$-module homomorphisms (for $\ent_i$);
\item \newsym{category of discrete vector spaces over a field $\mathbb K$ and linear transformations}{$\mathbf{Mod}_\mathbb K$} of discrete vector spaces over a field $\mathbb K$ and linear transformations (for $\ent_{\dim}$);
\item \newsym{category of linearly compact vector spaces over a discrete field $\mathbb K$ and continuous linear transformations}{$\mathbf{LCVect}_\mathbb K$} of linearly compact vector spaces over a discrete field $\mathbb K$ and continuous linear transformations (for $\ent^\star_{\dim}$);
\item \newsym{category of sets and maps}{$\mathbf{Set}$} of sets and maps (for $\mathfrak h$);
\item \newsym{category of sets and finite-to-one maps}{$\mathbf{Set}_{\mathrm{fin}}$} of sets and finite-to-one maps (for $\mathfrak h^*$).
\end{enumerate}

We dedicate to each specific entropy a subsection of \S\ref{known-sec} and of \S\ref{NewSec2}, each time giving explicitly the relevant functor that permits one to obtain the given entropy as a functorial entropy, and describing the basic properties of the specific entropy and how to deduce them from those of the semigroup entropy.
Some of these functors and the known entropies obtained as functorial entropies are summarized by the following diagram.
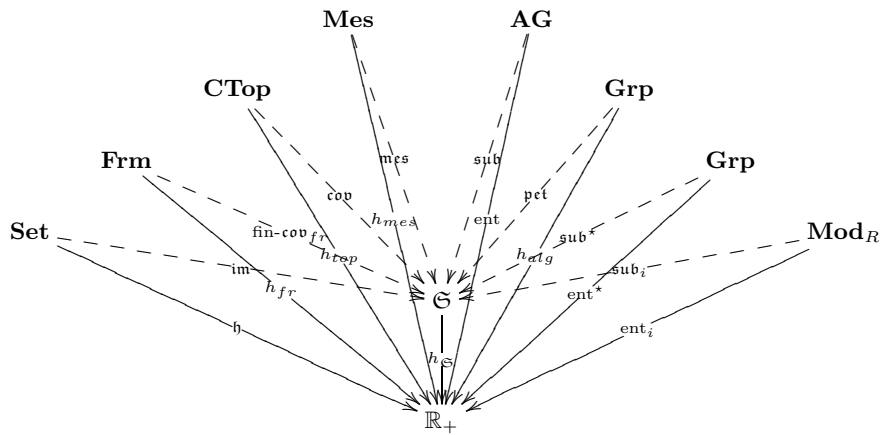
\begin{equation*}
\xymatrix@-0.9pc{
&&&\mathbf{Mes}\ar@{-->}[ddddr]|-{\mathfrak{mes}}\ar[ddddddr]|-{h_{mes}}& &\mathbf{AG}\ar[ddddddl]|-{\mathrm{ent}}\ar@{-->}[ddddl]|-{\mathfrak{sub}}&&\\
& &\mathbf{CTop}\ar@{-->}[dddrr]|-{\mathfrak{cov}}\ar[dddddrr]|-{h_{top}}&  & & & \mathbf{Grp}\ar@{-->}[dddll]|-{\mathfrak{pet}}\ar[dddddll]|-{h_{alg}} & &\\
& \Fr\ar@{-->}[ddrrr]|-{\fc_{fr}}\ar[ddddrrr]|-{h_{fr}} && & &	&  & \mathbf{Grp}\ar@{-->}[ddlll]|-{\mathfrak{sub}^\star}\ar[ddddlll]|-{\ent^\star} \\
\mathbf{Set}\ar@{-->}[drrrr]|-{\atr}\ar[dddrrrr]|-{\mathfrak h} && && & &	&  &\mathbf{Mod}_R\ar@{-->}[dllll]|-{\mathfrak{sub}_i}\ar[dddllll]|-{\ent_i} \\
& && & \mathfrak S \ar[dd]|-{h_\mathfrak S} && & \\
 \\
& && &{\R_+}	 &	& &
}		
\end{equation*}

\medskip
Moreover, in \S\ref{f-sec} we discuss the notion of Bernoulli shifts in an arbitrary (abstract or concrete) category $\mathfrak X$ with products or coproducts.  Moreover, for such categories $\mathfrak X$ we give the more general definitions of backward and forward generalized shift.

In particular, the backward generalized shift is a morphism $\sigma_\lambda : K^Y \to K^X$ in $\mathfrak X$, corresponding to a map $\lambda: X \to Y$ in the category $\mathbf{Set}$ of sets and maps and a fixed object $K$ of $\mathfrak X$. This defines a contravariant functor $\mathcal B_K: \mathbf{Set} \to \mathfrak X$, sending $X\in \mathbf{Set}$ to the product $K^X$ and $\lambda: X \to Y$ to $\sigma_\lambda$, which sends coproducts (in $\mathbf{Set}$) to products (in $\mathfrak X$). 
Up to natural equivalence, these are all the functors $\mathbf{Set} \to \mathfrak X$ with this property. 

Analogously, for a category $\mathfrak X$ with coproducts, the forward generalized shift is a morphism $\tau_\lambda : K^{(X)} \to K^{(Y)}$ in $\mathfrak X$, corresponding to a map $\lambda: X \to Y$ in $\mathbf{Set}$ and a fixed object $K$ of $\mathfrak X$. It defines a covariant functor $\mathcal F_K: \mathbf{Set} \to \mathfrak X$, sending $X\in \mathbf{Set}$ to the coproduct $K^{(X)}$ and $\lambda: X \to Y$ to $\tau_\lambda$, which sends coproducts (in $\mathbf{Set}$) to coproducts (in $\mathfrak X$). 
Up to natural equivalence, these are all the functors $\mathbf{Set} \to \mathfrak X$ with this property.

\subsection{Bridge Theorem}

The connections between pairs of entropies are analyzed in \S\ref{BT} from a categorical point of view. This is inspired by the following remarkable connection between the algebraic and the topological entropy, which will be called {\em Bridge Theorem} throughout this paper. 

For a locally compact abelian group $G$, denote by \newsym{Pontryagin dual group of a locally compact abelian group $G$}{$\widehat G$} its Pontryagin dual, that is, the group of all continuous homomorphisms $G\to\T$, where $\mathbb T=\mathbb R/\mathbb Z$ is the circle group, endowed with the compact-open topology; moreover, for an endomorphism $\phi:G\to G$, denote by $\widehat\phi:\widehat G\to \widehat G$ its dual, that is, $\phi(\chi)=\chi\circ\phi$ for every $\chi\in\widehat G$.

\begin{Theorem}\label{BT1}
If $G$ is an abelian group and $\phi:G\to G$ an endomorphism, then $h_{alg}(\phi)=h_{top}(\widehat\phi).$
Equivalently, if $K$ is a compact abelian group and $\psi:K\to K$ a continuous endomorphism, then $h_{top}(\psi) = h_{alg}(\widehat \psi).$
\end{Theorem}

This theorem was proved when $G$ is a torsion abelian group (i.e., $K$ is a totally disconnected compact abelian group) by Weiss \cite{W}; later Peters \cite{P} obtained a proof for $G$ countable and $\phi$ an automorphism (i.e., $K$ metrizable and $\psi$ a topological automorphism). The theorem in this general form was recently proved by the authors in \cite{DG-bridge}.

It is not known whether this result holds in general for locally compact abelian groups. Anyway, in \cite{DGB-BT} it was proved for locally compact abelian groups $G$ with totally disconnected Pontryagin dual (i.e., $G$ is compactly covered) as stated in Theorem~\ref{BT2}. Indeed, this hypothesis on the group $G$ permits one to compute more easily the algebraic entropy of the continuous endomorphism $\phi:G\to G$ and the topological entropy of its dual $\widehat\phi:\widehat G\to \widehat G$, avoiding the use of the Haar measure that appears in the definition (see \cite{DG-islam,DGB-BT}). Moreover, one can apply in this setting the so-called limit-free formulas, arising from an idea of Yuzvinski \cite{Y} exploited then in \cite{DG-lf,G,GBST,GBV}.

\begin{Theorem}\label{BT2}
Let $G$ be a locally compact abelian group such that $\widehat G$ is totally disconnected and $\phi:G\to G$ a continuous endomorphism. 
Then $h_{alg}(\phi)=h_{top}(\widehat\phi).$
\end{Theorem}

Furthermore, the following result for topological automorphisms of locally compact abelian groups was stated in \cite{Pet1}, but several gaps in the proof were pointed out in \cite{DSV}. Now it is a consequence of a much more general result from \cite{V_BT}.

\begin{Theorem}\label{BT3}
Let $G$ be a locally compact abelian group and $\phi:G\to G$ a topological automorphism. 
Then $h_{alg}(\phi)=h_{top}(\widehat\phi).$
\end{Theorem}

Inspired by Theorem~\ref{BT1}, we consider here a far reaching generalization (using the same term) that tends to relate two entropies $h_1:\mathfrak X_1 \to \R_+$ and $h_2:\mathfrak X_2 \to \R_+$, defined on two categories $\mathfrak X_1$ and $\mathfrak X_2$ connected by a functor 
\begin{equation}\label{Julu29}
\varepsilon: \mathfrak X_1 \to \mathfrak X_2.
\end{equation}
\begin{Definition}\label{BTdef}
Consider the functor \eqref{Julu29} and let $h_1:\mathfrak X_1 \to \R_+$ and $h_2:\mathfrak X_2 \to \R_+$ be entropies. The pair $(h_1, h_2)$ satisfies
the \emph{Bridge Theorem}\index{Bridge Theorem} with respect to $\varepsilon$ with constant $0<C\in\R_+$ (briefly, \newsym{Bridge Theorem with respect to the functor $\varepsilon$ with constant $C>0$}{$BT_{\varepsilon,C}$}) if, for every $\phi:X\to X$ in $\mathfrak X_1$,
\begin{equation*}
h_2(\varepsilon(\f)) = C h_1(\f).
\end{equation*}
If $C=1$ we write simply \newsym{Bridge Theorem with respect to the functor $\varepsilon$ with coefficient $1$}{$BT_\varepsilon$}.
\end{Definition}

One can summarize $BT_\varepsilon$ by (very roughly) saying that the following diagram commutes.

\begin{equation}\label{Buz}
\xymatrix@R=6pt@C=37pt
{\mathfrak{X}_1\ar[dd]_{\varepsilon}\ar[rrd]^{h_{1}}	&	&	\\
& &	\R_+		\\
\mathfrak{X}_2\ar[rru]_{h_{2}}						&	&
}
\end{equation}

In Definition~\ref{BTdef} we allow also $C=\infty$, with the natural convention 
$$\infty\cdot h_1(\f)  =\begin{cases}\infty & \text{if}\ h_1(\f) > 0,\\ 0 & \text{if}\ h_1(\f) = 0.\end{cases}$$
In these terms, if $C=\infty$, then $h_2(\varepsilon(\f))$ takes only two values, namely,
$$h_2(\varepsilon(\f)) =\begin{cases} \infty & \text{if}\  h_1(\f) > 0,\\ 0 & \text{if}\ h_1(\f) =0.\end{cases}$$ 

In this scheme, denoting by $$\,\ \widehat{}:\mathbf{LCA}\to\mathbf{LCA}$$ the Pontryagin duality functor, Theorem~\ref{BT1} can be read as: 
\begin{enumerate}[(a)]
\item for $\ \widehat{}:\mathbf{AG}\to\mathbf{CAG}$, the pair $(h_{alg},h_{top})$ satisfies $BT_{\;\widehat{}\;}$;
\item for $\ \widehat{}:\mathbf{CAG}\to\mathbf{AG}$, the pair $(h_{top},h_{alg})$ satisfies $BT_{\;\widehat{}\;}$.
\end{enumerate}

\smallskip
In \S\ref{BTsec}, using the functorial nature of the entropy seen in \eqref{schemeq}, we discuss various stronger levels of the Bridge Theorem, by passing through the category $\Se$ of normed semigroups and using $h_\Se$. We call these stronger versions Strong Bridge Theorems (see Definition~\ref{SBT} and \eqref{A}).
 
This more precise approach permits one to find a new and transparent proof of Weiss' Bridge Theorem (see Theorem~\ref{WBT}) as well as  of many other Bridge Theorems, that we state in these new terms and prove in \S\ref{BTat}. 

\medskip
Beyond their explicit beauty, the Bridge Theorems may offer a very clear practical advantage, by reducing the computation of an entropy to some more appropriate environment. The best example to this effect are the (Strong) Bridge Theorems making use of the set-theoretic entropies expounded in \S\ref{BTset}. Here the topological entropy of the backward generalized shifts in $\mathbf{CTop_2}$ and the algebraic entropy of the forward generalized shifts in $\mathbf{AG}$ are computed in terms of the set-theoretic entropies.

\subsection{The limits of the general scheme}

A natural side-effect of the wealth of nice properties of the functorial entropy $\h_F=h_\Se\circ F$, obtained from the semigroup entropy $h_\Se$ through functors $F:\mathfrak X\to \Se$, is the loss of some known entropies that we explain below.
 Indeed, such a functorial entropy admits the Invariance under inversion property $\h_F(\phi) = \h_F(\phi^{-1})$ (see Lemma~\ref{inversion}), while it is known that some specific entropies do not have this property.

\medskip
A first case is Bowen's topological entropy for uniformly continuous selfmaps of non-compact metric spaces, and its extension by Hood to uniform spaces. For a counterexample to $h_B(\phi^{-1})= h_B(\phi)$ for Bowen's entropy $h_B$, take the automorphism $\phi: \R \to \R$ defined by $\phi(x)= 2x$, which has $h_B(\phi)=\log 2$ and $h_B(\phi^{-1})=0$.

Fortunately, as far as the category $\mathbf{LCG}$ of locally compact groups and their continuous endomorphisms is concerned, one can include within our scheme Bowen's topological entropy, as well as its algebraic counterpart recalled above, using the larger category $\Se^*$ as done in \S\ref{NewSec2}.

On the other hand, in the general case of a non-compact metric space $X$, one would need a kind of poly-norm, i.e., a {\em family of norms} 
that corresponds to the family of all compact subspaces of $X$. Led by this idea, one could consider the category of semigroups provided with families of norms, and develop the subsequent theory in the line of that presented in this paper. 

This more general approach, tailored to include Bowen's topological entropy in full generality, unfortunately creates many technical difficulties (especially in defining the morphisms), a too high price to pay for adding a single entropy to the list.

\medskip
Another entropy which does not admit the Invariance under inversion property is the intrinsic algebraic entropy from \cite{DGSV} (see also \cite{SV2}), hence it cannot be obtained as a functorial entropy induced by a functor with target $\Se$. This entropy is a generalization of the algebraic entropy $\ent$ (as it coincides with $\ent$ on torsion abelian groups) but does not vanish on torsion-free abelian groups. 
A similar failure in satisfying the Invariance under inversion property can be observed in the pretty analogous case of the algebraic and the topological dimension entropy for locally linearly compact vector spaces from \cite{CGBalg,CGBtop} (see Remark~\ref{llc}).

\smallskip
Unlike the favorable outcome in the case of Bowen's topological entropy in the category $\mathbf{LCG}$, now one has a second more substantial obstacle. Namely, the definition the intrinsic entropy  is based on the concept of \emph{inert subgroup}, that was deeply investigated in several papers (see \cite{DDR,DR,DGSV1,DSZ,GSZ}). Similarly, the dimension entropies mentioned above use the fact that the open linearly compact linear subspaces are {\em inert} (in a similar, appropriate sense) with respect to all continuous linear maps (see \cite{DDR}).

In order to include the intrinsic and the dimension entropies in a scheme similar to that presented in this paper one needs to find an appropriate counterpart of the notion of inert subgroup (subspace). Unfortunately, the norm of a semigroup, used in the definition of the semigroup entropy to measure the growth of the partial trajectories of an endomorphism, 
seems inappropriate to ``encode" also this subtle property. Indeed, a notion of ``inert element" would require a different tool, namely, a kind of ``distance", that measures when two elements of the a semigroup are ``close". Therefore, a possible way to include the intrinsic and the dimension entropies mentioned above in a scheme similar to that presented in this paper would be to replace the target $\Se^*$ with an appropriate category of semilattices provided with a ``distance'' in the above sense (having some additional properties).

Again, as a consequence of this very specialized and additionally complicated approach tailored to cover only these few missing entropies, one would miss the relevant cases covered by our current approach (as the topological entropy and the algebraic entropy, among other entropies). This would  certainly push us apart from the main aim of the present paper to unify entropies.

\bigskip
In \cite{DGB_PC}, as a particular case of the notion of semigroup entropy $h_\Se$, we introduced the notion of semigroup entropy $h_{\Se}^0(x)$ of an element $x$ of a normed semigroup $S$. 

The idea of defining this notion was supported at least from three examples of entropy functions, that can be described in terms of $h_{\Se}^0$, namely, the entropy of finite length endomorphisms of Noetherian local rings \cite{MMS}, the dynamical degree of rational maps \cite{Silv}, and the growth rate of endomorphisms of finitely generated groups \cite{Bowen}.

\subsection{Final remarks} 

A preliminary version of these ideas is expounded in the survey \cite{DGB_PC}, with only few (concise) proofs and with a particular emphasis on the connection of the algebraic entropy to the growth of groups \cite[\S 4.3]{DGB_PC}.

In this paper,  we first extend our setting and we give all proofs and details. Moreover, we add some new entropies under the umbrella of $\h_F$ (in \S\ref{lintop-sec} and \S\ref{V-sec}), in particular  \cite[\S\S3.1--3.9]{DGB_PC} are covered (and largely extended) by \S\ref{known-sec} and \S\ref{NewSec2}. 

Furthermore, the major change and novelty in this paper is the completely renewed and extended \S\ref{BT} on the Bridge Theorem. 

\medskip
The second named author gave a talk on what is exposed in \S\ref{BT} of this paper about the Bridge Theorem in the special session ``Algebraic Entropy and Topological Entropy'' of the First Joint International Meeting RSME-SCM-SEMA-SIMAI-UMI held in Bilbao in 2014.

\smallskip
We started this project with our coauthor and friend Simone Virili, so we thank him also for many useful discussions and ideas; in particular, the idea to use semigroups in place of semilattices belongs to him. He is presenting his own version of this general approach in \cite{V_BT}. 

\smallskip
The essential and explanatory term Bridge Theorem was coined by our coauthor and friend Luigi Salce, who is one of the main contributor in the theory of the algebraic entropy for groups and modules since its rediscovery in the joint paper \cite{DGSZ} with Dikranjan, Goldsmith, and Zanardo.

\smallskip
It is fair to mention the substantial contribution of our friend and colleague Peter V\' amos towards a better and deeper understanding
of the algebraic entropy. At the Workshop on Commutative Rings and Their Modules held in Bressanone in 2007 he 
discussed with the first named author and Luigi Salce the subtle connection between entropy and length functions and multiplicity. 
These contacts continued further, involving also Simone Virili and led to \cite{SVV}, inspired also by \cite{DGSZ,SZ0,SZ}.

\section*{Acknowledgements}

It is a pleasure to thank George Janelidze for fruitful discussions with the first named author and in particular for his suggestion to introduce frame entropy. 

\smallskip
We thank also Lorenzo Busetti, who kindly allowed us to reproduce in our Introduction the diagram appearing in his Master's thesis \cite{LoBu}. 

\smallskip
We warmly thank Lydia Au\ss enhofer, Domenico Freni and Michal Misiuriewicz for pointing out some problems appearing in previous versions of the paper.

\medskip
Last but not least, we are indebted with the referees for their careful reading and very useful comments and suggestions.

\newpage

\section{The category $\Se$ of normed semigroups}\label{Se-sec}

\subsection{Definition and examples}\label{Se-sec1}

We denote by $\newsym{set of integers}{\Z}$ the set of integers, by $\newsym{set of natural numbers}{\N}$ the set of natural numbers and by $\newsym{set of positive integers}{\N_+}$ the set of positive integers.

We start introducing the notion of normed semigroup.

\begin{Definition}\label{Def1}  
A \emph{norm}\index{norm}\index{semigroup norm}  on a semigroup $(S,\cdot)$ is a map $v: S \to \R_{\geq 0}$. A
\emph{normed semigroup}\index{normed semigroup} is a semigroup provided with a norm.

If $S$ is a monoid, a \emph{monoid norm}\index{monoid norm} on $S$ is a semigroup norm $v$ such that $v(1)=0$; in such a case $S$ is called \emph{normed monoid}\index{normed monoid}.
\end{Definition}

Denote by \newsym{category of normed semigroups and semigroup homomorphisms}{$\Se^*$} the category of all normed semigroups and all semigroup homomorphisms. Nevertheless, the following stronger notion of morphisms in $\Se^*$ appears both natural and also quite useful as we shall see in what follows.

\begin{Definition}
A semigroup homomorphism $\f:(S,v)\to (S',v')$ between normed semigroups is \emph{contractive}\index{contractive semigroup homomorphism} if $$v'(\phi(x))\leq v(x)\ \text{for every}\ x\in S.$$
\end{Definition}

\medskip
We denote by \newsym{category of normed semigroups and contractive semigroup homomorphisms}{$\Se$} the category of normed semigroups, whose morphisms are all the contractive semigroup homomorphisms, i.e., $\Se$ is a non-full subcategory of $\Se^*$, having the same objects, but less morphisms. If $\phi:S\to S$ is an isomorphism in $\Se$, then $v(\phi(x))=v(x)$ for every $x\in S$.

Moreover, let \newsym{category of normed monoids and contractive monoid homomorphisms}{$\mathfrak M$} be the non-full subcategory of $\Se$ with objects all normed monoids and with morphisms all contractive monoid homomorphisms.
$$\xymatrix@-1pc{
\Se^*\ar@{-}[d] \\
\Se\ar@{-}[d]\\
\mathfrak M
}$$

Given a normed semigroup $(S,v)$, a \emph{normed subsemigroup}\index{normed subsemigroup} $T$ of $S$ is a subsemigroup of $S$ with norm the restriction of $v$ to $T$. Note that the inclusion $T\to S$ is a morphism in $\Se$. 

\begin{Convention}
In this paper,  most often we consider contractive semigroup endomorphism $\phi:S\to S$, so it is safe to think that a semigroup endomorphism $\phi$ is contractive, unless otherwise stated. 
\end{Convention}

We use the category $\Se^*$ when we define the semigroup entropy in \S\ref{hs-sec}, and then only in \S\ref{set-sec} 
(for the contravariant set-theoretic entropy), in \S\ref{BT} (to discuss the Strong Bridge Theorems) and in \S\ref{NewSec2} (for the algebraic and the topological entropy for locally compact groups).

\medskip
Now we introduce special forms of norm, in particular subadditive semigroup norms will be fundamental in the paper.

\begin{Definition}\label{subadditivedef}
Let $(S,v)$ be a normed semigroup. The norm $v$ is:
\begin{enumerate}[(a)]
\item \emph{subadditive}\index{subadditive norm} if $v(x \cdot y) \leq v(x) + v(y)$ for every $x,y\in S$, in such a case we call $(S,v)$ briefly \emph{subadditive semigroup}\index{subadditive semigroup} (or {\em subadditive monoid}\index{subadditive monoid}, in case it is a monoid);
\item \emph{bounded}\index{bounded norm} if there exists $C\in \N_+$ such that $v(x) \leq C$ for all $x\in S$;
\item \emph{arithmetic}\index{arithmetic norm} if for every $x\in S$ there exist a constant $C_x\in\N_+$
such that $v(x^n) \leq C_x\cdot  \log n$ for every $n\in\N$ with $n\geq 2$.
\end{enumerate}
\end{Definition}

Obviously, bounded norms are arithmetic.

\medskip
We denote by \newsym{category of subadditive semigroups and contractive semigroup homomorphisms}{$\Sesu$} the full subcategory of $\Se$ of subadditive semigroups and by \newsym{category of subaddive monoids and contractive monoid homomorphisms}{$\mathfrak M^\dag$} the full subcategory of $\mathfrak M$ of subadditive monoids.

$$\xymatrix@-1pc{
& \Se^*\ar@{-}[d] & \\
& \Se\ar@{-}[dl]\ar@{-}[dr] & \\
\Sesu & & \mathfrak M\\
& \mathfrak M^\dag \ar@{-}[ul]\ar@{-}[ur] &
}$$

\medskip
A sequence $\{a_n\}_{n\in\N}$ of non-negative real numbers is \emph{subadditive} if  $a_{n+m}\leq a_n  +a_m$ for every $n,m\in \N$. 
The following known fact is applied  in Theorem~\ref{limit} to prove the existence of the limit defining the semigroup entropy in $\Se^\dag$.

\begin{Lemma}[Fekete Lemma\index{Fekete Lemma} \cite{Fek}]\label{fekete}
For a subadditive sequence $\{a_n\}_{n\in\N}$ of non-negative real numbers, the sequence $\{\frac{a_n}{n}\}_{n\in\N}$ converges and $$\lim_{n\to\infty}\frac{a_n}{n}=\inf_{n\in\N}\frac{a_n}{n}.$$
\end{Lemma}

Next we give several examples of norms on the semigroup $(\N,+)$.

\begin{Example}\label{Fekete} 
Consider the monoid $S = (\N, +)$.
\begin{itemize}
\item[(a)] Subadditive norms $v$ on $S$ correspond to subadditive sequences $\{a_n\}_{n\in\N}$ in  $\R_{\geq0}$ via $v \mapsto \{v(n)\}_{n\in\N}$. Then $\lim_{n\to \infty} \frac{a_n}{n}= \inf_{n\in\N} \frac{a_n}{n}$ exists by Lemma~\ref{fekete}.
\item[(b)] For $0 < p \leq 1$, define $v_p: S \to \R_{\geq0}$ by $v_p(x) = x^p$ for $x\in S$. Then the norm $v_p$ is a subadditive monoid norm, but $v_p$ is not arithmetic.
\item[(c)] Define $v_l: S \to \R_{\geq0}$ by $v_l(x) =\log( x+1)$ for $x\in S$. Then $v_l$ is a subadditive monoid norm, as $v_l(0) = 0$ and, for $x,y \in S$,
$$v_l(x + y) = \log(x + y +1) \leq \log (xy+x+y+1) = \log (x+1) + \log (y+1) = v_l(x) + v_l(y).$$  
Moreover $v_l$ is arithmetic, since for every $x\in S$, $nx+1\leq n^{x+1}$ for every $n\in\N$ with $n\geq 2$. 
\item[(d)] For $a\in \N$, $a>1$, consider
$$v_a(m) =\begin{cases} 0 & \text{if}\ m\in a\N,\\ m &\text{if}\ m\in \N\setminus a\N.\end{cases}$$ 
Then the norm $v_a$ on $S$ is neither arithmetic nor subadditive.

For $b\in \N_+$ consider the semigroup endomorphism $$\varrho_b: x\mapsto bx$$ of $(S,v_a)$. Then $\varrho_b$ is in $\Se$ (i.e., $\varrho_b$ is contractive) if and only if either $b = 1$ or $a|b$.
\end{itemize}
\end{Example}

We can define the norm $v_l$ in item (c) also for $S=(\R_{\geq0},+)$, i.e., with $\R_{\geq0}$ in place of $\N$. Then $v_l$, defined by $v_l(x) =\log( x+1)$ for every $x\in S$, is a subadditive and arithmetic monoid norm. Moreover, in this case the range is completely covered by $v_l(S)$, i.e., $v_l(S)=\R_{\geq0}$.

\subsection{Normed preordered semigroups and normed semilattices}\label{preorder-sec} 

None of  the more specific forms of normed semigroup considered in this subsection is formally needed for the definition of the semigroup entropy. Nevertheless, they provide significant and natural examples, as well as useful tools in the proofs, to justify our attention. 

\begin{Definition}
A triple $(S,\cdot,\leq)$ is a \emph{preordered semigroup}\index{preordered semigroup} if $(S,\cdot)$ is a semigroup and $\leq$ is a preorder such that, for every $x,y,z \in S$,
$$x\leq y\quad \Rightarrow\quad x \cdot z \leq y \cdot z\ \text{and}\ z \cdot x \leq z \cdot y.$$
The \emph{positive cone}\index{positive cone} of $S$ is $$\newsym{positive cone}{P_+(S)}=\{a\in S:x\leq x \cdot a \ \text{and}\ x\leq a\cdot x\ \text{for every}\ x\in S\}.$$
\end{Definition}

If $S$ is a preordered monoid, $P_+(S)=\{a\in S:a\geq 1\}$.

\begin{Definition}
If $S$ is a preordered semigroup, a subset $T$ of $S$ is {\em cofinal} in $S$ if, for every $s\in S$, there exists $t\in T$ such that $s\leq t$. 
\end{Definition}

As usual, we call a map $\phi:(S_1,\leq)\to (S_2,\leq)$ between preordered sets \emph{monotone}\index{monotone map} provided $\phi(x_1)\leq \phi(x_2)$ for $x_1, x_2\in S_1$ with $x_1\leq x_2$.

\smallskip
 In the next example we see that for every commutative monoid there exists a natural preorder that makes it a preordered monoid and such that every monoid homomorphism is monotone.

\begin{Example}\label{leqd}
If $(S,\cdot)$ is a commutative monoid, it admits an intrinsic preorder \newsym{intrinsic preorder of a commutative monoid}{$\leq^d$} defined for every $x,y\in S$ by $x\leq^d y$ if and only if there exists $z\in S$ such that $x\cdot z=y$ (i.e., $x$ ``divides" $y$).  Then $(S,\cdot,\leq^d)$ is a preordered semigroup with $P_+(S) = S$.  If $(T,\cdot)$ is another commutative monoid, then every semigroup homomorphism $\phi: (S, \cdot) \to (T, \cdot)$ is monotone with respect to the respective preorders $\leq^d$ of $S$ and $T$.

So, the assignment $S \mapsto (S, \leq^d)$ gives a concrete functor (i.e., a functor that does not change the supporting sets and maps) from the category of commutative monoids to the category of preordered commutative monoids with morphisms the monotone monoid homomorphisms.
\end{Example} 

A \emph{semilattice}\index{semilattice} is a partially ordered set $(S,\leq)$ such that for every $x,y\in S$ there exists the least upper bound $x\vee y$ of $x$ and $y$.  Moreover, we assume that a semilattice $(S,\leq)$ admits a least element $0\in S$. Equivalently, a semilattice is a commutative monoid $(S,\vee)$ such that $x\vee x=x$ for every $x\in S$ (witnessed by the preorder $\leq^d$ which is a partial order in this case).

\begin{Example} 
\begin{itemize}
\item[(a)] Each lattice $(L, \vee, \wedge)$ with $0$ and $1$ gives rise to two semilattices, namely $(L, \vee)$ and $(L, \wedge)$.
\item[(b)] A filter $\mathcal F$ on a given set $X$ is a semilattice with respect to the intersection, with zero element the set $X$. 
\end{itemize}
\end{Example}

The following notion offers a convenient weaker form of semilattice. For a preordered set $(X,\leq)$ and $x,y\in X$, we write $$x\sim y\quad \Leftrightarrow\quad x\leq y \ \text{and}\ y\leq x.$$

\begin{Definition}\label{presemilattice}
A \emph{presemilattice}\index{presemilattice} is a preordered commutative monoid $(S,\vee,\leq)$ such that $x\vee x\sim x$ for every $x\in S$. 
\end{Definition}

Now we equip the above semigroups with a \emph{monotone norm}\index{monotone norm}, that is, a monotone map $v:(S,\cdot,\leq)\to (\R_{\geq0},+,\leq)$. 

\begin{Definition}\label{normedpreordered}
\begin{itemize}
\item[(a)] A \emph{normed preordered semigroup}\index{normed preordered semigroup} (respectively, {\em semilattice\index{normed semilattice}, presemilattice\index{normed presemilattice}}) is a preordered semigroup (respectively, semilattice, presemilattice) $(S,\cdot,\leq)$ endowed with a monotone norm.
\item[(b)] A preorder $\leq$ on a normed semigroup $(S,v)$ is \emph{compatible}\index{compatible preorder} with a semigroup endomorphism $\phi:(S,v)\to (S,v)$ if both $v$ and $\phi$ are monotone with respect to $\leq$.
\end{itemize}
\end{Definition}

Clearly, a (normed) semilattice is a (normed) presemilattice, while a (normed) presemilattice is a (normed) preordered semigroup.  Hereinafter we use the following notation:
\begin{enumerate}[-]
\item \newsym{category of normed preoredered semigroups and monotone se\-mi\-group homomorphisms}{$\Se_p$} is the subcategory of $\Se$ of normed preordered semigroups and monotone semigroup homomorphisms;
\item $\newsym{category of normed preordered monoids and monotone monoid homomorphisms}{\mathfrak M_p}=\Se_p\cap \mathfrak M$ is the subcategory of $\mathfrak M$ of normed preordered monoids and monotone monoid homomorphisms;
\item \newsym{category of normed presemilattices and monotone monoid homomorphisms}{$\PSL$} is the full subcategory of $\mathfrak M_p$ with objects all normed presemilattices;
\item \newsym{category of semilattices and se\-mi\-lattice homomorphisms}{$\SL$} is the full subcategory of $\mathfrak M$ with objects all normed semilattices.
\end{enumerate}
Note that a morphism in $\SL$ is necessarily monotone.

\smallskip
Moreover, let:
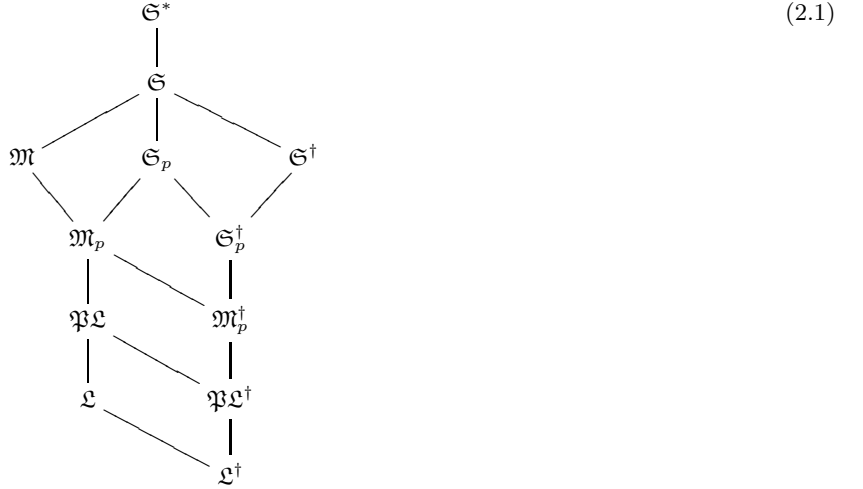
\begin{enumerate}[-]
\item $\newsym{category of preoredered subadditive semigroups and monotone se\-mi\-group homomorphisms}{\Se_p^\dag}= \Se_p\cap \Se^\dag$;
\item $\newsym{category of preordered subadditive monoids and monotone monoid homomorphisms}{\mathfrak M_p^\dag}= \mathfrak M_p\cap \Se^\dag $;
\item $\newsym{category of subadditive presemilattices and monotone monoid homomorphisms}{\PSL^\dag}=\PSL \cap \Se^\dag$;
\item $\newsym{category of subadditive semilattices and semilattice homomorphisms}{\SL^\dag}= \SL \cap \Se^\dag$.
\end{enumerate}
\begin{equation}\label{manycat}
\xymatrix@C=7pt@R=14pt{
& &\Se^* \ar@{-}[d] & & \\
& &\Se \ar@{-}[d] \ar@{-}[drr]\ar@{-}[dll]& & \\
\mathfrak M \ar@{-}[dr] & & \Se_p \ar@{-}[dl] \ar@{-}[dr] & & \Sesu \ar@{-}[dl]\\
& \mathfrak M_p\ar@{-}[d]\ar@{-}[drr] & & \Se_p^\dag\ar@{-}[d] &\\
& \PSL \ar@{-}[d]\ar@{-}[drr]& & \mathfrak M_p^\dag \ar@{-}[d]&\\
& \SL\ar@{-}[drr] & & \PSL^\dag\ar@{-}[d]\\
& & & \SL^\dag
}\end{equation}

Obviously, the norm of a normed presemilattice $(S,\vee)$ is arithmetic.

\medskip
Finally, we propose another notion of monotonicity for a semigroup norm which does not require the semigroup to be explicitly endowed with a preorder.
 
\begin{Definition}\label{d-def}
Let $(S,v)$ be a normed semigroup. The norm $v$ is \emph{d-monotone}\index{d-monotone norm} if
$$\max\{v(x), v(y)\}\leq v(x \cdot y)\ \text{for every}\ x,y \in S.$$
\end{Definition}

When $(S,v)$ is a commutative normed monoid, $v$ is d-monotone precisely when $v$ is monotone with respect to $\leq^d$.

The inequality in Definition~\ref{d-def} may become a too stringent condition when $S$ is close to being a group; indeed, if $S$ is a group, then it implies that $v(S) = \{v(1)\}$, that is, $v$ is constantly zero.  Nevertheless, this will have no impact on our approach to entropy since the specific semigroups that appear in all cases considered hereinafter are indeed quite far from being groups. 

\smallskip
The following connection between monotonicity and $s$-monotonicity is clear. 
 
\begin{Lemma}\label{cone}
Let $S$ be a preordered semigroup. If $S=P_+(S)$ (in particular, if $S$ is a lattice), then every monotone norm of $S$ is also d-monotone.
\end{Lemma}

\section{The semigroup entropy}\label{hs-sec}

\subsection{Definition}

In this section we introduce the concept, fundamental in this paper, of semigroup entropy. 

\medskip
For $(S,v)$ a normed semigroup, $\phi:S\to S$ an endomorphism (not necessarily contractive) and $n\in\N_+$, consider the \emph{$n$-th $\phi$-trajectory of $x\in S$}\index{trajectory}
$$\newsym{$n$-th $\phi$-trajectory of $x\in S$}{T_n(\f,x)} = x \cdot\f(x)\cdot\ldots \cdot\f^{n-1}(x)$$ and let $c_n(\f,x) = v(T_n(\f,x)).$

We give the following definition of semigroup entropy in the general case of the category $\Se^*$, but we will consider it mainly in the category $\Se$, that is, for contractive semigroup endomorphisms.

\begin{Definition}\label{SEofEndos}    
Let $S$ be a normed semigroup and $\f:S\to S$ an endomorphism in $\Se^*$. The \emph{semigroup entropy of $\phi$ with respect to $x\in S$} is
$$\newsym{semigroup entropy of $\phi$ with respect to $x\in S$}{h_{\Se^*}(\f,x)}= \limsup_{n\to\infty}\frac{c_n(\f,x)}{n}.$$
The \emph{semigroup entropy}\index{semigroup entropy} of $\f$ is $$\newsym{semigroup entropy of $\phi$}{h_{\Se^*}(\f)}=\sup_{x\in S}h_{\Se^*}(\f,x).$$
\end{Definition}

\begin{Convention}
When we are in the category $\Se$, we denote the semigroup entropy by $h_\Se$. In particular, when we write \newsym{semigroup entropy of $\phi$ with respect to $x\in S$ in $\Se$}{$h_{\Se}(\f,x)$} or \newsym{semigroup entropy of $\phi$ in $\Se$}{$h_{\Se}(\f)$}
we intend that $\f$ is contractive, even if we do not say that explicitly. 
\end{Convention}

In the next example we consider the semigroup entropy of the identity map $id_S$ of a normed semigroup $S$; since $id_S$ is always contractive, we can write $h_\Se(id_S)$. In item (a) we see that $h_\Se(id_S)=0$ for normed semigroups with arithmetic norm, but also in other cases.  This will be generalized to other contracting endomorphisms (namely, quasi-periodic ones) in Theorem~\ref{quasi-per0} and Theorem~\ref{quasi-per}. In item (b) we see that $h_\Se(id_S)$ can be infinite; this case can occur  even for subadditive semigroups as we will see in Example~\ref{subidinf}.
 
\begin{Example}\label{idex}
Let $(S,v)$ be a normed semigroup.
\begin{enumerate}[(a)]
\item It is easy to see that if $v$ is arithmetic, then $h_\Se(id_S)=0$. 

Nevertheless, $h_\Se(id_S)=0$ may occur also when $v$ is not arithmetic. To this end 
consider $S=(\N,+)$ with the norm $v_p$ from Example~\ref{Fekete}(b),  with $0<p<1$. Then $h_\Se(id_S,x)=0$ for every $x\in S$, since $\lim_{n\to\infty}\frac{(nx)^p}{n}=0$, and so $h_\Se(id_S)=0$.
\item Consider $S=(\N,+)$ with the norm $v_a$, for some $a\in\N$, $a>1$, from Example~\ref{Fekete}(d), which is not arithmetic and not subadditive. Then $h_\Se(id_S,x)=x$ for every $x\in\N$, and so $h_\Se(id_S)=\infty$.
\end{enumerate}
\end{Example}

The next is another example of computation of the semigroup entropy.

\begin{Example}\label{vlex}
\begin{enumerate}[(a)]
\item Consider $(\N,+)$ with the norm $v_l(x)=\log(x+1)$ for every $x\in\N$ from Example~\ref{Fekete}(c), which is subadditive and arithmetic.
Let $a\in\N_+$ and $\varrho_a:\N\to \N$ defined by $\varrho_a(x)=ax$ for every $x\in \N$; note that $\varrho_a$ is not contractive if $a > 1$.
Moreover, $h_{\Se^*}(\varrho_a,x)=\log a$ for every $x\in\N_+$, so $h_{\Se^*}(\varrho_a)=\log a$.
\item Consider $(\R_{\geq0},+)$ with the norm $v_l(x)=\log(x+1)$ for every $x\in\R_{\geq0}$. Let $a\in\R_{\geq0}$, $a>0$ and $\varrho_a:\R_{\geq0}\to \R_{\geq0}$ defined by $\varrho_a(x)=ax$ for every $x\in\R_{\geq0}$. If $a\leq 1$, then $\varrho_a$ is contractive and $h_\Se(\varrho_a)=0$. If $a>1$, then $\varrho_a$ is not contractive and $h_{\Se^*}(\varrho_a)=\log a$ as above.
\end{enumerate}
\end{Example}

An open problem in the context of the topological and the algebraic entropy is whether  the infimum of the positive values of entropy is still positive. 
This is equivalent to Lehmer's problem from number theory on the values of the Mahler measure (see \cite{DG-islam,DG} for more details).
Following this idea, for a fixed $S\in \Se$ we consider the set $$E_{\Se^*}(S)=\{h_{\Se^*}(\phi): (S,\phi)\in \mathbf{Flow}_{\Se^*}\}\subseteq\R_+$$
of all possible values of the semigroup entropy on endomorphisms of $S$,
and we let $$\ell_{\Se^*}(S)=\inf (E_{\Se^*}(S)\setminus \{0\}).$$
Inspired by the counterpart of Lehmer's problem for the topological and the algebraic entropy, one can ask how well the set $E_{\Se^*}(S)\setminus \{0\}$ approximates $0$, i.e., 
whether $\ell_{\Se^*}(S)=0$.
In contrast with the highly difficult case of the topological or the algebraic entropy, from Example~\ref{vlex}(b) one gets $$\ell_{\Se^*}(\R_{\geq0},v_l)=0$$ by taking $a>1$ arbitrarily close to $1$.
This is even more striking since a single semigroup $S$ allows to get $\ell_{\Se^*}(S)=0$, whereas in the framework of the topological or the algebraic entropy it is not known whether $0$ can be attained even by taking a second $\inf$ on (i.e., varying) the supporting ``space" $S$.

\subsection{Entropy in $\Se$}\label{entropy:in:Se}

From now on, in this section we consider entropy in $\Se$ (and not in $\Se^*$), so when we write homomorphism/endomorphism we mean in $\Se$ (i.e., contractive).

\medskip
We list the main properties of the semigroup entropy, starting from its monotonicity under taking factor flows.

\begin{Lemma}[Monotonicity for factors\index{Monotonicity for factors}]\label{mono/fac} 
Let $S$, $T$ be normed semigroups and $\f: S \to S$, $\psi:T\to T$ endomorphisms. If $\alpha:S\to T$ is a surjective homomorphism such that $\alpha \circ \phi = \psi\circ \alpha$, i.e., the following diagram commutes,
\begin{equation}\label{22} 
\xymatrix{S\ar[r]^\phi\ar[d]_\alpha & S\ar[d]^\alpha \\
T \ar[r]_{\psi}&T}
\end{equation}
then $$h_{\Se}(\psi) \leq h_{\Se}(\phi).$$
\end{Lemma}
\begin{proof} 
Fix $y\in T$ and find $x \in S$ with $y= \alpha(x)$. Then $c_n( \psi,y) \leq c_n(\phi, x)$ for every $n\in\N_+$.
Dividing by $n$ and taking the $\limsup$ gives $h_{\Se}(\psi,y) \leq h_{\Se}(\phi,x)$. So $h_{\Se}(\psi,y)\leq h_{\Se}(\phi)$. When $y$ runs over $T$, we conclude that $h_{\Se}(\psi) \leq h_{\Se}(\phi)$. 
\end{proof} 

Applying twice the above lemma, we obtain the following fundamental property of the semigroup entropy.
 
\begin{Corollary}[Invariance under conjugation\index{Invariance under conjugation}] \label{iuc}
Let $S$ be a normed semigroup and $\f: S \to S$ an endomorphism. If $\alpha:T\to S$ is an isomorphism, then $$h_{\Se}(\f)=h_{\Se}(\alpha\circ\f\circ\alpha^{-1}).$$
\end{Corollary}

The next lemma shows that also the monotonicity under taking subsemigroups is available. 

\begin{Lemma}[Monotonicity for subflows\index{Monotonicity for subflows}] \label{mons}
Let $(S,v)$ be a normed semigroup and $\phi:S\to S$ an endomorphism. If $T$ is a $\phi$-invariant normed subsemigroup of $(S,v)$, then $$h_{\Se}(\phi)\geq h_{\Se}(\phi\restriction_{T}).$$
Equality holds provided that $S$ is preordered, $\phi$ and $v$ are monotone and $T$ is cofinal in $S$. 
\end{Lemma}
\begin{proof}
The first part is just a consequence of the definitions. Suppose that $S$ is preordered, $\phi$ and $v$ are monotone and $T$ is cofinal in $S$. Given $s\in S$, choose $t\in T$ such that $t\geq s$ and let us show that, for all $n\in\N_+$, $T_n(\phi,s)\leq T_n(\phi,t)$. Indeed, for $n=1$ this comes from the choice of $t$. If $n>1$ and we already proved our result for $n-1$, then, using the monotonicity of $\phi$,
$$T_n(\phi,s)=T_{n-1}(\phi,s)\cdot\phi^{n-1}(s)\leq T_{n-1}(\phi,t)\cdot\phi^{n-1}(s)\leq T_{n-1}(\phi,t)\cdot\phi^{n-1}(t)=T_n(\phi,t).$$
We can conclude applying the definition of entropy and the monotonicity of $v$. 
\end{proof}

The above monotonicity for subsemigroups applies to the subsemigroups $S_i$ in the next result.

\begin{Proposition}[Continuity for direct limits\index{Continuity for direct/inverse limits}] \label{cont}
Let $(S,v)$ be a normed semigroup and $\phi:S\to S$ an endomorphism. If $\{S_i: i\in I\}$ is a directed family of 
$\phi$-invariant normed subsemigroups of $(S,v)$ with $ S =\varinjlim_{i\in I} S_i$ and $\phi=\varinjlim_{i\in I}\phi\restriction_{S_i}$, then $$h_{\Se}(\phi)=\sup_{i\in I} h_{\Se}(\phi\restriction_{S_i}).$$
\end{Proposition}
\begin{proof}
By Lemma~\ref{mons} we have that $h_\Se(\phi)\geq h_\Se(\phi\restriction_{S_i})$ for every $i\in I$, so 
$$h_\Se(\phi)\geq\sup_{i\in I} h_{\Se}(\phi\restriction_{S_i}).$$ 
To verify the converse inequality, let $x\in S$. Since $S=\varinjlim_{i\in I} S_i$, there exists $i\in I$ such that $x\in S_i$. Then $$h_\Se(\phi,x)=h_\Se(\phi\restriction_{S_i},x)\leq h_\Se(\phi\restriction_{S_i}).$$ Hence, we can conclude that $h_\Se(\phi)\leq \sup_{i\in I}h_\Se(\phi\restriction_{S_i})$.
\end{proof}

The next lemma fully exploits our blanket hypothesis that $\phi$ is an automorphism in $\Se$ (see \S\ref{NewSec2}). 

\begin{Lemma}[Invariance under inversion\index{Invariance under inversion}]\label{inversion}
Let $S$ be a commutative normed semigroup and $\phi:S\to S$ an automorphism. Then $$h_{\Se}(\f^{-1})=h_{\Se}(\f).$$
\end{Lemma}
\begin{proof} 
It suffices to see that $h_{\Se}(\f^{-1},x)=h_{\Se}(\f,x)$ for each $x\in S$.   In order to compute $h_{\Se}(\f^{-1},x)$ note that, for every $n\in\N_+$, 
\begin{align*} c_n(\f,x) &=  v(x\cdot \f(x)\cdot \ldots \cdot \f^{n-1}(x))\\
&=v(\f^{n-1}(x\cdot \f^{-1}(x)\cdot \ldots \cdot \f^{-n+1}(x))\\
&= v(x\cdot  \f^{-1}(x)\cdot\ldots \cdot \f^{-n+1}(x))\\
&=c_n(\f^{-1},x).\end{align*}
Now the definition of $h_\Se(\f^{-1},x)$ and $h_\Se(\f,x)$ applies. 
\end{proof}  

Notice that in the above lemma we have to impose the hypothesis that our semigroup is commutative. In \S\ref{alternative} we give an example of a flow of normed semigroups $(S,\phi)$ whose semigroup entropy does not coincide with the entropy of its inverse flow $(S,\phi^{-1})$. We refer to that section for a more complete description of the entropy of the inverse flow. Similar considerations hold for the second part of the following lemma.
   
\begin{Proposition}[Logarithmic Law\index{Logarithmic Law}] \label{ll}
Let $(S,v)$ be a commutative d-monotone normed semigroup and $\f:S\to S$ an endomorphism.  Then, for every $k\in \N_+$,
$$h_{\Se}(\f^{k})=k\cdot h_{\Se}(\f).$$ 
Furthermore, if $\phi:S\to S$ is an automorphism, then for all $k \in \Z\setminus\{0\}$, $$h_{\Se}(\phi^k) = |k|\cdot h_{\Se}(\phi).$$
\end{Proposition}
\begin{proof} 
Fix $k \in \N_+$.
Let $x\in S$ and let $y= x\cdot\f(x)\cdot\ldots\cdot\f^{k-1}(x)$. For every $n\in\N_+$ we have that $$c_n(\phi^k,y)=v (y\cdot \f^k(y)\cdot\ldots \cdot\f^{(n-1)k}(y)) =c_{nk}(\f,x).$$
Then
$$h_{\Se}(\f^k)\geq h_{\Se}(\f^k, y)=\limsup_{n\to\infty} \frac{c_{n}(\f^k,y)}{n}= k \cdot \limsup_{n\to \infty} \frac{c_{nk}(\f,x)}{nk}=k\cdot h_\Se(\f,x).$$
The last equality in the above formula holds for the following reason. For every $j\in\N_+$ there exists $n\in\N$ such that $nk\leq j<(n+1)k$. 
Since $$T_j(\phi,x)=T_{nk}(\phi,x)\cdot\phi^{nk}\cdot\ldots\cdot\phi^{j-1}(x)\ \text{and}\ T_{(n+1)k}(\phi,x)=T_j(\phi,x)\cdot\phi^j(x)\cdot\ldots\cdot\phi^{(n+1)k-1}(x),$$ it follows that 
$$c_{nk}(\phi,x)\leq c_j(\phi,x)\leq c_{(n+1)k}(\phi,x),$$ and so 
$$\frac{nk}{j}\cdot \frac{c_{nk}(\phi,x)}{nk}\leq \frac{c_j(\phi,x)}{j}\leq \frac{(n+1)k}{j}\cdot\frac{c_{(n+1)k}(\phi,x)}{(n+1)k}.$$
Since $\lim_{n\to\infty}nk/j=\lim_{n\to\infty}(n+1)k/j=1$, we have 
$$\limsup_{n\to\infty}\frac{c_{nk}(\phi,x)}{nk}=\limsup_{n\to\infty}\frac{c_{(n+1)k}(\phi,x)}{(n+1)k}=\limsup_{j\to\infty}\frac{c_j(\phi,x)}{j}=h_\Se(\phi,x).$$
We conclude that $h_\Se(\f^k)\geq k\cdot h_\Se(\f,x)$ for all $x\in S$, and consequently, $h_\Se(\f^k)\geq k\cdot h_\Se(\f)$. 
  
\smallskip
Suppose $v$ to be d-monotone. Then, for every $n\in\N_+$ and $x\in S$, we have that
$$c_{nk}(\phi,x)=v(T_{nk}(\phi,x))\geq v(x\cdot\f^k(x)\cdot\ldots\cdot(\f^k)^{n-1}(x))= v(T_n(\phi^k,x))=c_n(\phi^k,x).$$
Therefore, 
\begin{equation*}\begin{split}
h_{\Se}(\f,x)\geq \limsup_{n\to\infty}\frac{c_{nk}(\phi,x)}{nk}\geq \limsup_{n\to\infty} \frac{c_n(\phi^k,x)}{n\cdot k}= \frac{1}{k}\cdot\limsup_{n\to\infty}\frac{c_n(\phi^k,x)}{n}=\frac{h_{\Se}(\f^k,x)}{k}.
\end{split}\end{equation*}
Hence, $k\cdot h_{\Se}(\f)\geq h_{\Se}(\f^k,x)$ for every $x\in S$, and so $k\cdot h_{\Se}(\f)\geq h_{\Se}(\f^k)$.  We conclude that $h_{\Se}(\f^{k})= k\cdot h_{\Se}(\f)$.

\smallskip
If $\phi$ is an automorphism and $k\in\Z\setminus\{0\}$, apply the previous part of the proposition and Lemma~\ref{inversion}.
\end{proof}

We call Logarithmic Law the property $h_\Se(\phi^k)=k\cdot h_\Se(\phi)$ for every $k\in\N_+$. Of course, if the Invariance under inversion is also available, then one has $h_\Se(\phi^k)=|k| \cdot h_\Se(\phi)$ for every $k\in\Z$,  with $k\neq0$, when $\phi$ is an automorphism. In case $k=0$, that is, $\phi^k=id_S$, the equality holds only if $h_\Se(id_S)=0$ (e.g., when $(S,v)$ is arithmetic). We shall adopt the terminology Logarithmic Law also with respect to other entropy functions.

\medskip
A flow $(S,\f)$ of $\Se$ is \emph{quasi-periodic}\index{quasi-periodic flow} if there exists a pair of naturals $m<k$ such that $\f^k= \f^m$. 

\begin{Proposition}[Vanishing on quasi-periodic flows\index{Vanishing on quasi-periodic flows}]\label{quasi-per0}
If $(S,\f)$ is a quasi-periodic flow of $\Se$ such that $S$ is commutative and d-monotone, then either $h_\Se(\f)=0$ or $h_\Se(\f)=\infty$.
\end{Proposition}
\begin{proof}
Assume that $\f^k= \f^m$ for a pair of naturals $m<k$.  Then, by Proposition~\ref{ll}, $$k\cdot h_{\Se}(\f)= h_{\Se}(\f^{k})= h_{\Se}(\f^m) = m\cdot h_{\Se}(\f).$$ 
Since $m<k$, the equality $k\cdot h_{\Se}(\f) = m \cdot h_{\Se}(\f)$ implies that either $h_\Se(\f)=0$ or  $h_\Se(\f)=\infty$. 
\end{proof}

Now we consider products in $\Se$. Let $\{(S_i,v_i):i\in I\}$ be a family of normed semigroups and let $S=\prod_{i \in I}S_i$ be their direct product in the category of semigroups. In case $I$ is finite, $S$ becomes a normed semigroup with the $\max$-norm $v_\Pi$, i.e., for $x=(x_i)_{i\in I}\in S$,
\begin{equation}\label{max-norm}
v_\Pi(x)=\sup\{v_i(x_i):i\in I\};
\end{equation}
so $(S,v_\Pi)$ is the product of the family $\{S_i:i\in I\}$ in the category $\Se$.

\begin{Theorem}[weak Addition Theorem - products\index{weak Addition Theorem}]\label{WAT}
Let $(S_i,v_i)$ be a normed semigroup and $\f_i:S_i\to S_i$ an endomorphism for $i=1,2$. Then the endomorphism $ \f=\f_1 \times \f_2$ of $ (S,v)=(S _1 \times S_2,v_\Pi)$ has $$h_\Se(\f)= \max\{ h_\Se(\f_1),h_\Se(\f_2)\}.$$
\end{Theorem}
\begin{proof} 
For $x\in S$, let $x_1\in S_1$ and $x_2\in S_2$ be such that $x=(x_1,x_2)$. Then 
\begin{align*}
h_\Se(\phi,x)&=\limsup_{n\to \infty}\frac{v(T_n(\phi,x))}{n}\\
&= \limsup_{n\to \infty}\frac{\max\{v_1(T_n(\phi_1,x_1)),v_2(T_n(\phi_2,x_2))\}}{n}\\
&=\max\{h_\Se(\f_1,x_1),h_\Se(\f_2,x_2)\}.
\end{align*}
Using the fact that for families of positive real numbers  $\{a_i\}_{i\in I}$ and $\{b_j\}_{j\in J}$
\begin{equation}\label{easy:claim}
\sup_{(i,j)\in I\times J}\max\{a_i,b_j\}=\max\{\sup_{i\in I}a_i,\sup_{j\in J}b_j\},
\end{equation}
we can conclude that $h_\Se(\phi)=\max\{ h_\Se(\f_1),h_\Se(\f_2)\}$.
\end{proof}

If $I$ is infinite, $S=\prod_{i\in I}S_i$ need not carry a semigroup norm $v$ such that every projection $p_i: (S,v) \to (S_i,v_i)$ is a morphism in $\Se$. This is why the product of the family $\{(S_i,v_i):i\in I\}$ in $\Se$ is actually the subset 
$$S_{\mathrm{bnd}}=\{(x_i)_{i\in I}\in S: \sup_{i\in I}v_i(x_i)\in\R\}$$
of $S$ with the norm $v_\Pi$ defined by \eqref{max-norm} for $x=(x_i)_{i\in I}\in S_{\mathrm{bnd}}.$

We will again consider the product of normed semigroups in \S\ref{Msec} below, in particular a comparison of Theorem~\ref{WAT} 
and Theorem~\ref{wAT}.

\subsection{Entropy in $\Sesu$}\label{semi-add}

In this section we consider the semigroup entropy for endomorphisms $\phi:S\to S$ of subadditive semigroups. 
Obviously, a subsemigroup of a subadditive semigroup $S$ is subadditive, too. 

\medskip
As we see in the next theorem, the superior limit in the definition of the semigroup entropy is actually a limit in this setting. Recall that here endomorphism means contractive semigroup endomorphism. 

\begin{Theorem}\label{limit}
Let $(S,v)$ be a subadditive semigroup and $\f:S\to S$ an endomorphism. Then for every $x \in S$ the limit 
\begin{equation}\label{hs-eq}
h_{\Se}(\f,x)= \lim_{n\to\infty}\frac{c_n(\f,x)}{n}
\end{equation}
 exists and satisfies $h_{\Se}(\f,x)\leq v(x)$.  
\end{Theorem}
\begin{proof} 
The sequence $\{c_n(\f,x)\}_{n\in\N_+}$ is subadditive. Indeed, for every $n,m\in\N_+$,
\begin{align*}  
c_{n+m}(\f,x)&= v(x\cdot\f(x)\cdot\ldots\cdot\f^{n-1}(x)\cdot\f^{n}(x)\cdot\ldots\cdot\f^{n+m-1}(x))\\  
&=v((x\cdot\f(x)\cdot\ldots\cdot\f^{n-1}(x))\cdot\f^{n}(x\cdot\ldots\cdot\f^{m-1}(x)) \\ 
&\leq c_n(\phi,x)+v(\f^{n}(T_m(\phi,x)) \\  
&\leq c_n(\f,x)+v(T_m(\phi,x))=c_n(\f,x)+c_m(\f,x),
\end{align*} 
where the first inequality follows from the subadditivity of $v$.
By Fekete Lemma~\ref{fekete},  the limit $\lim_{n\to\infty} \frac{c_n(\f,x)}{n}$ exists and coincides with $\inf_{n\in\N_+} \frac{c_n(\f,x)}{n}$.
Finally, $h_{\Se}(\f,x)\leq v(x)$ follows from $c_n(\phi,x) \leq n \cdot v(x)$ for every $n\in\N_+$.
\end{proof}

\begin{Remark}
In the above notation, we have seen at the end of the proof of Theorem~\ref{limit} that, for every $n\in\N_+$, $$c_n(\phi,x) \leq n\cdot v(x).$$ Hence, the growth of the function $n \mapsto c_n(\f,x)$ is at most linear. 
\end{Remark}

By Theorem~\ref{limit} we have that $h_{\Se}(\f,x)$ is always finite if $(S,v)$ is a subadditive semigroup. 
On the other hand, $h_{\Se}(\f,x)=\infty$ may occur, if $(S,v)$ is not subadditive as the following example shows.

\begin{Example}
Consider $S = (\N,+)$ with the norm $v_2$ as in Example~\ref{Fekete}(d), which is not subadditive. The endomorphism $\varrho_2:(S,v_2)\to (S,v_2)$ defined by $x\to 2x$ for every $x\in S$ is  contractive (see Example~\ref{Fekete}(d)). For every $n\in\N_+$, 
$$T_n(\varrho_2,1)=2^n-1,$$
and so $$c_n(\varrho_2,1)=v_2(T_n(\varrho_2,1))=2^n-1.$$ Hence, applying the definition, we have that $$h_\Se(\varrho_2,1)=\infty.$$
\end{Example}

In Example~\ref{idex}(a) we have seen that $h_\Se(id_S)=0$ when the norm of the normed semigroup $S$ is arithmetic. We show that $h_\Se(id_S)$ can be infinite even if the norm is subadditive.

\begin{Example}\label{subidinf}
Consider $S=(\N,+)$ with the norm $v(x)=x$, which is subadditive; anyway, $h_\Se(id_S,x)=x$ for every $x\in\N$, and so $h_\Se(id_S)=\infty$.
\end{Example}

Now we extend item (a) of Example~\ref{idex} as well as Theorem~\ref{quasi-per0}. 
A flow $(S,\f)$ of $\Se$ is {\em locally quasi-periodic}\index{locally quasi-periodic flow}  if for every $x\in S$ there exists a pair of naturals $m<k$ such that $\f^k(x) = \f^m(x)$. 

\begin{Theorem}\label{quasi-per}
If $(S,\f)$ is a locally quasi-periodic flow of $\Sesu$ such that the norm of $S$ is arithmetic, then $h_\Se(\f)=0$. 
\end{Theorem}
\begin{proof}
We have to prove that $h_\Se(\f,x)=0$ for an arbitrarily chosen $x\in S$. Pick a pair  of naturals $m<k$ such that $\f^k(x) = \f^m(x)$
and let $d= k-m>0$. Then, for all $i\in \N$, $$\f^{m+id}(x)=\f^m(x).$$
Let $$w=\f^m(T_d(\f,x)).$$
Since the norm is arithmetic, there exist $C_w \in\N_+$ such that $v(w^n)\leq C_w\log n$ for every $n\geq 2$.

Pick an arbitrary natural $n> m+2d$ and find $\bar\imath\in\N$ such that
$$m+\bar\imath d < n \leq m+(\bar\imath+1)d;$$
then $\bar\imath\geq 2$.
Put $l=n-(m+\bar\imath d)$, so that $0< l\leq d$.
Now let $$u =T_l(\f,x);$$
then $\f^{m+\bar\imath d}(u)=\f^{m+\bar\imath d}(x)\cdot\ldots\cdot \f^{n-1}(x)$. 
Therefore, since $n=m+\bar\imath d+l$,
$$T_n(\f,x) =T_m(\phi,x) \cdot w^{\bar\imath}\cdot  \f^{m+\bar\imath d}(u).$$
Thus, by the subaddititvity of the norm, $$c_n(\phi,x)=v(T_n(\f,x)) \leq v(T_m(\f,x)) + C_w\log\bar\imath + v(T_l(\f,x)).$$
Since $v(T_m(\f,x))$ and $v(T_l(\f,x))$ are bounded, we deduce that $$h_\Se(\phi,x)=\lim_{n\to \infty} \frac{c_n(\phi,x))}{n} =0,$$
 as required.
\end{proof}

\subsection{Entropy in $\mathfrak M$}\label{Msec}

We collect below some additional properties of the semigroup entropy in the category $\mathfrak M$ of normed monoids, where also coproducts are available. 

If $(S_i,v_i)$ is a normed monoid for every $i\in I$, the direct sum 
$$S_\oplus= \bigoplus_{i\in I} S_i =\left\{(x_i)\in S=\prod_{i\in I}S_i: |\{i\in I: x_i \ne 1\}|<\infty\right\}$$ becomes a normed monoid with the norm $$v_\oplus(x) = \sum_{i\in I} v_i(x_i)\ \text{for any}\ x = (x_i)_{i\in I} \in S_\oplus.$$ 
This definition makes sense since each $v_i$ is a monoid norm, so 
$v_i(1) = 0$. Hence, $(S_\oplus,v_\oplus)$ is a coproduct of the family $\{(S_i,v_i):i\in I\}$ in $\mathfrak  M$. 

Note that $$S_\oplus \subseteq S_{\mathrm{bnd}},$$ so one can consider on $S_\oplus$ both the norm $v_\oplus$ and the norm induced by $v_\Pi$; by the definitions of the two norms, we have that $$v_\Pi(x)\leq v_\oplus(x)$$ for every $x\in S_\oplus$.

\begin{Lemma}
The category  $\mathfrak M^\dag$ is stable under taking submonoids and direct sums.  
\end{Lemma}
\begin{proof}
The first assertion follows from the fact we mentioned above that $\mathfrak S^\dag$ is stable under taking subsemigroups. 
If $(S_i,v_i)$ is a normed monoid for every $i\in I$ and $S_\oplus= \bigoplus_{i\in I} S_i $, then we have to check that $v_\oplus$ is subadditive whenever each $(S_i,v_i)$ is subadditive. This easily follows from the definitions. 
\end{proof}

Now we consider the case when $I$ is finite, that can easily be reduced to the case of binary products,  so we assume $I=\{1,2\}$ without loss of generality and we have two normed monoids $(S_1,v_1)$ and $(S_2,v_2)$. The product and the coproduct have the same underlying monoid $S=S_1\times S_2$, but the norms $v_\oplus$ and $v_\Pi$ on $S$ are different and give different values of the semigroup entropy $h_\Se$ (compare Theorem~\ref{WAT} and the following one).  

\begin{Theorem}[weak Addition Theorem - coproducts\index{weak Addition Theorem}] \label{wAT}
Let $(S_i,v_i)$ be a normed monoid and $\f_i:S_i\to S_i$ an endomorphism for $i=1,2$; moreover, let $(S,v)=(S _1 \oplus S_2,v_\oplus)$ and $\phi=\f_1 \oplus \f_2:S\to S$. Then 
$$h_\Se(\f)\leq h_\Se(\f_1) + h_\Se(\f_2).$$
If $v$ is subadditive, then $$h_\Se(\f) = h_\Se(\f_1) + h_\Se(\f_2).$$
\end{Theorem}
\begin{proof}
For $x\in S$, let $x_1\in S_1$ and $x_2\in S_2$ be such that $x=(x_1,x_2)$. Then 
\begin{align*}
h_\Se(\phi,x)&=\limsup_{n\to \infty}\frac{v(T_n(\phi,x))}{n}\\
&= \limsup_{n\to \infty}\frac{v_1(T_n(\phi_1,x_1))+v_2(T_n(\phi_2,x_2))}{n}\\
&\leq h_\Se(\f_1,x_1)+h_\Se(\f_2,x_2).
\end{align*}
Therefore, applying the definition and \eqref{easy:claim}, we have that $h_\Se(\phi)\leq  h_\Se(\f_1)+h_\Se(\f_2)$.

\smallskip
Assume that the norm $v$ is subadditive; clearly this occurs if and only if both norms $v_1, v_2$ are subadditive.
Hence, by Theorem~\ref{limit},
\begin{align*}
h_\Se(\phi,x)&=\limsup_{n\to \infty}\frac{v(T_n(\phi,x))}{n}\\
&= \limsup_{n\to \infty}\frac{v_1(T_n(\phi_1,x_1))+v_2(T_n(\phi_2,x_2))}{n}\\
&= h_\Se(\f_1,x_1)+h_\Se(\f_2,x_2).
\end{align*}
Applying the definition and \eqref{easy:claim}, we conclude that $h_\Se(\phi)= h_\Se(\phi_1)+h_\Se(\phi_2)$.
\end{proof}

Now we consider the main example in entropy theory, that is, the Bernoulli shifts, in the category of normed monoids.

\begin{Example}
For a normed monoid $(M,v) \in \mathfrak  M$, let $B(M)$ be the direct sum $M^{(\N)}$ equipped with the coproduct norm.
The \emph{right Bernoulli shift}\index{Bernoulli shift} is defined by $$\newsym{right Bernoulli shift}{\beta_M}:B(M)\to B(M),\quad  \beta_M(x_0,\ldots,x_n,\ldots)=(1,x_0,\ldots,x_n,\ldots),$$
while the \emph{left Bernoulli shift} is $$\newsym{left Bernoulli shift}{{}_M\beta}:B(M)\to B(M),\ {}_M\beta(x_0,x_1,\ldots,x_n,\ldots)=(x_1,x_2, \ldots,x_n,\ldots).$$
Then:
\begin{itemize}
\item[(a)] $h_\Se({}_M\beta) = 0$;
\item[(b)] $h_\Se (\beta_M)=\sup_{x\in M}v(x)$.
\end{itemize}

To verify (a), note that for every $\underline{x}=(x_n)_{n\in\N}\in B(M)$, there exists $m\in\N_+$ such that ${}_M\beta^m(x) = 1$. So 
$$T_n({}_M\beta,x)=T_m({}_M\beta,x)\ \text{for every}\ n\geq m,$$ hence dividing by $n$ and letting $n$ converge to infinity we obtain $h_\Se({}_M\beta,\underline{x})=0$.

We are left with (b). For $x\in M$ consider $\underline{x}=(x_n)_{n\in\N}\in B(M)$ such that $x_0=x$ and $x_n=1$ for every $n\in\N_+$.  
Then, for every $n\in\N_+$, $$v_\oplus(T_n(\beta_M,\underline{x}))=n\cdot v(x),$$ so $h_\Se(\beta_M,\underline{x})=v(x)$. Hence, $h_\Se(\beta_M)\geq \sup_{x\in M}v(x)$.  
Now let $\underline{x}=(x_n)_{n\in\N}\in B(M)$ and let $k\in\N$ be the greatest index such that $x_k\neq 1$; then, for every $n\geq k$,
\begin{equation*}\begin{split}
v_\oplus(T_n(\beta_M,\underline{x}))&= \sum_{i=0}^{k+n-1} v(T_n(\beta_M,\underline{x})_i)\\
&= \sum_{i=0}^{k-1} v(x_i\cdot\ldots\cdot x_0) + (n-k)\cdot v(x_k\cdot\ldots\cdot x_0)+\sum_{i=1}^{k} v(x_k\cdot\ldots\cdot x_i).
\end{split}\end{equation*}
Since the first and the last summand do not depend on $n$, after dividing by $n$ and letting $n$ tend to infinity, we obtain that  
$$h_\Se(\beta_M,\underline{x})=\lim_{n\to \infty} \frac{v_\oplus(T_n(\beta_M,\underline{x}))}{n}= v(x_k\cdot\ldots\cdot x_0)\leq \sup_{x\in M}v(x),$$ 
and this gives the equality in (b).
\end{Example}

\subsection{Alternatives for the definition of trajectories}\label{alternative}

Here we discuss a possible different notion of semigroup entropy. 
Let $(S,v)$ be a normed semigroup, $\phi:S\to S$ an endomorphism, $x\in S$ and $n\in\N_+$. One could define the ``left'' $n$-th $\phi$-trajectory of $x$ as $$T_n^{\#}(\phi,x)=\phi^{n-1}(x)\cdot\ldots\cdot\phi(x)\cdot x,$$ changing the order of the factors with respect to the above definition of $T_n(\phi,x)$.
With these new trajectories it is possible to define another entropy letting 
$$h_\Se^{\#}(\phi,x)=\limsup_{n\to\infty}\frac{v(T_n^{\#}(\phi,x))}{n},$$ and 
$$h_\Se^{\#}(\phi)=\sup\{h_\Se^{\#}(\phi,x):x\in S\}.$$ 
In the same way as above, one can see that the limit superior in the definition of $h_\Se^{\#}(\phi,x)$ is a limit when $v$ is subadditive.

\medskip
Obviously $h_\Se^{\#}$ and $h_\Se$ coincide on the identity map and in commutative normed semigroups. Anyway, one can produce examples of flows (e.g., see Example~\ref{Ex_ast}(b)) whose entropy $h_\Se$ differs from their ``left entropy" $h_\Se^\#$. On the other hand, in the next result we describe the relation between these two notions of entropy in the case of an automorphism  (compare with Lemma~\ref{inversion}).

\begin{Proposition}\label{sharp}
Let $(S,v)$ be a normed semigroup and let $\phi:S\to S$ be an automorphism. Then $$h_{\Se}(\f^{-1})=h^\#_{\Se}(\f).$$
\end{Proposition}
\begin{proof}
It suffices to see that $h_{\Se}(\f^{-1},x)=h^\#_{\Se}(\f,x)$ for each $x\in S$.  Indeed, given $x\in S$ and $n\in\N_+$,
\begin{align*} v(T_n(\phi^{-1},x)) &=  v(x\cdot \f^{-1}(x)\cdot \ldots \cdot \f^{-n+1}(x))\\
&=v(\f^{-n+1}(\f^{n-1}(x)\cdot \ldots \cdot \f(x)\cdot x)=\\
&= v(\f^{n-1}(x)\cdot \ldots \cdot \f(x)\cdot x)=v(T^\#_n(\f,x)).\end{align*}
Apply the definitions of $h_\Se(\f^{-1},x)$ and $h^\#_\Se(\f,x)$ to conclude. 
\end{proof}

Next we give suitable hypotheses on a flow in $\Se$ to conclude that $h_\Se$ coincides with $h_\Se^\#$. Recall that an anti-isomorphism $i:S\to S$ is a bijective map such that $i(x\cdot y)=i(y)\cdot i(x)$ for every $x,y\in S$. The second statement in the following proposition should be compared with Lemma~\ref{inversion}.

\begin{Proposition}\label{antiiso}
Let $(S,v)$ be a normed semigroup, $\phi:S\to S$ an endomorphism and assume that there exists an anti-isomorphism $i:S\to S$ such that  $i\circ\phi=\phi\circ i$ and $v(i(x))=v(x)$ for every $x\in S$. Then $$h_\Se(\phi)=h_\Se^\#(\phi).$$
In particular, if $\phi$ is an automorphism, then $$h_\Se(\phi)=h_\Se(\phi^{-1}).$$
\end{Proposition}
\begin{proof}
For every $n\in\N_+$ and $x\in S$ we have that $v(T_n(\phi,x))=v(i(T_n(\phi,x))$ and 
$$i(T_n(\phi,x))=i(\phi^{n-1}(x)) \cdot\ldots\cdot i(\phi(x))i(x)=\phi^{n-1}(i(x))\cdot\ldots\cdot \phi(i(x))i(x)=T_n^\#(\phi,i(x));$$ 
so, $v(T_n(\phi,x))=v(T_n^\#(\phi,i(x)))$. Therefore $h_\Se(\phi,x)=h_\Se^\#(\phi,i(x))$, and hence $h_\Se(\phi)=h_\Se^\#(\phi)$.

That $h_\Se(\phi)=h_\Se(\phi^{-1})$ in case $\phi$ is an automorphism follows from the first assertion and from Proposition~\ref{sharp}.
\end{proof}  
 
Part (a) of the following example shows that it may occur that $h_\Se^{\#}$ and $h_\Se$ do not coincide ``locally'', while they coincide ``globally''.
Moreover, modifying appropriately the norm in part (a), Jan Spev\'ak found the example in part (b) for which $h_\Se^{\#}$ and $h_\Se$ do not coincide even ``globally''. This example was given in \cite{DGB_PC},  we repeat it here for the sake of completeness.

\begin{Example}\label{Ex_ast}
Let $X=\{x_n\}_{n\in\Z}$ be a faithfully enumerated countable set and let $S$ be the free semigroup generated by $X$. An element $w\in S$ is a word $w=x_{i_1}x_{i_2}\ldots x_{i_m}$ with $m\in\N_+$ and $i_j\in\Z$ for $j= 1,2, \ldots,  m$. In this case $m$ is called the {\em length} $\ell_X(w)$ of $w$, and a subword of $w$ is any $w'\in S$ of the form $w'=x_{i_k}x_{i_k+1}\ldots x_{i_l}$ with $1\le k\le l\le n$. 

Consider the automorphism $\phi:S\to S$ determined by $\phi(x_n)=x_{n+1}$ for every $n\in\Z$.

\begin{itemize}\label{ex-jan}
\item[(a)] Let $s(w)$ be the number of adjacent pairs $(i_k,i_{k+1})$ in $w$ such that $i_k<i_{k+1}$. The map $v:S\to\R_{\geq0}$ defined by $v(w)=s(w)+1$ is a subadditive semigroup norm. 
Then $\phi:(S,v)\to (S,v)$ is an automorphism of normed semigroups.

It is straightforward to prove that, for $w=x_{i_1}x_{i_2}\ldots x_{i_m}\in S$:
\begin{itemize}
\item[(i)] $h_\Se^\#(\phi,w)=h_\Se(\phi,w)$ if and only if $i_1>i_m+1$ or $i_1\leq i_m-1$;
\item[(ii)] $h_\Se^\#(\phi,w)=h_\Se(\phi,w)-1$ if and only if $i_m=i_1$ or $i_m=i_1-1$.
\end{itemize}
Moreover, 
\begin{itemize}
\item[(iii)]$h_\Se^\#(\phi)=h_\Se(\phi)=\infty$.
\end{itemize}
In particular, $h_\Se(\phi,x_0)=1$ while $h_\Se^\#(\phi,x_0)=0$.

\item[(b)] Define a subadditive semigroup norm $\nu: S\to \R_{\geq0}$ as follows. For 
$$w=x_{i_1}x_{i_2}\ldots x_{i_n}\in S$$ 
consider its subword $w'=x_{i_k}x_{i_{k+1}}\ldots x_{i_l}$ with maximal length satisfying $$i_{j+1}=i_j+1$$ for every $j\in\Z$ with $k\le j\le l-1$ and let $\nu(w)=\ell_X(w')$.
Then $\phi:(S,\nu)\to (S,\nu)$ is an automorphism of normed semigroups.

It is possible to prove that, for $w\in S$:
\begin{enumerate}
\item[(i)] if $\ell_X(w)=1$, then $\nu(T_n(\phi,w))=n$ and $\nu(T^\#_n(\phi,w))=1$ for every $n\in\N_+$;
\item[(ii)] if $\ell_X(w)=k$ with $k>1$, then $\nu(T_n(\phi,w))< 2k$ and $\nu(T^\#_n(\phi,w))< 2k $ for every $n\in\N_+$.
\end{enumerate}
From (i) and (ii), and from the definitions, we immediately obtain that
\begin{itemize}
\item[(iii)] $h_\mathfrak{S}(\phi)=1\neq 0=h^\#_\mathfrak{S}(\phi)$. 
\end{itemize}
\end{itemize}
\end{Example}

\section{The functorial entropy}\label{f-sec}

\subsection{Definition and basic properties}

In this section $\mathfrak X$ will always be a category and $F:\mathfrak X\to \Se$ a functor. 

\medskip
We define below the functorial entropy and establish its basic properties. We give the proofs only for the case when the functor $F:\mathfrak X\to \Se$ is covariant (for contravariant $F$ one can consider $F:\mathfrak X^{op}\to \Se$).

Recall that we write $\h_F:\mathfrak X\to \mathbb R_+$ in place of $\h_F:\mathrm{Flow}_\mathfrak X\to \mathbb R_+$ for the sake of simplicity.

\begin{Definition}
Define the \emph{functorial entropy}\index{functorial entropy} ${\h_{F}}:\mathfrak X\to \mathbb R_+$ on the category $\mathfrak X$ by letting, for any endomorphism $\phi: X \to X$ in $\mathfrak X$, $$\newsym{functorial entropy of $\phi$}{\h_F(\phi)}=h_{\Se}(F(\phi)).$$
For $x\in F(X)$, \emph{the functorial entropy of $\phi$ with respect to $x$} is $$\newsym{functorial entropy of $\phi$ with respect to $x$}{\H_F(\phi,x)}=h_\Se(F(\phi),x).$$
\end{Definition}

Clearly, in the above notation, $$\h_F(\phi)=\sup_{x\in F(X)}\H_F(\phi,x).$$

The same definition can be naturally extended to the more general case of a functor $F:\mathfrak X\to \Se^*$,  for which we keep the same notation $\h_F$.

\medskip
Now we prove the basic properties of the functorial entropy. The proofs of most of them require the target of the functor to be in $\Se$. 

\begin{Lemma}[Invariance under conjugation\index{Invariance under conjugation}]
If $(X,\phi)$ and $(Y,\psi)$ are isomorphic flows of $\mathfrak X$, then $$\h_F(\phi)=\h_F(\psi).$$
\end{Lemma}
\begin{proof}
Let $\alpha:X\to Y$ be an isomorphism in $\mathfrak X$ such that $\psi=\alpha\circ\phi\circ\alpha^{-1}$.
Since $F(\alpha):F(X)\to F(Y)$ is an isomorphism in $\Se$, we have that $F(\psi)=F(\alpha)\circ F(\phi)\circ F(\alpha)^{-1}$,
and it suffices to apply Corollary~\ref{iuc} to conclude that $$\h_F(\psi)=h_\Se(F(\alpha)\circ F(\phi)\circ F(\alpha)^{-1})=h_\Se(F(\phi))=\h_F(\phi).$$
This completes the proof.
\end{proof}

Next we see the an invertible flow of $\mathfrak X$ and its inverse flow have the same functorial entropy.

\begin{Lemma}[Invariance under inversion\index{Invariance under inversion}]\label{inv}
Let $\phi:X\to X$ be an automorphism in $\mathfrak X$ and $F(X)$ a commutative normed semigroup. Then $$\h_F(\f^{-1})=\h_F(\f).$$
\end{Lemma}
\begin{proof}
Since $F(\phi):F(X)\to F(X)$ is an automorphism in $\Se$, Lemma~\ref{inversion} gives immediately $$\h_F(\f^{-1})=h_\Se(F(\phi)^{-1})=h_\Se(F(\phi))=\h_F(\f),$$
and this concludes the proof.
\end{proof}

\begin{Lemma}[Logarithmic Law\index{Logarithmic Law}] 
Let $(X,\phi)$ be a flow of $\mathfrak X$. Then $\h_F(\f^{k})\leq k\cdot \h_F(\f)$ for all $k\in \N_+$. If $F(X)$ is commutative and has a d-monotone norm, then for all $k\in \N_+$, $$\h_F(\f^{k})= k\cdot \h_F(\f).$$
Moreover, if $\phi$ is an automorphism, then  for all $k\in \Z\setminus\{0\}$, $$\h_F(\f^{k})=|k|\cdot \h_F(\f).$$
\end{Lemma}
\begin{proof}
According to the definition, $$\h_F(\f^{k})=h_\Se(F(\phi^k))=h_\Se(F(\phi)^k)$$ and $\h_F(\phi)=h_\Se(F(\phi))$, so it suffices to apply Proposition~\ref{ll}.
\end{proof}

Now we see that, under suitable conditions, the functorial entropy of quasi-periodic flows vanishes.  In particular, the entropy of the identity morphism is zero.

\begin{Lemma}[Vanishing on quasi-periodic flows\index{Vanishing on quasi-periodic flows}]\label{teo:provvisorio}
If $F:\mathfrak X\to \Se^\dag$ has the property that all $F(X)$ have arithmetic norm, then for every quasi-periodic flow $(X,\phi)$ of $\mathfrak X$ one has $\h_F(\phi) = 0$. 
\end{Lemma}
\begin{proof}
If $(X,\phi)$ is a quasi-periodic flow of $\mathfrak X$, then $(F(X),F(\phi))$ is a quasi-periodic flow of $\Se^\dag$ and the norm of $F(X)$ is arithmetic by hypothesis. Then, by Theorem~\ref{quasi-per}, we conclude that $\h_F(\phi)=h_\Se(\F(\phi))=0$.
\end{proof}

 In the following lemma we see that the functorial entropy of a subflow of a flow $(X,\phi)$ of $\mathfrak X$ is always smaller than the functorial entropy of $(X,\phi)$.

\begin{Lemma}[Monotonicity for subflows\index{Monotonicity for subflows}]\label{Fmons} 
 Let $(X,\phi)$ be a flow of $\mathfrak X$ and $Y$ a $\f$-invariant subobject of $X$. If $F$ is covariant and $F(Y)$ is a subsemigroup of $F(X)$ (or if $F$ is contravariant and $F(Y)$ is a factor of $F(X))$, then $$\h_F(\f\restriction_Y)\leq \h_F(\f).$$
\end{Lemma}
\begin{proof}
Assume that $F$ is covariant and that $F(Y)$ is a subsemigroup of $F(X)$. Since $$\h_F(\phi\restriction_Y)=h_\Se(F(\phi\restriction_Y))=h_\Se(F(\phi)\restriction_{F(Y)}),$$ Lemma~\ref{mons} can be applied to conclude that $h_\Se(F(\phi)\restriction_{F(Y)})\leq h_\Se(F(\phi))=\h_F(\phi)$.
\end{proof}

Now we see that under suitable conditions, if $(X,\phi)$ and $(Y,\psi)$ are flows of $\mathfrak X$ such that $(Y,\psi)$ is a factor of $(X,\phi)$, then $\h_F(\psi) \leq \h_F(\phi)$.

\begin{Lemma}[Monotonicity for factors\index{Monotonicity for factors}]\label{Fmonf}
Let $(X,\phi)$ be a flow of $\mathfrak X$ and $\alpha:X\to Y$ a quotient in $\mathfrak X$ such that $\alpha\circ \phi = \psi\circ \alpha$.
$$\xymatrix{
X\ar[r]^\phi\ar[d]_\alpha & X\ar[d]^\alpha \\
Y \ar[r]_{\psi}&Y
}$$
If $F$ is covariant and $F(\alpha):F(X)\to F(Y)$ is a surjective homomorphism in $\Se$ (or, if $F$ is contravariant and $F(\alpha):F(Y)\to F(X)$ is a subobject embedding in $\Se$), then $$\h_F(\psi) \leq \h_F(\phi).$$
\end{Lemma}
\begin{proof}
Assume that $F$ is covariant and $F(\alpha):F(X)\to F(Y)$ is a surjective homomorphism in $\Se$. Since $F(\alpha\circ\phi)=F(\psi\circ\alpha)$ implies $F(\alpha)\circ F(\phi)=F(\psi)\circ F(\alpha)$, Lemma~\ref{mono/fac} yields $$\h_F(\psi)=h_\Se(F(\psi))\leq h_\Se(F(\phi))=\h_F(\phi),$$
and this concludes the proof.
\end{proof}

Next we discuss the ``continuity" of the functorial entropy with respect to direct and inverse limits.

\begin{Lemma}\label{limfin}\index{Continuity for direct/inverse limits}
Assume that for every flow $(X,\phi)$ of $\mathfrak X$, $F(X)$ is a preordered normed semigroup such that the norm is monotone, and $F(\phi):F(X)\to F(X)$ is monotone as well.
\begin{itemize}
\item[(a)] \emph{(Continuity for direct limits)}
Assume that the functor $F$ is covariant.
Let $(X,\phi)$ be a flow of $\mathfrak X$ and $X=\varinjlim X_i$, with $X_i$ a $\phi$-invariant subobject of $X$ for every $i\in I$. If $\varinjlim F(X_i)$ is cofinal in $F(X)$, then 
\begin{equation}\label{new:eq}
\h_F(\phi)=\sup_{i\in I} \h_F(\phi\restriction_{X_i}).
\end{equation}
\item[(b)] \emph{(Continuity for inverse limits)}
If $F$ is contravariant and $(X,\phi)$ is a flow of $\mathfrak X$ such that $X=\varprojlim X_i$, with $(X_i,\phi_i)$ a factor of $(X,\phi)$ for every $i\in I$, and $\varinjlim F(X_i)$ is cofinal in $F(X)$, one has
\begin{equation}\label{new:eq1}
\h_F(\phi)=\sup_{i\in I} \h_F(\phi_i).
\end{equation}
\end{itemize}
\end{Lemma}
\begin{proof}
(a) Let $Y = \varinjlim F(X_i)$.  By Lemma~\ref{mons},  $h_\Se(F(\phi))=h_\Se(F(\phi \restriction_Y))$. On the other hand,  
$$h_\Se(F(\phi \restriction_Y))=\sup_{i\in I} h_\Se(F(\phi\restriction_{X_i})),$$ 
by Proposition~\ref{cont}. This proves \eqref{new:eq}. 

(b) As $F$ is contravariant, the functor $F:\mathfrak X^{op}\to \Se$ is covariant, so (a) applies.
\end{proof}

\begin{Corollary}\label{contlim}
\begin{itemize}
\item[(a)] \emph{(Continuity for direct limits)}
Assume that the functor $F$ is covariant, sending direct limits to direct limits, and $(X,\phi)$ is a flow of $\mathfrak X$ with $X=\varinjlim X_i$,  such that $X_i$ a $\phi$-invariant subobject of $X$ for every $i\in I$. Then \eqref{new:eq} holds. 
\item[(b)] \emph{(Continuity for inverse limits)}
For a contravariant functor $F$, sending inverse limits to direct limits, and a flow $(X,\phi)$ of $\mathfrak X$ such that $X=\varprojlim X_i$, with $(X_i,\phi_i)$ a factor of $(X,\phi)$ for every $i\in I$,  \eqref{new:eq1} holds. 
\end{itemize}
\end{Corollary}

Now we pass to finite products and coproducts.

\begin{Lemma}\label{wATg}
Assume that the functor $F:\mathfrak X\to\mathfrak M$ preserves subobjects and that, for every object $X$ in $\mathfrak X$, $F(X)$ is a preordered normed monoid such that the norm is subadditive and monotone.

 Let $(X,\phi)$ and $(Y,\psi)$ be flows of $\mathfrak X$ such that $F(\phi):F(X)\to F(X)$ and $F(\psi):F(Y)\to F(Y)$ are monotone. 
\begin{itemize}
\item[(a)] Assume that $F$ is covariant and sends the finite coproduct $X\oplus Y$ in $\mathfrak X$ to an object $F(X\oplus Y)$ in $\mathfrak M$ such that $F(X)\oplus F(Y)$ is contained and cofinal in $F(X\oplus Y)$.  Then $$\h_F(\phi\oplus\psi)=\h_F(\phi)+\h_F(\psi).$$
\item[(b)] Assume that $F$ is contravariant and sends the finite product $X\times Y$ in $\mathfrak X$ to an object $F(X\times Y)$ in $\mathfrak M$ such that $F(X)\oplus F(Y)$ is contained and cofinal in $F(X\times Y)$.  Then $$\h_F(\phi\times\psi)=\h_F(\phi)+\h_F(\psi).$$
\end{itemize}
\end{Lemma}
\begin{proof}
(a) Let $f = \phi\oplus\psi$. 
We first show that the subobject $F(X)\oplus F(Y)$ of $F(X\oplus Y)$ is $F(f)$-invariant. To this end we note that $X$ and $Y$ are $f$-invariant subobjects of 
$X\oplus Y$. This obviously implies that both  $F(X)$ and $F(Y)$ are $F(f)$-invariant subobjects of 
$F(X\oplus Y)$. Then $F(X)\oplus F(Y)$, as a sum of two $F(f)$-invariant subobjects, is still an $F(f)$-invariant subobject of $F(X\oplus Y)$.
Now we can apply Lemma~\ref{mons} to the $F(f)$-invariant subobject $F(X)\oplus F(Y)$ of $F(X\oplus Y)$ to deduce 
$$h_\Se(F(\phi\oplus\psi))=h_\Se(F(\phi)\oplus F(\psi)).$$ 
Furthermore, Theorem~\ref{wAT} yields the equality $$h_\Se(F(\phi)\oplus F(\psi))=h_\Se(F(\phi))+h_\Se(F(\psi)).$$ To conclude it suffices to apply the definition of functorial entropy.

(b) As $F$ is contravariant, the functor $F:\mathfrak X^{op}\to \Se$ is covariant, so (a) applies.
\end{proof}

The first part of the following result follows from the above proposition, the second part from the first one by the Duality Principle in category theory.

\begin{Corollary}[weak Addition Theorem\index{weak Addition Theorem}]\label{wATco}
Let $(X,\phi)$ and $(Y,\psi)$ be flows of $\mathfrak X$.
\begin{itemize}
\item[(a)] Assume that $F:\mathfrak X\to \mathfrak M^\dag$ is covariant and sends finite coproducts in $\mathfrak X$ to finite coproducts in $\mathfrak M^\dag$.  
Then $$\h_F(\phi\oplus\psi)=\h_F(\phi)+\h_F(\psi).$$
\item[(b)] Assume that $F:\mathfrak X\to \mathfrak M^\dag$ is contravariant and sends finite products in $\mathfrak X$ to finite coproducts in $\mathfrak M^\dag$.
Then $$\h_F(\phi\times\psi)=\h_F(\phi)+\h_F(\psi).$$
\item[(a$^*$)]  Assume that $F:\mathfrak X\to \mathfrak M^\dag$ is covariant and sends finite products in $\mathfrak X$ to finite products in $\mathfrak M^\dag$.  
Then $$\h_F(\phi\oplus\psi)=\max\{\h_F(\phi),\h_F(\psi)\}.$$
\item[(b$^*$)] Assume that $F:\mathfrak X\to \mathfrak M^\dag$ is contravariant and sends finite coproducts in $\mathfrak X$ to finite products in $\mathfrak M^\dag$.
\end{itemize}
\end{Corollary}

\subsection{Shifts}\label{shiftsec}

We define the Bernoulli shifts in an arbitrary category $\mathfrak X$ admitting countably infinite powers. Let $K$ be an object of $\mathfrak X$, consider $K^\N$ and for every $n\in\N$ let $\pi_n:K^\N\to K$ be the $n$-th projection.
The \emph{(one-sided) left Bernoulli shift}\index{one-sided left Bernoulli shift}\index{Bernoulli shift}\index{left Bernoulli shift}
\begin{equation*}\label{Xbeta}
\newsym{(one-sided) left Bernoulli shift}{{}_K\beta}:K^\N\to K^\N
\end{equation*}
is the unique morphism in $\mathfrak X$ such that, for every $n\in\N$,
\begin{equation}\label{pbp}
\pi_n\circ {}_K\beta=\pi_{n+1}.
\end{equation}
$$\xymatrix{K^\N\ar[dr]_{\pi_{n+1}}\ar[r]^{{}_K\beta} & K^\N \ar[d]^{\pi_n}\\
& K}$$
Analogously, the \emph{(two-sided) left Bernoulli shift}\index{two-sided left Bernoulli shift}
\begin{equation}\label{Xbarbeta}
\newsym{(two-sided) left Bernoulli shift}{{}_K\bar\beta}:K^\Z\to K^\Z
\end{equation}
is the unique morphism in $\mathfrak X$ such that \eqref{pbp} holds for every $n\in\Z$.

Moreover, the \emph{(two-sided) right Bernoulli shift}\index{two-sided right Bernoulli shift}
\begin{equation}\label{barbetaX}
\newsym{(two-sided) right Bernoulli shift}{\bar\beta_K}:K^\Z\to K^\Z
\end{equation}
is the unique morphism in $\mathfrak X$ such that, for every $n\in\Z$,
\begin{equation}\label{pbp*}
\pi_{n+1}\circ \bar\beta_K=\pi_{n}.
\end{equation}
$$\xymatrix{K^\Z\ar[dr]_{\pi_{n}}\ar[r]^{\bar\beta_K} & K^\Z \ar[d]^{\pi_{n+1}}\\
& K}$$
It is easy to deduce from \eqref{pbp} and \eqref{pbp*}, that ${}_K\bar\beta\circ \bar\beta_K =  \bar\beta_K \circ {}_K\bar\beta = id_{K^\Z}$, 
so ${}_K\bar\beta$ and $\bar\beta_K$ are isomorphisms in $\mathfrak X$, inverse to each other.

We will see in \S\ref{known-sec} the Bernoulli shifts in concrete categories.

\medskip
Now we introduce the notion of backward generalized shift in an arbitrary category $\mathfrak X$ admitting arbitrary powers, which extends that of Bernoulli shift.
First we note that this blanket condition on $\mathfrak X$ implies the existence of a terminal object of $\mathfrak X$ that is uniquely determined up to isomorphism (as it is the product of the empty family of objects of $\mathfrak X$; it will be denoted hereinafter by $1$). 
Let $K$ be an object of $\mathfrak X$ and let $X$ be a non-empty set. For $x\in X$ we denote by $\pi_x: K^X \to K$ the projection relative to the $x$-th member 
of the family consisting of $|X|$-many copies of $K$. 

\begin{Definition}\label{bgs}
For non-empty sets $X$, $Y$, and a map $\lambda:X\to Y$, define the \emph{backward} (or \emph{contravariant}) \emph{generalized shift}\index{backward generalized shift}
$$\newsym{backward generalized shift}{\sigma_\lambda}:K^Y \to K^X$$
as the unique morphism in $\mathfrak X$ such that, for every $x\in X$,
\begin{equation*}\label{pbp1}
\pi_x\circ \sigma_\lambda=\pi_{\lambda(x)}, 
\end{equation*}
\end{Definition}

$$\xymatrix{K^Y\ar[dr]_{\pi_{\lambda(x)}}\ar@{.>}[r]^{\sigma_\lambda} & K^X \ar[d]^{\pi_x}\\ & K}$$
If $\mathfrak X$ is a concrete category with forgetful functor $U: \mathfrak X \to \Set$ such that $U(K^X)=U(K)^X$ and the $U(\pi_x)$ are the canonical projections $U(K)^X\to U(K)$ in $\mathbf{Set}$, then
$$U(\sigma_\lambda): U(K)^Y \to U(K)^X$$
is simply the map $f\mapsto f\circ \lambda$. 

For a selfmap $\lambda : X \to X$, the backward generalized shift  $\sigma_\lambda: K^X \to K^X$ is an endomorphism in $\mathfrak X$, so one can discuss its entropy once $ \mathfrak X $ has an entropy defined on its flows. See \S\ref{htop-sec}, \S\ref{sec:h_alg} and \S\ref{BTset} for more details on the backward generalized shifts in specific categories.

\begin{Remark}\label{shiftrem}
If $X=Y=\N$ and $\lambda:\N\to\N$ is defined by $n\mapsto n+1$, then $$\sigma_\lambda={}_K\beta$$ is the one-sided left Bernoulli shift. Analogously, the two sided Bernoulli shifts can be seen as backward generalized shifts as well.
\end{Remark}

Now we see that the generalized shifts represent essentially all contravariant functors $\mathbf{Set}\to\mathfrak X$ sending coproducts to products.
For a fixed $K\in \mathfrak X $, let
\begin{equation}\label{represenatble}
\newsym{functor of the backward generalized shift}{\mathcal B_K}: \mathbf{Set}\to  \mathfrak X 
\end{equation}
be the contravariant functor defined by sending a non-empty set $X$ to $$\mathcal B_K(X)= K^X$$ and $\emptyset$ to the fixed terminal object  $1$ of $\mathfrak X$. 
For a map $\lambda:X\to Y$ let $$\mathcal B_K(\lambda)=\sigma_\lambda:K^Y\to K^X$$ when both $X$ and $Y$ are non-empty, and let $\mathcal B_K(\lambda): K^Y \to 1 =\mathcal B_K(\emptyset)$ be the only morphism to $1$ when $X$ is empty. 

\begin{Remark}
By the functoriality of $\mathcal B_K$, if $\lambda:X\to X$ is a selfmap of a set $X$, then
\begin{equation}\label{powers}
\sigma_{\lambda^m} = (\sigma_\lambda)^m\ \text{for all}\ m\in \N. 
\end{equation}
\end{Remark}

It is not hard to prove that the contravariant functors $\mathcal B_K: \mathbf{Set}\to \X$ send coproducts to products. Up to natural equivalence these are the unique functors $ \mathbf{Set}\to \X$ with this property:

\begin{Theorem}\label{New_Lemma}
Every contravariant functor $\varepsilon: \mathbf{Set}\to \X$ sending coproducts to products is naturally equivalent to 
$\mathcal B_K$ for an appropriate $K\in \X$. 
\end{Theorem}
\begin{proof}
For a singleton $\{x\}$ let $K_x=\varepsilon(\{x\})$, and for a pair of singletons $\{x_1\}, \{x_2\}$ let $$j_{x_1,x_2}:\{x_1\}\to  \{x_2\}$$ be the unique bijective map. This gives an isomorphism
$$\xi_{x_2,x_1}=\varepsilon(j_{x_1,x_2}):K_{x_2}\to K_{x_1}$$
in $\X$.
Fix a singleton $*$ and for every singleton $\{x\}$ rename, for brevity, $j_x =j_{*,x} :*\to\{x\}$ and $K=\varepsilon(*) =K_*$. This gives the following corresponding commutative diagrams of isomorphisms in $\X$.
\begin{equation}\label{jxi}
\xymatrix{ {*} \ar[r]^{j_{x_1}} \ar[dr]_{j_{x_2}} & \{x_1\} \ar[d]^{j_{x_1,x_2}} & & K   & K_{x_1} \ar[l]_{\varepsilon(j_{x_1})} \\
& \{x_2\} & & & K_{x_2} \ar[ul]^{\varepsilon(j_{x_2})} \ar[u]_{\xi_{x_2,x_1}}}
\end{equation}

Every non-empty set $X$ can be written as a coproduct of its singletons $$X=\bigoplus_{x\in X}\{x\}$$ determined by the inclusions $$i_x:\{x\}\to X.$$ Then, by hypothesis,
$$ \varepsilon(X)=\prod_{x\in X}K_x $$ and $$p_x=\varepsilon(i_x):\prod_{x\in X}K_x\to K_x$$
is the canonical projection. We distinguish the product $\prod_{x\in X}K_x$ from the power $K^X$, where all components coincide with $K$. In this case we denote the projection relative to $x\in X$ by $\pi_x: K^X \to K$. 

\smallskip
Since the initial object of $\mathbf{Set}$ is $\emptyset$, while the terminal object of $\X$ is $1$, we deduce that $\varepsilon(\emptyset) =1$, in view of our hypotheses.

The morphism  $\varepsilon(\lambda) :\prod_{y\in Y}K_y\to\prod_{x\in X}K_x$ corresponding to a map $\lambda: X \to Y$ between two sets $X$, $Y$ is the 
unique morphism in  $\X$ such that, for every $x\in X$, the following square commutes.
\begin{equation}\label{eqq}
\xymatrix{\prod_{x\in X}K_x \ar[d]_{p_x} & \prod_{y\in Y}K_y \ar[l]_{\varepsilon(\lambda)} \ar[d]^{p_{\lambda(x)}} \\
K_x & K_{\lambda(x)} \ar[l]^{\xi_{\lambda(x),x}}}
\end{equation}

Let $\eta_\emptyset= id_1:1\to 1$ and, for every non-empty set $X$, let
$$\eta_X:\prod_{x\in X}K_x\to K^X$$
be the unique morphism $\prod_{x\in X}K_x\to K^X$ in $\mathfrak X$ such that, for every $x\in X$, 
\begin{equation}\label{eqqq}
\pi_x\circ\eta_X=\varepsilon(j_x)\circ p_x,
\end{equation}
i.e., the following diagram commutes.
\begin{equation*}
\xymatrix{\prod_{x\in X} K_x\ar[d]_{p_x}\ar[r]^{\eta_X} & K^X\ar[d]^{\pi_x}\\
K_x\ar[r]_{\varepsilon(j_x)} & K}
\end{equation*}

It remains to prove that $\eta$ is a natural equivalence between $\varepsilon$ and $\mathcal B_K$. Since $\eta_X$ is clearly an isomorphism for every $X$ in $\mathbf{Set}$, we need to prove that, for every map $\lambda:X\to Y$ in $\mathbf{Set}$,
\begin{equation}\label{naturaleq}
\eta_X\circ\varepsilon(\lambda)=\sigma_\lambda\circ\eta_Y,
\end{equation}
namely, the following diagram commutes.
$$\xymatrix{
\varepsilon(X)\ar[r]^{\eta_X} & K^X\\
\varepsilon(Y)\ar[r]_{\eta_Y}\ar[u]^{\varepsilon(\lambda)} & K^Y\ar[u]_{\sigma_\lambda}}$$
By the categorical properties of the product $K^X$, \eqref{naturaleq} is equivalent to the conjunction of the equalities 
$$\pi_x\circ \eta_X\circ\varepsilon(\lambda)=\pi_x\circ\sigma_\lambda\circ\eta_Y, \  \  x\in X.$$
Note that  $\pi_x\circ \eta_X = \varepsilon(j_x)\circ p_x$ and $\pi_x\circ\sigma_\lambda= \pi_{\lambda(x)}$, in view of \eqref{eqqq} and the definition of the generalized shift $\sigma_\lambda$. This ensures the first and the last equality in the following chain of five equalities
\begin{equation*}\begin{split}
\pi_x\circ \eta_X\circ\varepsilon(\lambda) = \varepsilon(j_x)\circ p_x \circ\varepsilon(\lambda) = 
\varepsilon(j_x)\circ \xi_{\lambda(x),x}\circ p_{\lambda(x)}= \\= \varepsilon(j_{\lambda(x)})\circ p_{\lambda(x)} = \pi_{\lambda(x)}\circ \eta_Y = \pi_x\circ\sigma_\lambda\circ\eta_Y,
\end{split}\end{equation*}
while the second one follows from \eqref{eqq}, the third one from \eqref{jxi}, and the fourth one from \eqref{eqqq}, applied to $\eta_Y$ and $\lambda(x) \in Y$. 
\end{proof}

Now we introduce the notion of forward generalized shift in an arbitrary category $\mathfrak X$ with arbitrary coproducts. This blanket condition on $\mathfrak X$ implies the existence of an initial object of $\mathfrak X$ that is uniquely determined up to isomorphism (as is the coproduct of the empty family of objects of $\mathfrak X$; it will be denoted by $0$).  Let $K$ be an object of $\mathfrak X$ and let $X$ be a non-empty set. For $x\in X$ we denote by $\iota_x: K \to K^{(X)}$ the canonical morphism relative to the $x$-th member. 

\begin{Definition}\label{fgs}
For non-empty sets $X$, $Y$, and a map $\lambda:X\to Y$, define the \emph{forward} (or \emph{covariant}) \emph{generalized shift}\index{forward generalized shift}
$$\newsym{forward generalized shift}{\tau_\lambda}:K^{(X)} \to K^{(Y)}$$
as the unique morphism in $\mathfrak X$ such that, for every $x\in X$,
\begin{equation}\label{ibi}
\tau_\lambda\circ\iota_x=\iota_{\lambda(x)}. 
\end{equation}
\end{Definition}
$$\xymatrix{K^{(X)}\ar@{.>}[r]^{\tau_\lambda} & K^{(Y)}\\
 K\ar[u]^{\iota_x}\ar[ur]_{\iota_{\lambda(x)}} &
}$$

For a selfmap $\lambda : X \to X$ of a set $X$, the forward generalized shift $\tau_\lambda: K^{(X)} \to K^{(X)}$ is an endomorphism in $\mathfrak X$, so one can discuss its entropy once $\mathfrak X $ has an entropy defined on its flows. See \S\ref{sec:h_alg} and \S\ref{BTset} for more details on the generalized shifts in concrete categories.

\medskip
Now we see  that the forward generalized shifts represent essentially all covariant functors $\mathbf{Set}\to\mathfrak X$ sending coproducts to coproducts.
For a fixed $K\in \mathfrak X $, let
\begin{equation}\label{represenatble2}
\newsym{functor of the forward generalized shift}{\mathcal F_K}: \mathbf{Set}\to  \mathfrak X 
\end{equation}
be the covariant functor defined by sending a non-empty set $X$ to $$\mathcal F_K(X)= K^{(X)}$$ and $\emptyset$ to the fixed initial object $0$ of $\mathfrak X$. 
For a map $\lambda:X\to Y$ let $$\mathcal F_K(\lambda)=\tau_\lambda:K^{(X)}\to K^{(Y)}$$ when both $X$ and $Y$ are non-empty, and let $\mathcal F_K(\lambda): 0 =\mathcal F_K(\emptyset) \to K^{(Y)}$ be the only morphism from $0$, when $X$ is empty. 

It is not hard to prove that the covariant functors $\mathcal B_K: \mathbf{Set}\to \X$ send coproducts to coproducts. That up to natural equivalence these are the unique functors $\mathbf{Set}\to \X$ with this property follows from Theorem~\ref{New_Lemma} by the general Duality Principle in category theory.

\begin{Theorem}\label{New_Lemma2}
Every covariant functor $\gamma: \mathbf{Set}\to \X$ sending coproducts to coproducts is naturally equivalent to $\mathcal F_K$ for an appropriate $K\in \X$. 
\end{Theorem}

\section{Obtaining known dynamical invariants}\label{known-sec}

In this section we describe how the known entropies can be obtained as functorial entropies $\h_F$ for appropriate functors $F:\mathfrak X\to \Se$.
The range of the functors $F:\mathfrak X\to \Se$ considered in \S\ref{known-sec} and \S\ref{BT} is most of the time the category $\SL^\dag$, sometimes $\PSL^\dag$, only once $\mathfrak M_p^\dag$ and once $\Se^*$. Also in \S\ref{NewSec2} we make use of functors $F:\mathfrak X\to \Se^*$. 

\smallskip
By an observation made after the definition of the category $\SL$, for a functor $F:\mathfrak X\to \SL^\dag$ and a morphism $\phi:X_1\to X_2$ in $\mathfrak X$, the morphism $F(\phi)$ is automatically monotone, so we do not write it explicitly each time.

\medskip
We follow the next general scheme.
\begin{enumerate}[(FE1)]
\item Definition of the specific (classical) entropy $h:\mathfrak X\to \R_+$.
\item Description of the assigned normed semigroup and the functor $F:\mathfrak X\to \Se$.
\item Proof of the equality $h=\h_F$.
\item Basic properties of $h$ derived from the known general properties of $\h_F$: 
\begin{enumerate}[-]
\item Invariance under conjugation, 
\item Invariance under inversion, 
\item Logarithmic Law,
\item Vanishing on quasi-periodic flows,
\item Monotonicity for subflows, 
\item Monotonicity for factors, 
\item Continuity for direct/inverse limits, 
\item weak Addition Theorem.
\end{enumerate}
\end{enumerate}

We use the following observations. As the norms are subadditive in all cases considered in this section, the limit in the definition of each entropy is justified in view of Theorem~\ref{limit}.  This holds true with the exception of the contravariant set-theoretic entropy, for which we use a functor with target $\Se^*$ instead of $\Se$. Moreover, as noted in \S\ref{preorder-sec}, if $(S,v)$ is a normed semilattice then $v$ is arithmetic. By Lemma~\ref{cone}, if $(S,v)$ is a normed preordered monoid and $S=P_+(S)$ (in particular, when $S$ is a semilattice), then $v$ is d-monotone.

\subsection{Set-theoretic entropy}\label{set-sec}

First we consider the category $\mathbf{Set}$ of sets and maps and we construct the functor $\atr:\mathbf{Set}\to\SL^\dag$, which gives the covariant set-theoretic entropy $\mathfrak h$ introduced in \cite{AZD} as a functorial entropy. Then we discuss the contravariant set-theoretic entropy $\mathfrak h^*$ from \cite{DG-islam}. These entropies $\mathfrak h$ and $\mathfrak h^*$ are related to invariants for selfmaps of sets (i.e., the string number and the antistring number, see \cite{AADGH,DGV,G0,GV1}).

\smallskip
For a set $X$, denote by \newsym{family of all finite subsets of a set $X$}{$\mathcal S(X)$} the family of all finite subsets of $X$.

\begin{Definition}
Let $X$ be a set and $\lambda: X \to X$  a selfmap. For $D\in\mathcal S(X)$ and $n\in\N_+$ the \emph{$n$-th $\lambda$-trajectory of $D$} is 
$$\mathfrak T_n(\lambda,D) = D\cup\lambda(D)\cup\cdots\cup\lambda^{n-1}(D).$$
The \emph{covariant set-theoretic entropy of $\lambda$ with respect to $D\in\mathcal S(X)$} is
$$\mathfrak h (\lambda, D)=\lim_{n\to\infty} \frac{|\mathfrak T_n(\lambda,D) |}{n}.$$ 
The \emph{covariant set-theoretic entropy}\index{covariant set-theoretic entropy} of $\lambda$ is $\newsym{covariant set-theoretic entropy}{\mathfrak h}(\lambda) = \sup\left\{ \mathfrak h (\lambda, D) :D \in\mathcal S(X)\right\}.$
\end{Definition}

For a set $X$, define $v(A) = |A|$ for every $A\in\mathcal S(X)$. Then:
\begin{itemize}
\item[(i)] $({\mathcal S}(X),\cup,v,\subseteq)$ is a normed semilattice with neutral element $\emptyset$;
\item[(ii)] $v$ is subadditive, arithmetic, monotone and d-monotone. 
\end{itemize}
Consider a map $\lambda:X\to Y$ between sets and define $\atr(\lambda):\mathcal S(X)\to \mathcal S(Y)$ by $A\mapsto \lambda(A)$ for every $A\in\mathcal S(X)$. With $\atr(X)=\mathcal S(X)$, we have a covariant functor 
$$\boxed{\newsym{functor $\atr: \mathbf{Set} \to \mathfrak L^\dag$}{\atr}: \mathbf{Set} \to \mathfrak L^\dag.}$$

\begin{Theorem}
Let $X$ be a set and $\lambda:X\to X$ a selfmap. Then $\mathfrak h(\lambda,D)=\H_{\atr}(\lambda,D)$ for every $D\in\mathcal S(X)$, so $$\mathfrak h(\lambda)=\h_{\atr}(\lambda).$$
\end{Theorem}

\begin{proof} Let $D\in\mathcal S(X)$. Since $\mathfrak T_n(\lambda,D)=T_n(\lambda,D)$ for every $n\in\N_+$, we have that $|\mathfrak T_n(\lambda,D)|=c_n(\atr(\lambda),D)$ for every $n\in\N_+$. Hence, $$\mathfrak h(\lambda, D)=h_\Se(\atr(\lambda),D)=\H_{\atr}(\lambda,D),$$ so the thesis.
\end{proof}

In view of the properties of the functor $\atr$, by the results in \S\ref{f-sec}, it is easy to check that $\mathfrak h$ is invariant under conjugation and under inversion, is monotone for invariant subsets and for functors, satisfies the Logarithmic Law and vanishes on quasi-periodic flows. Moreover, the weak Addition Theorem holds for coproducts in $\mathbf{Set}$, indeed, if $(X_1,\lambda_1)$ and $(X_2,\lambda_2)$ are flows of $\mathbf{Set}$ then for the coproduct $X_1\sqcup X_2$, then $\mathfrak h(\lambda_1\sqcup\lambda_2)=\mathfrak h(\lambda_1)+\mathfrak h(\lambda_2)$.

\medskip
In analogy with the covariant set-theoretic entropy, we give here another notion of entropy for selfmaps, using counterimages in place of images. 

\begin{Definition}
Let $X$ be a set and $\lambda: X \to X$ a finite-to-one selfmap of $X$. For $D\in\mathcal S(X)$ and $n\in\N_+$, the \emph{$n$-th $\lambda$-cotrajectory of $D$} is
$$\mathfrak T_n^*(\lambda,D) = D\cup\lambda^{-1}(D)\cup\ldots\cup\lambda^{-n+1}(D).$$
The \emph{contravariant set-theoretic entropy of $\lambda$ with respect to $D\in\mathcal S(X)$} is $$\mathfrak h^*(\lambda, D)=\limsup_{n\to\infty} \frac{|\mathfrak T_n^*(\lambda,D)|}{n}.$$
The \emph{contravariant set-theoretic entropy}\index{contravariant set-theoretic entropy} of $\lambda$ is $\newsym{contravariant set-theoretic entropy}{\mathfrak h^*}(\lambda) = \sup\left\{ \mathfrak h^*(\lambda, D) :D \in\mathcal S(X)\right\}.$
\end{Definition}

\begin{Remark}\label{h*tilde}
For surjective finite-to-one selfmaps the contravariant set-theoretic entropy ${\mathfrak h^*}$ defined above coincides
with the  contravariant set-theoretic entropy defined in \cite{DG-islam} (this occurs for injective selfmaps as well, see below). We denote  here that entropy by $\mathfrak h_p^*$, in order to distinguish it from $\mathfrak h^*$, as these two entropies may differ for non-surjective finite-to-one selfmaps.

To recall the definition of $\mathfrak h_p^*$, we first recall the definition of surjective core of a map $\lambda:X\to X$ of a set $X$, given by 
$$\mathrm{sc}(\lambda)=\bigcap_{n\in\N}\lambda^n(X).$$ 
Then $\lambda\restriction_{\mathrm{sc}(\lambda)}$ is surjective
and this is the largest restriction of $\lambda$ that is surjective. 

For a selfmap $\lambda:X\to X$ in $\mathbf{Set}_\mathrm{fin}$, $$\mathfrak h^*_p(\lambda)=\mathfrak h^*(\lambda\restriction_{\mathrm{sc}(\lambda)}).$$ 
Clearly, $\mathfrak h_p^*(\lambda)\leq \mathfrak h^*(\lambda)$ for every $\lambda:X\to X$ in $\mathbf{Set}_\mathrm{fin}$.

For every $D\in\mathcal S(X)$, $$\mathfrak h^*( \lambda,D) =\mathfrak h^*(\lambda,D\cap \mathrm{sc}(\lambda)),$$
as the increasing chain $\{\mathfrak T_n^*(\lambda,D)\}_{n\in\N_+}$ stabilizes whenever $D\cap \mathrm{sc}(\lambda) = \emptyset$. 
Then the computation of  $\mathfrak h^*(\lambda,D)$ can be limited to finite subsets $D$ of $\mathrm{sc}(\lambda)$. Nevertheless, even for $D \subseteq \mathrm{sc}(\lambda)$, the trajectory $\mathfrak T_n^*(\lambda,D)$ may be much larger than $\mathfrak T_n^*(\lambda\restriction_{\mathrm{sc}(\lambda)},D)$; in particular, $\mathfrak h_p^*(\lambda)=1< \infty = \mathfrak h^*(\lambda)$ may occur (for an example see \cite[Remark 3.2.41]{DG-islam}). 

On the other hand, if $\lambda$ is injective, then $\mathfrak T_n^*(\lambda,D)\subseteq  \mathrm{sc}(\lambda)$
for every finite $D \subseteq  \mathrm{sc}(\lambda)$, so 
$$\mathfrak h^*(\lambda)=\mathfrak h^*(\lambda\restriction_{\mathrm{sc}(\lambda)}) = \mathfrak h^*_p(\lambda).$$

\smallskip
Our preference to $\mathfrak h^*$ here is based on the possibility to obtain it as a functorial entropy (see Theorem~\ref{PakiPaki}) in the sense of this paper. Further comments on the possibility of obtaining also $\mathfrak h_p^*$ in a functorial way are given in Remark~\ref{TSUNAMI-h*}.
\end{Remark}

The limit superior in the above definition was proved to be a limit when $\lambda$ is surjective in \cite{DG-islam}, even if in general the sequence $\{|\mathfrak T_n^*(\lambda,D)|\}_{n\in\N_+}$  does not need to be subadditive, as the following example from \cite{DG-islam} shows, and so Fekete Lemma does not applies (one can see that it is subadditive when $\lambda$ is injective).

\begin{Example}\label{pakex}
Let $\lambda:\N\to \N$ be a selfmap defined by $\lambda(1)=\lambda(0)=0$, $\lambda(2n+2)=2n$ and $\lambda(2n+3)=2n+1$ for every $n\in\N$.
Then $\mathfrak T_2^*(\lambda,\{0\})=\{0,1,2\}$ and so $|\mathfrak T_2^*(\lambda,\{0\})|=3$, while $\mathfrak T_1^*(\lambda,\{0\})=\{0\}$ and hence $|\mathfrak T^*_1(\lambda,\{0\})|+|\mathfrak T^*_1(\lambda,\{0\})|=2<3$.
\end{Example}

Now we aim to obtain $\mathfrak h^*$ as a functorial entropy. To this end, for a set $X$, let $\str(X)=\mathcal S(X)$, while for a finite-to-one map $\lambda:X\to Y$ the morphism $\str(\lambda):\str(Y)\to\str(X)$ is given by $A \mapsto\lambda^{-1}(A)$. This defines a contravariant functor 
$$\boxed{\newsym{functor $\str: \mathbf{Set}_\mathrm{fin}\to\Se^*$}{\str}: \mathbf{Set}_\mathrm{fin}\to\Se^*.}$$
The necessity to ``enlarge" the target category (from $\Se$ to $\Se^*$) comes from the fact that $\str(\lambda)$ is contractive if and only if $\lambda$ is injective. So one necessarily ends up in $\Se^*$, not in $\Se$. 
In particular, this shows that the blanket hypothesis on contractivity of the endomorphisms in Theorem~\ref{limit}  is necessary.

\begin{Theorem}\label{PakiPaki}
Let $X$ be a set and $\lambda:X\to X$ a selfmap. Then $\mathfrak h^*(\lambda,D)=\H_{\str}(\lambda,D)$ for every $D\in\mathcal S(X)$, so $$\mathfrak h^*(\lambda)=\h_{\str}(\lambda).$$
\end{Theorem}
\begin{proof}
We can assume without loss of generality that $\lambda$ is surjective. Let $D\in\mathcal S(X)$. Since $\mathfrak T_n^*(\lambda,D)=T_n(\lambda,D)$ for every $n\in\N_+$, we have that $|\mathfrak T_n^*(\lambda,D)|=c_n(\str(\lambda),D)$ for every $n\in\N_+$. Hence, $$\mathfrak h^*(\lambda, D)=h_\Se(\str(\lambda),D)=\H_{\str}(\lambda,D),$$
and this concludes the proof.
\end{proof}

It is known from \cite{DG-islam} that the entropy $\mathfrak h_p^*$ is invariant under conjugation and inversion and it is monotone for invariant subsets and for factors. Moreover, $\mathfrak h_p^*$ vanishes on locally quasi-periodic flows, the Logarithmic Law holds and the weak Addition Theorem holds for coproducts
(that is, if $(X_1,\lambda_1)$ and $(X_2,\lambda_2)$ are flows of $\mathbf{Set}$ then their coproduct $\lambda_1\sqcup\lambda_2: X_1\sqcup X_2 \to X_1\sqcup X_2$ satisfies
$\mathfrak h^*_p(\lambda_1\sqcup\lambda_2)=\mathfrak h_p^*(\lambda_1)+\mathfrak h_p^*(\lambda_2)$). 
On the other hand, the Continuity for inverse limits is not available, since it was observed in \cite{AZD} that in the category $\mathbf{Set}_\mathrm{fin}$ the inverse limits need not exist.

The entropy $\mathfrak h^*$ introduced here shares the same properties as those of $\mathfrak h^*_p$, with a few exceptions: 
$\mathfrak h^*$ need not vanishes on locally quasi-periodic flows and we are  not aware whether $\mathfrak h^*$ satisfies the Logarithmic Law (nevertheless, $\mathfrak h^*$ vanishes on quasi-periodic flows).

\subsection{Topological entropy for compact spaces}\label{htop-sec}

In this subsection we show that the topological entropy $h_{top}$ introduced in \cite{AKM} can be obtained as a functorial entropy via an appropriate functor $\cov:\mathbf{CTop}\to\PSL^\dag$, where $\mathbf{CTop}$ is the category of compact spaces and continuous maps.

\medskip
For a topological space $X$ let \newsym{family of all open covers $\mathcal U$ of a topological space $X$}{$\cov(X)$} be the family of all open covers $\mathcal U$ of $X$, with the convention that $\emptyset$ may belong to $\mathcal U$. Clearly, every base of $X$ belongs to $\cov(X)$. For $m\in\N_+$ and $\U_1, \ldots ,\U_m\in\cov(X)$, let 
$$\U_1 \vee \ldots \vee \U_m=\left\{\bigcap _{i=1}^m U_i: U_i\in \U_i\right\}.$$ 
For a continuous selfmap $\phi:X\to X$, $\U\in \cov (X)$ and $n\in\N_+$, let 
$$\phi^{-n}(\U)=\{\phi^{-n}(U): U\in \U\}.$$ 
Then $\phi^{-n}(\U_1\vee \ldots \vee\U_m)=\phi^{-n}(\U_1)\vee \ldots \vee \phi^{-n}(\U_m)$ for every $n,m\in\N_+$.

\begin{Definition}\label{htop-def}
Let $X$ be a compact space and $\phi:X\to X$ a continuous selfmap. For $\U\in\cov(X)$ let $$N(\U)=\min\{|\mathcal V|: \mathcal V\ \text{is a finite subcover of}\ \U\}.$$
The \emph{topological entropy of $\phi$ with respect to $\U\in\cov(X)$} is 
\begin{equation}
H_{top}(\phi,\U)=\lim_{n\to\infty} \frac{\log N(\U\vee \phi^{-1}(\U) \vee \ldots \vee \phi^{-n+1}(\U))}{n}.
\end{equation}
The \emph{topological entropy}\index{topological entropy} of $\phi$ is $$\newsym{topological entropy}{h_{top}}(\phi)=\sup\{H_{top}(\phi,\U):\U \in \cov(X)\}.$$
\end{Definition}

For a topological space $X$, let $\U_X$ denote the largest open cover (i.e., the whole topology of $X$) and $\mathcal E_X=\{X\}$ the trivial cover. 
Then $(\cov(X), \vee, \mathcal E_X)$ is a commutative monoid. This monoid has a natural partial order by inclusion that turns it into a (pre)ordered monoid. In what follows we consider a richer preorder that turns out to be more relevant.

For $\U,\V \in \cov(X)$ we say that $\V$ \emph{refines} $\U$ (denoted by $\U \prec\V $) if for every $V \in \V$ there exists $U\in \U$ such that $V\subseteq U$. Let $\U \sim \V$ if $ \U \prec \V$ and $\V \prec\U$.
Then $\prec$ is a preorder on $\cov(X)$ that is not an order and has bottom element $\mathcal E_X\sim\U_X$ (if $\U$ is an open cover of $X$ and $X\in\U$ then $\U\sim\mathcal E_X$).  

For $\U, \V \in \cov (X)$ we let $$\V_\U=\{V \in \V: V \subseteq U\ \text{for some}\ U\in \U\}.$$ Clearly, $\V_\U = \V$ if and only if $\U\prec \V$.  Moreover,  $\V_\U $ refines $\U$, although it need not be a cover when $\V$ itself does not refine $\U$. If $\V$ is a base, then $\V_\U$ is still a base, so in particular, a cover. In this case $\V_\U$ is a subcover of $\V$ that refines $\U$. 

For a compact space $X$ and $\U\in \cov(X)$, let $$v(\U)=\log N(\U).$$ If $\U\prec\V$ then $v(\U)\leq v(\V)$.

In general,  $\U \vee \U \ne \U $, yet  $\U \vee \U \sim \U $, and more generally $\U \vee \U  \vee \ldots \vee \U \sim \U$.
Therefore, $v(\U \vee \U  \vee \ldots \vee \U) = v(\U)$. 

Then:
\begin{itemize}
\item[(i)] $(\mathfrak{cov}(X), \vee, v,\prec)$ is a normed pre-semilattice with zero $\mathcal E_X$;
\item[(ii)] $v$ is subadditive, arithmetic, monotone and d-monotone.
\end{itemize}
For $X$, $Y$ topological spaces, a continuous map $\f:X\to Y$ and $\U\in \cov (Y)$, let 
$$\cov (\f): \cov (Y)\to \cov (X)$$ 
be defined by $\U \mapsto \f^{-1}(\U)$. Obviously,  $\cov(\phi)$ is monotone with respect to the order $\prec$.
So, this defines a contravariant functor $\mathfrak{cov}$ from the category of all topological spaces to the category of commutative semigroups (actually, presemilattices). 

For every continuous map $\f:X\to Y$ of compact spaces and $\W\in \cov(Y)$, the inequality  $v(\f^{-1}(\W))\leq v(\W)$ holds (if $\phi$ is surjective, then equality holds).
Consequently, the assignments $X \mapsto \cov(X)$ and $\phi\mapsto\cov(\phi)$ define a contravariant functor 
$$\boxed{\newsym{functor $\cov:\mathbf{CTop}\to \PSL^\dag$}{\cov}:\mathbf{CTop}\to \PSL^\dag.}$$

\begin{Theorem}\label{realization:top:ent}
Let $X$ be a compact space and $\phi:X\to X$ a continuous selfmap. Then $H_{top}(\phi,\mathcal U)=\H_{\mathfrak{cov}}(\phi,\mathcal U)$ for every $\mathcal U\in\cov(X)$, and so $$h_{top}(\phi)=\h_{\cov}(\phi).$$
\end{Theorem}
\begin{proof}
Let $\mathcal U\in\cov(X)$. For every $n\in\N_+$, since $$\U\vee \phi^{-1}(\U) \vee \ldots \vee \phi^{-n+1}(\U)=T_n(\cov(\phi),\mathcal U),$$ we conclude that $$\log N(\U\vee \phi^{-1}(\U) \vee \ldots \vee \phi^{-n+1}(\U))=c_n(\cov(\phi),\mathcal U),$$ and so $H_{top}(\phi,\mathcal U)=h_\Se(\cov(\phi),\mathcal U)=\H_{\mathfrak{cov}}(\phi,\mathcal U)$.
\end{proof}

The contravariant functor $\mathfrak{cov}$ takes factors in $\CT$ to subobject embeddings in $\PSL^\dag$, subobject embeddings in $\CT$ to surjective morphisms in $\PSL^\dag$. Therefore, the properties proved in \S\ref{f-sec} yield that the topological entropy $h_{top}$ is invariant under conjugation and inversion, it is monotone for restrictions to invariant subspaces and for factors, it satisfies the Logarithmic Law, and it vanishes on quasi-periodic continuous selfmaps.

The fact that $\prec$ is a preorder on $\cov(X)$ compatible with $\cov(\phi)$ will  now be used to check the Continuity for inverse limits. 
To see this, we apply Proposition~\ref{New:corollary2} below, Lemma~\ref{limfin} and Theorem~\ref{realization:top:ent}, to conclude that $$h_{top}(\phi)=\sup_{i\in I}h_{top}(\phi_i),$$ for an inverse system $\{(X_i,p_{i,j})\}_{i\in I}$ of compact spaces $X_i$ with $p_{i,j}: X_i \to X_j$, for $i,j\in I$ with $j\leq i$, and its inverse limit $X=\varprojlim_{i\in I} X_i$ with canonical projections $p_i: X \to X_i$. Indeed, in this case each $\cov(p_i): \cov(X_i) \to \cov(X)$ is an embedding that allows us to consider $\varinjlim_{i\in I} \cov(X_i)$ in $\cov(X)$. 

\begin{Proposition}\label{New:corollary2} 
In the above notation, $\varinjlim_{i\in I} \cov(X_i)$ is cofinal in $\cov(X)$.  \end{Proposition}
\begin{proof}
For every $i\in I$ we identify $\cov(X_i)$ with the family $$\mathcal B_i^*=\{ \cov(p_i)(\V): \V\in \cov(X_i)\}$$
in $\cov(X)$. Then  $ \varinjlim_{i\in I} \cov(X_i)$ can be identified with $L=\bigcup_{i\in I}\mathcal B_i^*$. It is known that $$\mathcal B= \bigcup L$$ is a base of $X$ by \cite[2.5.5]{E}, in particular $\mathcal B\in \cov(X)$.
 
In order to check that $L$  is cofinal in $\cov(X)$, pick  $\U\in\cov(X)$. Then $\mathcal B_\U$ is a subcover of $\mathcal B$ such that $\U\prec \mathcal B_\U$; moreover, $\mathcal B_\U$ is a base of $X$ as noted above, so $\mathcal B_\U \in \cov(X)$. 
By the compactness of $X$ there exists a finite subcover $$\W = \{ p_{i_{k}}^{-1}(V_{i_k}): V_{i_k}\in \U_{X_{i_k} },  p_{i_{k}}^{-1}(V_{i_k}) \subseteq U_k,\ \text{for some}\ U_k \in \U,\ k = 1, \ldots,n\}$$ of $\mathcal B_\U$, in particular $\U \prec \W$. 
It remains to prove that $$\W\in L.$$ To this end, take an index $i_0 \geq i_k$, for $k = 1,\ldots,n$ and let $$W_k= p_{i_0, i_k}^{-1}(V_{i_k})$$ for $k = 1, \ldots , n$. Since $p_{i_k} = p_{i_0, i_k}\circ p_{i_0}$, we have that $p_{i_{k}}^{-1}(V_{i_k})=p_{i_0}^{-1}(W_k)$, and obviously $W_k$ is an open subset of $X_{i_0}$ for $k = 1, \ldots , n$. 
Let $$\W^*= \{W_k: k = 1, \ldots , n\}.$$
Since $p_{i_0}:X \to X_{i_0}$ is surjective and $\W = p_{i_0}^{-1}(\W^*) \in \cov(X)$, we deduce that 
$\W^*\in \cov(X_{i_0}).$
Hence, $\W= \cov(p_{i_0})(W^*) \in L$.
\end{proof}

It is known that the Weak Addition Theorem holds for the topological entropy, that is, for any pair of continuous selfmaps $\phi:X\to X$ and $\psi:Y\to Y$ one has $$h_{top}(\phi\times\psi)=h_{top}(\phi)+h_{top}(\psi).$$
This was announced in \cite[Theorem 3]{AKM} and correctly proved in \cite{Good}.

The topological entropy satisfies also another version of weak Addition Theorem for coproducts (see \cite[Theorem 4]{AKM} and \cite[Proposition 4.1.9]{DG-islam}). Indeed, if $(X,\phi)$ and $(Y,\psi)$ are flows in $\mathbf{CTop}$, and we consider the coproduct $X\sqcup Y$, then $h_{top}(\phi\sqcup\psi)=\max\{h_{top}(\phi),h_{top}(\psi)\}$.
 This result follows from Corollary~\ref{wATco} and Theorem~\ref{realization:top:ent}.

\smallskip
It is worth recalling that in the computation of the topological entropy it is possible to reduce to surjective continuous selfmaps of compact spaces (see \cite{S,Wa}). 

\medskip
In the category $\mathbf{CTop}$ the Bernoulli shifts defined in \eqref{Xbeta}, \eqref{Xbarbeta} and \eqref{barbetaX} have obviously the following concrete form. 
For $K\in\mathbf{CTop}$
\begin{equation}\label{leftBernoulli1}
{}_K\beta:K^{\N}\to K^{\N},\quad {}_K\beta(x_0,x_1,\ldots,x_n,\ldots)=(x_1,x_2, \ldots,x_{n+1},\ldots),
\end{equation}
while
\begin{equation}\label{leftBernoulli2}
{}_K\bar\beta:K^{\Z}\to K^{\Z},\quad {}_K\beta((x_n)_{n\in\Z})=(x_{n+1})_{n\in\Z},
\end{equation}
and
\begin{equation}\label{rightBernoulli2}
\bar\beta_K:K^{\Z}\to K^{\Z},\quad {}_K\beta((x_n)_{n\in\Z})=(x_{n-1})_{n\in\Z}.
\end{equation} 

As far as the topological entropy of the Bernoulli shifts is concerned, it is known (see also Theorem~\ref{SetBT2star} and Corollary~\ref{reflection:corollary} below) that in case $K$ is a compact Hausdorff space,
\begin{equation}\label{htopbeta}
h_{top}({}_K\beta)=h_{top}({}_K\bar\beta)=h_{top}(\bar\beta_K)=\log|K|,
\end{equation}
with the convention that $\log|K|=\infty$ if $X$ is infinite.

\smallskip
More generally, for a selfmap $\lambda:X\to X$ of a non-empty set $X$, and $K$ a topological space, the backward generalized shift from Definition~\ref{bgs} has the form
\begin{equation}\label{bgseq}
\sigma_\lambda:K^X\to K^X,\quad f\mapsto \lambda\circ f;
\end{equation}
it was introduced and studied in \cite{AADGH,AZD} as a generalized version of the Bernoulli shifts recalled in \eqref{leftBernoulli1}, \eqref{leftBernoulli2}, \eqref{rightBernoulli2} (see Remark~\ref{shiftrem}).
It is known from \cite{AZD} that, if $K$ is a compact Hausdorff space, then
\begin{equation}\label{htoph}
h_{top}(\sigma_\lambda)=\mathfrak h(\lambda)\cdot\log|K|,
\end{equation}
with the convention that $\log|K|=\infty$ if $K$ is infinite. In particular, this covers the formula in \eqref{htopbeta} and will follow from Theorem~\ref{SetBT2star}.

\subsection{Entropy for topological spaces and entropy for frames}\label{fr-sec}

Topological entropy functions for non-compact topological spaces were discussed by Hofer \cite{Hof}, who proposed a quite natural extension of the topological entropy $h_{top}$ to continuous selfmaps of arbitrary topological spaces, by simply replacing the open covers by finite open covers.  For a topological space $X$ let \newsym{family of all finite open covers of a topological space $X$}{$\fc(X)$} denote the subfamily of $\cov(X)$ consisting of all finite open covers of $X$. For every continuous selfmap $\f: X \to X$ and for every $\mathcal U \in \fc(X)$ define $\Hf(\f, \mathcal U)$ and \newsym{topological entropy}{$\hf(\f)$} as in Definition~\ref{htop-def}.

For continuous selfmaps of compact spaces obviously $\hf = h_{top}$, as every open cover of a compact space has a finite open subcover. 

\medskip
Obviously, $\fc(X)$ is a submonoid of the commutative monoid $(\mathfrak{cov}(X), \vee)$. 
It is important to underline that if the topological space $X$ is not compact, then the latter monoid is not normed. Nevertheless, the norm $v=\log N(-)$ is well-defined on the submonoid $\fc(X)$ and furthermore, considering on $\fc(X)$ the restriction of the refinement relation $\prec$ recalled above for $\cov(X)$:
\begin{itemize}
\item[(i)]  $(\fc(X), \vee, v,\prec)$ is a normed presemilattice with zero $\mathcal E_X$;
\item[(ii)] $v$ is subadditive, arithmetic, monotone and d-monotone. 
\end{itemize}
Since for every continuous map $f:X \to Y$ the map $\cov(f): \cov(Y) \to \cov (X)$ sends $\fc(Y)$ to $\fc(X)$, the restriction $\fc(f)$ of $\cov(f)$ to $\fc(Y)$ 
defines a morphism $\fc(f) : \fc(Y) \to \fc(X)$ which is monotone and so it is in $\PSL^\dag$. In this way we obtain a new contravariant functor 
$$\boxed{\newsym{functor $\fc: \Top \to \PSL^\dag$}{\fc}: \Top \to \PSL^\dag.}$$

The proof of the next theorem is similar to that of Theorem~\ref{realization:top:ent}. 

\begin{Theorem}\label{realization:fin-top:ent}
Let $X$ be a topological space and $\phi:X\to X$ a continuous selfmap. Then $\Hf(\phi,\mathcal U)=\H_{\mathfrak{cov}}(\phi,\mathcal U)$ for every $\mathcal U\in\fc(X)$, 
and so 
$$\hf(\phi)=\h_{\fc}(\phi).$$
\end{Theorem}

\begin{Remark}\label{fcrem} The functor $\fc$ does not extend the functor $\cov: \CT \to \PSL^\dag$ defined above. In fact, if $X$ is a compact space, then $$\fc(X)\subseteq \cov(X).$$ Anyway, $\fc(X)$ is cofinal in $\cov(X)$.
Moreover, if $\phi:X\to X$ is a continuous selfmap of a compact space $X$, then $\fc(X)$ is a $\cov(\phi)$-invariant normed subsemigroup of $\cov(X)$, since $\fc(\phi)$ is defined as the restriction of $\cov(\phi)$ to $\fc(X)$.
Then Lemma~\ref{mons} gives that $$\h_{\fc}(\phi)=h_\Se(\fc(\phi))=h_\Se(\cov(\phi))=\h_\cov(\phi),$$ and hence $h_{top}(\phi)=\hf(\phi)$ by Theorems~\ref{realization:top:ent} and~\ref{realization:fin-top:ent}.
\end{Remark}

By the results in \S\ref{f-sec}, $\hf$ is invariant under conjugation and invariant under inversion, moreover it satisfies the Logarithmic Law and it vanishes on quasi-periodic continuous selfmaps.
We consider the monotonicity for closed invariant subspaces and for factors in Theorem~\ref{realization:fintop:ent} below.

\begin{Lemma}\label{surj}
If $Y$ is a closed subspace of a topological space $X$, and $j: Y \hookrightarrow X$ is the subspace embedding of $Y$ in $X$, then $\fc(j): \fc(X) \to \fc(Y)$ is surjective. 
\end{Lemma}
\begin{proof}
For every $\ms{U} = \{U_1, \ldots, U_n\}\in \fc(Y)$ define a cover $\ms{U}^*=  \{U_0^*, U_1^*, \ldots, U_n^*\}\in \fc(X)$, where $U_0^*=X\setminus Y$ and $U_i^*$ is an open set of $X$ such that $U_i^*\cap Y = U_i$ for $i=1,2,\ldots, n$. 
\end{proof}

On the other hand, if $Y$ is not closed in $X$, then  $\fc(j)$ need not be surjective even when $X$ is compact, as the following example shows. 

\begin{Example}\label{ExaNatale}
Let $Y=\N$ endowed with the discrete topology, and $X = a\N$ be the one-point Aleksandrov compactification of $Y$. Let $U_1$ (respectively, $U_2$) be the set of all odd (respectively, even) numbers in $Y$. Then $\ms{U} = \{U_1, U_2\} \in \fc(Y)$, yet $\ms{U} \ne \fc(j)(\ms{U}^*)$ for any $\ms{U}^* \in \fc(X)$. 
\end{Example}

As a corollary of the general properties in \S\ref{f-sec}, we obtain the following properties of the entropy $\hf$, announced in \cite{Hof} and proved in \cite{DK}.

\begin{Theorem} \label{realization:fintop:ent}
Let $X$ be a topological space and $\f: X\to X$ a selfmap.
\begin{itemize}
\item[(a)] If $Y$ is a topological space, $q: X\to Y$ is a  continuous surjective map and $\overline\f: Y \to Y$ a selfmap such that $\overline\f \circ q =q \circ \f$, then $\hf(\overline\f) \leq \hf(\f)$. 
\item[(b)] If $Y$ is a closed invariant subspace of $X$, then $\hf(\f\restriction_Y)\leq  \hf(\f)$. 
\end{itemize}
\end{Theorem}
\begin{proof} 
(a) The map $\fc(q): \fc(Y) \to \fc(X)$ is injective and the image of $ \fc(Y)$ in $ \fc(X)$ is invariant under $\fc(\overline\f)$. Hence, Lemma~\ref{Fmons} applies. 

(b) Let $j: Y \hookrightarrow X$ be the subspace embedding of $Y$ in $X$. As we noticed in Lemma~\ref{surj}, $\fc(j): \fc(Y) \to \fc(X)$ is surjective, so Lemma~\ref{Fmonf} applies. 
\end{proof}

One may object that $\hf(\f\restriction_Y)\leq \hf(\f)$ may still remain true regardless of the fact that $Y$ is closed or not in $X$. That this is not the case follows from an example similar to that provided in Example~\ref{ExaNatale}, where $Y$ is just $\Z$ and $X = aY$ is the one-point Aleksandrov compactification of $Y$. The selfmap $\f: X \to X$ is defined by $\f(p) = p$ (that is the extra point added to $Y$ to get $X = aY$) and $\f(n) = n+1$. Then $\hf(\f\restriction_Y)= \infty$ (see \cite{Hof}), while $\hf(\f) = 0$ (actually, $\hf(\f)= 0$ for every continuous selfmap $\f$ of $X$, according to \cite{AZD}). 

We do not know whether the Continuity for inverse limits and the weak Addition Theorem hold for $\hf$. The arguments applied in the previous section for $h_{top}$ use the compactness of the spaces, so they do not apply to the present case.

\bigskip
A careful analysis of the above definitions of topological entropy shows that the points of the space are completely absent from the definitions. All necessary information to define the entropies is encoded in the complete lattice \newsym{family of all open sets of a topological space $X$}{$\ms{O}(X)$} of all open sets of the topological space $X$. 
Note that in $\ms{O}(X)$ one has the distributive law $$\left(\bigcup_{i\in I} U_i\right )  \cap V = \bigcup_{i\in I} (U_i  \cap V).$$
Recall that an algebraic structure with this property is called a \emph{frame}\index{frame} (or a \emph{complete Heyting algebra}). Namely, a
frame consists of a supporting set $L$ with two operations $\vee$ and $\wedge$ such that $(L,\vee)$ is a complete semilattice, $(L,\vee,\wedge)$ is a distributive lattice and the following stronger distributive law holds as well; in fact,  for arbitrary sets $I$,
$$\left(\bigvee_{i\in I} u_i\right ) \wedge v = \bigvee_{i\in I} (u_i  \wedge v).$$ In particular, $L$ has a top element $1 = \bigvee_{u\in L} u$. A frame homomorphism $\f: L \to L'$ preserves the operations (hence, it preserves the bottom and the top element as well). This defines the category \newsym{category of frames and frame homomorphisms}{$\Fr$} of all frames and frame homomorphisms, and we will introduce entropy in the category $\Fr$. 

\begin{Definition}
Let $L$ be a frame. A \emph{(finite) cover} of $L$ is a (finite) subset $\ms{U}= \{u_i: i \in I\}$ of $L$ such that $\bigvee_{i\in I} u_i = 1$. 
\end{Definition}

If $\;\ms{U}$ is a cover and $\;\ms{U}' \subseteq \ms{U}$ is still a cover of $L$, then we call $\;\mathcal{U}'$ a \emph{subcover} of $\;\mathcal{U}$. 
For two covers $\ms{U}$ and  $\ms{U}'$ of a frame $L$, let $$\ms{U}\vee \ms{U}' = \{u\wedge u': u\in \ms{U}, u' \in \ms{U}'\}.$$ One can check that this is still a cover of $L$. 
Denoting by $\fc_{fr}(L)$ the family of all finite covers of a frame $L$, it is easy to see that if $\ms{U}, \ms{U'} \in \fc_{fr}(L)$, then $\ms{U} \vee \ms{U'} \in \fc_{fr}(L)$. 
This turns $(\fc_{fr}(L), \vee)$ into a commutative monoid with neutral element the cover $\{1\}$.  One can define also a preorder $\prec$ on $\fc_{fr}(L)$ as above given by the refinement.
Moreover, define a norm on $\fc_{fr}(L)$ by letting $v(\ms{U})$ be the logarithm of the minimum size of a subcover of $\ms{U}$. Obviously, $v(\{1\}) = 0$, so $v$ is a monoid norm. Similarly to the case of $\fc$, we have that:
\begin{itemize}
\item[(i)] $(\fc_{fr}(L), \vee, v,\prec)$ is a normed presemilattice with zero $\mathcal E_X$;
\item[(ii)] $v$ is subadditive, arithmetic, monotone and d-monotone. 
\end{itemize}
For a frame homomorphism $\f:L\to L'$ and $\U\in \fc_{fr}(L)$ let $\f(\U) = \{\f(u): u\in \U\}$. This defines a a monotone monoid homomorphism $\fc_{fr}(\f): \fc_{fr}(L) \to \fc_{fr}(L')$, and consequently a covariant functor
$$\boxed{\newsym{functor $\fc_{fr}:\Fr\to\PSL^\dag$}{\fc_{fr}}:\Fr\to\PSL^\dag}$$
by sending $L\mapsto \fc_{fr}(L)$ and $\phi\mapsto \fc_{fr}(\phi)$. 

\medskip
In particular, for every frame endomorphism $\f:L\to L$ and $n\in\N_+$ one has the possibility to define the $n$-th $\phi$-trajectory $T_n(\phi,\ms{U})=\ms{U}\vee \f(\ms{U}) \vee \ldots 
\vee \f^{n-1}(\ms{U}) $ of a (finite) cover $\ms{U}$. If $\ms{U}\in \fc_{fr}(L)$, then also $T_n(\phi,\ms{U})\in \fc_{fr}(L)$, so one can define the frame entropy as follows. 

\begin{Definition} Let $(L,\phi)$ be a flow of $\Fr$. The \emph{frame entropy of $\phi$ with respect to $\ms{U}\in \fc_{fr}(L)$} is
$$H_{fr}(\phi,\mathcal U)=\lim_{n\to\infty}\frac{v(T_n(\phi,\ms{U}))}{n}.$$
The \emph{frame entropy}\index{frame entropy} of $\phi$ is $$\newsym{frame entropy}{h_{fr}}(\phi)=\sup\{H_{fr}(\phi,\mathcal U):\mathcal U\in\fc_{fr}(L)\}.$$ 
\end{Definition}

It follows directly from the definition that $H_{fr}(\phi,\mathcal U)=\H_{\fc_{fr}}(\phi,\mathcal U)$ for every flow $(L,\phi)$ of $\Fr$ and every $\mathcal U\in \fc_{fr}(L)$. So we can conclude that 
\begin{equation}\label{realization:hfr}
h_{fr}=\h_{\fc_{fr}}.
\end{equation}

\subsection{Measure entropy}\label{mes-sec}

In this subsection we consider the category $\MS$ of probability measure spaces $(X, \mathfrak B, \mu)$ and measure preserving maps, constructing a functor $\mathfrak{mes}:\MS\to\mathfrak L^\dag$ in order to obtain from our general scheme the measure entropy $h_{mes}$ from \cite{K} and \cite{Sinai}.

\medskip
We recall that a \emph{measure space} is a triple $(X,{\mathfrak B},\mu)$, where $X$ is a set, $\mathfrak B$ is a $\sigma$-algebra over $X$ (the elements of $\mathfrak B$ are called \emph{measurable sets}) and $\mu:\mathfrak B\to \R_{\geq0}\cup\{\infty\}$ is a probability measure. A selfmap $\psi:X\to X$ is a \emph{measure preserving trasformation} if $\mu(\psi^{-1}(B))=\mu(B)$ for every $B\in\mathfrak B$.
 
\smallskip
For a measure space $(X,{\mathfrak B},\mu)$ and a measurable partition $\xi=\{A_i:i=1,\ldots,k\}$ of $X$, define the \emph{entropy} of $\xi$ by Boltzmann's Formula  $$H(\xi)=-\sum_{i=1}^k \mu(A_i)\log \mu(A_i).$$

For two partitions $\xi, \eta$ of $X$, let $$\xi \vee \eta=\{U\cap V: U\in \xi, V\in \eta\}$$ and define $\xi_1\vee \xi_2\vee \ldots \vee \xi_n$ analogously for partitions $\xi_1,\xi_2,\ldots,\xi_n$ of $X$. For a measure preserving transformation $\psi:X\to X$ and a measurable partition $\xi=\{A_i:i=1,\ldots,k\}$ of $X$, let $$\psi^{-j}(\xi)=\{\psi^{-j}(A_i):i=1,\ldots,k\}.$$

For a measure space $(X,\mathfrak{B},\mu)$ let \newsym{family of measurable partitions of a probability measure space $X$}{$\mathfrak{P}(X)$} 
be the family of all measurable partitions $\xi=\{A_1,A_2,\ldots,A_k\}$ of $X$.

\begin{Definition}
Let $X$ be a measure space and $\psi:X\to X$ a measure preserving transformation. The \emph{measure entropy of $\psi$ with respect to $\xi\in\mathfrak P(X)$} is $$H_{mes}(\psi,\xi)=\lim_{n\to\infty}\frac{H(\bigvee_{j=0}^{n-1}\psi^{-j}(\xi))}{n}.$$
The \emph{measure entropy}\index{measure entropy} of $\psi$ is $$\newsym{measure entropy}{h_{mes}}(\psi)=\sup\{H_{mes}(\psi,\xi): \xi\in\mathfrak P(X)\}.$$
\end{Definition}


For a measure space $(X,\mathfrak{B},\mu)$ and $\xi\in\mathfrak P(X)$, we have that $\xi \vee \xi = \xi$. Consider again on $\mathfrak P(X)$ the preorder $\prec$ given by the refinement. Then:
\begin{itemize}
\item[(i)] $(\mathfrak{P}(X),\vee,H,\prec)$ is a normed semilattice with zero $\xi_0=\{X\}$;
\item[(ii)] $H$ is subadditive (see \cite{Wa}), arithmetic, monotone and d-monotone.
\end{itemize}
Consider a measure preserving map $T:X\to Y$. For $\xi=\{A_i\}_{i=1}^k\in \mathfrak{P}(Y)$ let $$T^{-1}(\xi)=\{T^{-1}(A_i)\}_{i=1}^k.$$ 
Since $T$ is measure preserving, one has $T^{-1}(\xi)\in \mathfrak{P}(X)$ and $\mu (T^{-1}(A_i)) = \mu(A_i)$ for all $i=1,\ldots,k$. Hence, $$H(T^{-1}(\xi)) = H(\xi)$$ and so
$\mathfrak{mes}(T):\mathfrak{P}(Y)\to\mathfrak{P}(X)$, defined by $\xi\mapsto T^{-1}(\xi)$, is a morphism in $\SL^\dag$.
Therefore the assignments $X \mapsto\mathfrak{P}(X)$ and $T\mapsto\mathfrak{mes}(T)$ define a contravariant functor 
$$\boxed{\newsym{functor $\mathfrak{mes}:\MS\to\SL^\dag$}{\mathfrak{mes}}:\MS\to\SL^\dag.}$$

\begin{Theorem}\label{realization:mes:ent}
Let $X$ be a measure space and $\psi:X\to X$ a measure preserving transformation. Then $H_{mes}(\psi,\xi)=\H_{\mathfrak{mes}}(\psi,\xi)$ for every $\xi\in\mathfrak P(X)$, and so $$h_{mes}(\psi)=\h_{\mathfrak{mes}}(\psi).$$
\end{Theorem}
\begin{proof}
Let $\xi\in\mathcal P(X)$. For every $n\in\N_+$, since $$\xi\vee \psi^{-1}(\xi) \vee \ldots \vee \psi^{-n+1}(\xi)=T_n(\mathfrak{mes}(\psi),\xi),$$ applying the definitions we can conclude that $$H(\U\vee \psi^{-1}(\U) \vee \ldots \vee \psi^{-n+1}(\U))=c_n(\mathfrak{mes}(\psi),\xi),$$ and so $H_{mes}(\psi,\xi)=h_\Se(\mathfrak{mes}(\psi),\xi)=\H_{\mathfrak{mes}}(\psi,\xi)$.
\end{proof}

The functor $\mathfrak{mes}$ is covariant, and sends subobjects embeddings in $\MS$ to surjective morphisms in $\SL$ and  surjective maps in $\MS$ to embeddings in $\SL$. Hence, similarly to $h_{top}$, also the measure entropy $h_{mes}$ is invariant under conjugation and inversion, it is monotone with respect to taking restrictions to invariant subspaces and factors, it satisfies the Logarithmic Law and it vanishes on quasi-periodic measure preserving transformations. 

It is known that the measure entropy satisfies also the weak Addition Theorem, namely, if $(X,\phi)$ and $(Y,\psi)$ are flows of $\mathbf{Mes}$, then $h_{mes}(\phi\times\psi)=h_{mes}(\phi)+h_{mes}(\psi)$, where $\phi\times\psi:X\times Y\to X\times Y$. But it is not clear whether it is possible to apply directly Lemma~\ref{wATg}, indeed to prove the needed cofinality a preliminary restriction is required (see \cite[Theorem 4.23]{Wa} and its proof).

\subsection{Algebraic entropy}\label{sec:h_alg}

Here we consider the category $\mathbf{Grp}$ of all groups and their homomorphisms and its subcategory $\mathbf{AG}$ of all abelian groups. We construct two functors $\mathfrak{sub}:\mathbf{AG}\to\SL^\dag$ and $\mathfrak{pet}:\mathbf{Grp}\to\mathfrak M^\dag_p$ that permit to find from the general scheme the two algebraic entropies $\ent$ and $h_{alg}$ as functorial entropies.

\medskip
We start recalling the definition of the entropies $\ent$ and $h_{alg}$ following \cite{DG-islam}. For a group $G$ let $\newsym{family of all finite non-empty subsets of a group $G$}{\sM(G)}=\mathcal S(G)\setminus\{\emptyset\}$
be the family of all finite non-empty subsets of $G$ and \newsym{family of all finite subgroups of a group $G$}{$\sF(G)$} be its subfamily of all finite subgroups of $G$. 

\begin{Definition}
Let $\phi:G\to G$ be an endomorphism. For $F\in\mathcal H(G)$, and for $n\in\N_+$, the \emph{$n$-th $\phi$-trajectory} of $F$ is 
\begin{equation*}\label{T_n}
T_n(\phi,F)=F\cdot\phi(F)\cdot\ldots\cdot\phi^{n-1}(F).
\end{equation*}
The \emph{algebraic entropy of $\phi$ with respect to $F\in H(G)$} is 
\begin{equation*}\label{H}
H_{alg}(\phi,F)={\lim_{n\to \infty}\frac{\log|T_n(\phi,F)|}{n}};
\end{equation*}
The \emph{algebraic entropy}\index{algebraic entropy} of $\phi:G\to G$ is $$\newsym{algebraic entropy}{h_{alg}}(\phi)=\sup\{H_{alg}(\phi,F): F\in\mathcal H(G)\},$$
while $$\newsym{algebraic entropy}{\ent}(\phi)=\sup\{H_{alg}(\phi,F): F\in\mathcal F(G)\}.$$
\end{Definition}

If $G$ is abelian, then 
\begin{equation}\label{H---}
\ent(\phi)=\ent(\phi\restriction_{t(G)})= h_{alg}(\phi\restriction_{t(G)}).
\end{equation}
Moreover, $h_{alg}(\phi) = \ent(\phi)$ if $G$ is locally finite, that is every finite subset of $G$ generates a finite subgroup; note that every locally finite group is obviously torsion, while the converse holds true under the hypothesis that the group is solvable (the solution of Burnside's problem shows that even groups of finite exponent fail to be locally finite).

\medskip
For a group $G$, let $v(F) = \log|F|$ for every $F \in \sH(G)$. In case $G$ is abelian one has: 
\begin{itemize}
\item[(i)] $(\sF(G),+,v,\subseteq)$ is a normed semilattice with zero $\{0\}$;
\item[(ii)] $v$ is subadditive, arithmetic, monotone and d-monotone.
\end{itemize}

For every group homomorphism $\f: G \to H$, the map $\sF(\f): \sF(G) \to \sF(H)$, defined by $F\mapsto \f(F)$, is a morphism in $\SL^\dag$.
Therefore, the assignments $G\mapsto \sF(G)$ and $\phi\mapsto \sF(\phi)$ define a covariant functor 
$$\boxed{\newsym{functor $\mathfrak{sub}: \AG \to \SL^\dag$}{\mathfrak{sub}}: \AG \to \SL^\dag.}$$

\begin{Theorem}\label{halg-sub}
Let $G$ be a group and $\phi:G\to G$ an endomorphism. Then $H_{alg}(\phi,F)=\H_{\mathfrak{sub}}(\phi,F)$ for every $F\in\mathcal F(G)$, and so $$\ent(\phi)=\h_{\mathfrak{sub}}(\phi).$$
\end{Theorem}
\begin{proof}
Let $F\in\mathcal F(G)$. Since $T_n(\phi,F)=T_n(\mathcal F(\phi),F)$ for every $n\in\N_+$, applying the definitions we can conclude that $\log|T_n(\phi,F)|=c_n(\mathcal F(\phi),F)$ for every $n\in\N_+$, and so $$H_{alg}(\phi,F)=h_\Se(\mathcal F(\phi), F)=\H_{\mathfrak{sub}}(\phi,F);$$
this concludes the proof.
\end{proof}

Since the covariant functor $\mathfrak{sub}$ takes factors in $\mathbf{AG}$ to surjective morphisms in $\Se$, embeddings in $\mathbf{AG}$ to embeddings in $\Se$, and direct limits in $\mathbf{AG}$ to direct limits in $\Se$, in view of the properties proved in \S\ref{f-sec}, the algebraic entropy $\ent$ is invariant under conjugation and inversion, it is monotone for restrictions to invariant subspaces and for factors, it satisfies the Logarithmic Law, it vanishes on quasi-periodic endomorphisms and it is continuous for direct limits.

The weak Addition Theorem holds as well. Indeed, for an abelian group $G$ and an endomorphism $\phi:G\to G$, the order $\subseteq$ is compatible with $\mathcal F(\phi)$. So, if $\psi:H\to H$ is another group endomorphism and $H$ is abelian, since $\mathcal F(G)\oplus\mathcal F(H)$ is cofinal in $\mathcal F(G\oplus H)$ and the restriction of the norm of $\mathcal F(G\oplus H)$ to $\mathcal F(G)\oplus\mathcal F(H)$ coincides with $v_\oplus$,  Lemma~\ref{wATg} applies and together with Theorem~\ref{halg-sub} gives $$\ent(\phi\oplus\psi)=\ent(\phi)+\ent(\psi).$$

\medskip
For an arbitrary group $G$, one has:
\begin{itemize}
\item[(i)] $(\sM(G),\cdot,v,\subseteq)$ is an ordered normed monoid with neutral element $\{1\}$;
\item[(ii)] $v$ is subadditive, monotone and d-monotone.
\end{itemize}

\begin{Remark}\label{polgrowth}
If $G$ is an abelian group, then $\mathcal H(G)$ is commutative and $v$ is arithmetic since for every $F\in \sM(G)$,
$$|T_n(id_G,F)|\leq (n+1)^{|F|},$$
where clearly $$T_n(id_G,F)= \underbrace{F + \ldots + F}_n$$ (note that the latter inequality extends also to nilpotent groups, see \cite{dlH-ue}).

More precisely, for a group $G$ the following conditions are equivalent:
\begin{enumerate}[(a)]
\item $(\mathcal H(G),v)$ is arithmetic;
\item $G$ has polynomial growth;
\item $G$ is locally virtually nilpotent.
\end{enumerate}
Indeed, we say that $G$ has polynomial growth if for every $F\in\mathcal H(G)$ with $1\in G$, the map $n\mapsto |T_n(id_G,F)|$ is polynomial, so the equivalence (a)$\Leftrightarrow$(b) follows from a straightforward computation. Moreover, (b)$\Leftrightarrow$(c) follows from the celebrated Gromov's theorem, stating that a finitely generated group $G$ has polynomial growth if and only if $G$ is virtually nilpotent. For more details on the connection of the algebraic entropy with the group growth from geometric group theory see \cite{DG-islam,DGB_PC,GBSp,GBSp1}.
\end{Remark}

For every group homomorphism $\f:G \to H$, the map $\sM(\f): \sM(G) \to \sM(H)$, defined by $F\mapsto \f(F)$, is a morphism in $\mathfrak{M}^\dag_p$.
Consequently the assignments $G \mapsto (\sM(G),v)$ and $\phi\mapsto \sM(\phi)$ give a covariant functor 
$$\boxed{\newsym{functor $\mathfrak{pet}:\mathbf{Grp}\to \mathfrak{M}^\dag_p$}{\mathfrak{pet}}:\mathbf{Grp}\to \mathfrak{M}^\dag_p}$$
(the notation $\mathfrak{pet}$ was chosen to honor Justin Peters who inspired the definition of the entropy $h_{alg}$). 
The functor $\mathfrak{sub}$ is a subfunctor of $\mathfrak{pet}$ as $ \sF(G) \subseteq  \sH(G)$ for every abelian group $G$. 

\begin{Theorem}\label{hpet}
Let $G$ be a group and $\phi:G\to G$ an endomorphism. Then $H_{alg}(\phi,F)=\H_{\mathfrak{pet}}(\phi,F)$ for every $F\in\mathcal H(G)$, and so $$h_{alg}(\phi)=\h_{\mathfrak{pet}}(\phi).$$
\end{Theorem}
\begin{proof}
Let $F\in\mathcal H(G)$. Since, for every $n\in\N_+$, $$T_n(\phi,F)=T_n(\mathcal H(\phi),F),$$ applying the definitions we conclude that $$\log|T_n(\phi,F)|=c_n(\mathcal H(\phi),F).$$ So, $H_{alg}(\phi,F)=h_\Se(\mathcal H(\phi), F)=\H_{\mathfrak{pet}}(\phi,F)$.
\end{proof}

As for the algebraic entropy $\ent$, since the covariant functor $\mathfrak{pet}$ takes factors in $\mathbf{Grp}$ to surjective morphisms in $\Se_p^\dag$, embeddings in $\mathbf{Grp}$ to embeddings in $\Se_p^\dag$, and direct limits in $\mathbf{Grp}$ to direct limits in $\Se$, in view of the properties in \S\ref{f-sec} we have automatically that the algebraic entropy $h_{alg}$ is invariant under conjugation, it is monotone for restrictions to invariant subgroups and for quotients, it satisfies the Logarithmic Law in the abelian case (for the general setting see \cite[Proposition 5.1.8]{DG-islam}) and it is continuous for direct limits. 
 The Invariance under inversion follows in the abelian case from Lemma~\ref{inv}, in the general case from Proposition~\ref{antiiso} applied to the inversion $x\buildrel{i}\over\mapsto x^{-1}$; it follows also from a remark in \cite{DGB_PC}.

The weak Addition Theorem holds as well. Indeed, for a group $G$ and an endomorphism $\phi:G\to G$, the order $\subseteq$ is compatible with $\mathcal H(\phi)$. So, if $\psi:H\to H$ is another group endomorphism,  since $\mathcal H(G)\oplus\mathcal H(H)$ is cofinal in $\mathcal H(G\oplus H)$ and the restriction of the norm of $\mathcal H(G\oplus H)$ to $\mathcal H(G)\oplus\mathcal H(H)$ coincides with $v_\oplus$,  Lemma~\ref{wATg} applies and together with Theorem~\ref{hpet} gives $$h_{alg}(\phi\oplus\psi)=h_{alg}(\phi)+h_{alg}(\psi).$$

In view of Remark~\ref{polgrowth} and Lemma~\ref{teo:provvisorio}, the algebraic entropy $h_{alg}$ vanishes on quasi-periodic endomorphisms of locally virtually nilpotent groups. In particular, $h_{alg}$ of the identity automorphism of a locally virtually nilpotent group is zero, while this is no more true in general.  In fact, $h_{alg}(id_G)>0$ precisely when the group $G$ has exponential growth, and in this case $h_{alg}(id_G)=\infty$ by the Logarithmic Law. On the other hand, there exist groups $G$ of intermediate growth, as Grigorchuck's group, namely, groups having growth that is neither polynomial nor exponential (and so $h_{alg}(id_G)=0$ yet $G$ is not locally virtually nilpotent).

\medskip
In the context of the algebraic entropy one considers the right Bernoulli shifts for {\em direct sums}. Let $G$ be a group; the \emph{(one-sided) right Bernoulli shift}\index{Bernoulli shift}\index{one-sided right Benoulli shift}\index{right Bernoulli shift} is
\begin{equation}\label{rightBernoulli1}
\newsym{(one-sided) right Bernoulli shift of a group $G$}{\beta^\oplus_G}:G^{(\N)}\to G^{(\N)},\quad (x_0,x_1,\ldots,x_n,\ldots)\mapsto (1,x_0,x_1, \ldots,x_{n-1},\ldots),
\end{equation}
while the \emph{(two-sided) right Bernoulli shift} is
\begin{equation}\label{rightBernoulli2a}
\newsym{{(two-sided) right Bernoulli shift} of a group $G$}{\bar\beta_G^\oplus}:G^{(\Z)}\to G^{(\Z)},\quad (x_n)_{n\in\Z}\mapsto (x_{n-1})_{n\in\Z}.
\end{equation}
Obviously, the two-sided right Bernoulli shift $\bar\beta_G^\oplus$ coincides with the restriction to the direct sum of the two-sided right Bernoulli shift defined in \eqref{barbetaX} considered in the category $\mathbf{Grp}$.

\begin{Remark}
We recall that, if $G$ is an abelian group and $\widehat G$ is its compact Pontryagin dual group, then $$\widehat{\beta_G^\oplus}={}_{\widehat G}\beta\quad\text{and}\quad\widehat{\bar\beta_G^\oplus}={}_{\widehat G}\bar\beta.$$
\end{Remark}

It is known that, for any group $G$,
\begin{equation}\label{halgbeta}
h_{alg}(\beta_G^\oplus)=h_{alg}(\bar\beta_G^\oplus)=\log|G|,
\end{equation}
with the convention that $\log|G|=\infty$ if $G$ is infinite.

\smallskip
For a finite-to-one map $\lambda:X\to Y$ between two non-empty sets, and a group $K$,  the generalized shift $\sigma_\lambda:K^Y \to K^X$ sends $K^{(Y)}$ to $K^{(X)}$ (as $f \in K^{(X)}$ precisely when $f$ has finite support). We denote 
\begin{equation}\label{^oplus}
\sigma_\lambda^\oplus=\sigma_\lambda\restriction_{K^{(Y)}}.
\end{equation}
In particular, when $Y =X$, then $K^{(X)}$ is a $\sigma_\lambda$-invariant subgroup of $K^X$ precisely when $\lambda$ is finite-to-one. It is known from  \cite[Theorem 7.3.3]{DG-islam} and \cite{AADGH} that, for the contravariant set-theoretic entropy $\mathfrak h_p^*$ defined there (see Remark~\ref{h*tilde}),
\begin{equation}\label{halgh}
h_{alg}(\sigma_\lambda^\oplus)=\mathfrak h_p^*(\lambda)\log|K|,
\end{equation}
with the convention that $\log|K|=\infty$ if $K$ is infinite.  As proved in \cite{AADGH}, the formula in \eqref{halgbeta} can be deduced from \eqref{halgh}. 

Moreover, in the concrete category $\mathbf{AG}$, the forward generalized shift from Definition~\ref{fgs} has the following description.
Let $K$ be an abelian group and $\lambda:X\to Y$ a map. Then 

\begin{equation}\label{fgseq}
\tau_\lambda:K^{(X)}\to K^{(Y)},\quad (k_x)_{x\in X}\mapsto \left(\sum_{x\in\lambda^{-1}(y)}k_x\right)_{y\in Y},
\end{equation}
with the convention that a sum on the empty set is $0$.

\begin{Remark}
Let us see now that the right Bernoulli shifts are forward generalized shifts, in the same way as we have seen above that the left Bernoulli shifts are backward generalized shifts. Indeed, if $\lambda:\N\to\N$ is defined by $n\mapsto n+1$, then $$\tau_\lambda=\beta_K^\oplus$$ is the one-sided right Bernoulli shift recalled in \eqref{rightBernoulli1}. The same occurs for the two-sided right Bernoulli shift $\bar\beta^\oplus_K$ in \eqref{rightBernoulli2a}.
\end{Remark}

It is proved in \cite{DFGB}, in the more general case of amenable semigroup actions, that, for $K$ an abelian group and $\lambda:X\to X$ a selfmap,
$$h_{alg}(\tau_\lambda)=\mathfrak h(\lambda)\cdot\log|K|,$$
with the convention that $\log|K|=\infty$ if $K$ is infinite. This generalizes the formula \eqref{halgbeta} for the right Bernoulli shifts and it will follow from Lemma~\ref{SetBT3} below.

\subsection{Algebraic $i$-entropy}\label{i-sec}

Let $R$ be a ring. Here $i: \mathbf{Mod}_R \to \R_{\geq0}$ is an invariant of $\mathbf{Mod}_R$ (i.e., $i(0)=0$ and $i(M) = i(N)$ whenever $M\cong N$).
We consider the algebraic $i$-entropy introduced in \cite{SZ}, giving a functor ${\mathfrak{sub}_i}:\mathbf{Mod}_R\to \SL^\dag$, to find $\ent_i$ from the general scheme. 

Consider the following conditions: 
\begin{itemize}
\item[(a)] $i(N_1 + N_2)\leq i(N_1) + i(N_2)$ for all submodules $N_1$, $N_2$ of $M$;
\item[(b)] $i(M/N)\leq i(M)$ for every submodule $N$ of $M$;
\item[(b$^*$)] $i(N)\leq i(M)$ for every submodule $N$ of $M$. 
\end{itemize}
The invariant $i$ is called \emph{subadditive} if (a) and (b) hold, \emph{additive} if equality holds in (a), \emph{preadditive} if (a) and (b$^*$) hold.

\smallskip
For a right $R$-module $M$, denote by $\La(M)$ the family of all its submodules and let
$$\sF_i(M)=\{N\in\La(M): i(N)< \infty\}.$$

\begin{Definition}
Let $i$ be a subadditive invariant, $M$ be an $R$-module and $\phi:M\to M$ an endomorphism. 
The \emph{algebraic $i$-entropy of $\phi$ with respect to $N\in\mathcal F_i(M)$} is 
\begin{equation*}
H_i(\phi,N)=\lim_{n\to \infty}\frac{i(T_n(\phi,N))}{n};
\end{equation*}
The \emph{algebraic $i$-entropy}\index{algebraic $i$-entropy} of $\phi$ is $$\newsym{algebraic $i$-entropy}{\ent_i}(\phi)=\sup\{H_i(\phi,N): N\in \mathcal F_i(M)\}.$$
\end{Definition}

For $M\in\mathbf{Mod}_R$, $\La(M)$ is a lattice with the operations of intersection and sum of two submodules, the bottom element is $\{0\}$ and the top element is $M$. For a subadditive invariant $i$ of $\mathbf{Mod}_R$ and for a right $R$-module $M$, letting $v=i$ on $\sF_i(M)$, we have that:
\begin{itemize}
 \item[(i)] $(\sF_i(M),+,i,\subseteq)$ is a normed subsemilattice of $\La(M)$ with zero $\{0\}$;
 \item[(ii)] $v$ is subadditive, since $i$ is subadditive; moreover, $v$ is arithmetic.
\end{itemize}
The norm $v$ is not necessarily monotone. On the other hand,
\begin{itemize}
 \item[(iii)]  if $i$ is both subadditive and preadditive, $v$ is monotone and d-monotone.
\end{itemize}
For every homomorphism $\f: M \to N$ in $\mathbf{Mod}_R$, $\sF_i(\f): \sF_i(M) \to \sF_i(N)$, defined by $\sF_i(\f)(H) =\f(H)$, is a morphism in $\SL^\dag$.
Therefore, the assignments $M \mapsto \sF_i(M)$ and $\phi\mapsto \sF_i(\phi)$ define a covariant functor 
$$\boxed{\newsym{functor $\mathfrak{sub}_i:\mathbf{Mod} _R\to \SL^\dag$}{\mathfrak{sub}_i}:\mathbf{Mod} _R\to \SL^\dag.}$$

\begin{Theorem}\label{subi}
Let $i$ be a subadditive invariant. Let $M$ be a right $R$-module and $\phi:M\to M$ an endomorphism. Then $H_i(\phi,N)=\H_{\mathfrak{sub}_i}(\phi,N)$ for every $N\in\mathcal F_i(M)$, and so $$\ent_i(\phi)=\h_{\mathfrak{sub}_i}(\phi).$$
\end{Theorem}
\begin{proof}
Let $N\in\mathcal F_i(G)$. Since, for every $n\in\N_+$, $$T_n(\phi,N)=T_n(\mathcal F_i(\phi),N),$$ applying the definitions we can conclude that $i(T_n(\phi,N))=c_n(\mathcal F_i(\phi),N).$ So, $$H_i(\phi,N)=h_\Se(\mathcal F_i(\phi), N)=\H_{\mathfrak{sub}_i}(\phi,N),$$
and this concludes the proof.
\end{proof}

By the results in \S\ref{f-sec}, $\ent_i$ is invariant under conjugation and under inversion,  moreover it vanishes on quasi-periodic endomorphisms.

It is known that in general $\ent_i$ is not monotone for quotients.
If $i$ is preadditive, the covariant functor ${\mathfrak{sub}_i}$ sends monomorphisms to embeddings and so $\ent_i$ is monotone for invariant submodules.
If $i$ is both subadditive and preadditive then for every $R$-module $M$ the norm of ${\mathfrak{sub}_i}(M)$ is d-monotone, so $\ent_{i}$ satisfies also the Logarithmic Law. 

Under these assumptions also the weak Addition Theorem holds for the algebraic $i$-entropy. Indeed, for every $R$-module $M$ and every endomorphism $\phi:M\to M$, the order given by the inclusion $\subseteq$ is compatible with $\sF_i(\phi):\sF_i(M)\to \sF_i(M)$. So let $N$ be another $R$-module and $\psi:N\to N$ an endomorphism. Now $\sF_i(M)\oplus \sF_i(N)$ is cofinal in $\sF_i(M\oplus N)$ with respect to the order given by $\subseteq$, and the restriction of the norm of $\sF_i(M\oplus N)$ to $\sF_i(M)\oplus \sF_i(N)$ coincides with $v_\oplus$. Therefore, Lemma~\ref{wATg} gives $$\h_{\mathfrak{sub}_i}(\phi\oplus\psi)=\h_{\mathfrak{sub}_i}(\phi)+\h_{\mathfrak{sub}_i}(\psi),$$ so by Theorem~\ref{subi} we have the weak Addition Theorem for the algebraic $i$-entropy.

\medskip
It was proved in \cite{SVV} that, if $i$ is a length function (i.e., $i$ is upper continuous and additive) and $i$ is discrete (i.e., the set of finite values of $i$ is order-isomorphic to $\N$), then $\ent_i$ is continuous for direct limits. More precisely, under these assumptions, for an $R$-module endomorphism $\phi:M\to M$, we have that $$\ent_i(\phi)=\sup_{F\in\sF^{f}_i(M)}H_i(\phi,F),$$ where $\sF^{f}_i(M)$ is the subsemilattice of $\sF_i(M)$ of all finitely generated submodules of $M$ with finite $i$ (see \cite{SVV}). 
Therefore, if one defines the functor $$\mathfrak{sub}^f_i:\mathbf{Mod} _R\to \SL^\dag$$ by $M \mapsto \sF^f_i(M)$ and $\phi\mapsto \sF^f_i(\phi)$ as above, one obtains $$\ent_i=\h_{\sub_i^f}.$$ 
Now $\sub_i^f$ sends direct limits to direct limits and so $\ent_i$ is continuous for direct limits by Corollary~\ref{contlim} and Theorem~\ref{subi}.

\medskip
A clear example, studied in detail in \cite{GBS}, is given by vector spaces and $i=\dim$ (see \S\ref{V-sec}).

\subsection{Adjoint algebraic entropy}\label{adj-sec}

Again we consider the category $\mathbf{Grp}$ of all groups and their homomorphisms, giving a functor $\mathfrak{sub}^\star:\mathbf{Grp}\to \SL^\dag$ such that the entropy defined using this functor coincides with the adjoint algebraic entropy $\ent^\star$ introduced in \cite{DGS} for abelian groups.  The abelian groups of zero adjoint entropy are studied in \cite{SZ1}.

\medskip
For a group $G$ denote by \newsym{family of all finite index subgroups of a grouo $G$}{$\sC(G)$} the family of all subgroups of finite index in $G$. For an endomorphism $\phi:G\to G$, $N\in \sC(G)$ and $n\in\N_+$, 
defined the \emph{$n$-th $\phi$-cotrajectory of $N$}\index{contrajectory} by
$$C_n(\phi,N)=N\cap\phi^{-1}(N)\cap\ldots\cap\phi^{-n+1}(N).$$ 
Since the map induced by $\phi^n$ on the partitions $\{\phi^{-n}(N)g:g\in G\}\to \{Ng:g\in G\}$ is injective, it follows that $\phi^{-n}(N)\in\sC(G)$ for every $n\in\N$. Therefore, $C_n(\phi,N)\in \sC(G)$ for every $n\in\N_+$, because $\sC(G)$ is closed under finite intersections. 

\begin{Definition}
Let $G$ be a group and $\phi:G\to G$ an endomorphism.
The \emph{adjoint algebraic entropy of $\phi$ with respect to $N\in\sC(G)$} is 
\begin{equation}\label{H*}
H^\star(\phi,N)={\lim_{n\to \infty}\frac{\log[G:C_n(\phi,N)]}{n}}.
\end{equation}
The \emph{adjoint algebraic entropy of $\phi$}\index{adjoint algebraic entropy} is $$\newsym{adjoint algebraic entropy}{\ent^\star}(\phi)=\sup\{H^\star(\phi,N):N\in\sC(G)\}.$$
\end{Definition}

The pair $(\sC(G),\cap)$ is a subsemilattice of $(\La(G), \cap)$. For $N\in\sC(G)$, let $v(N) = \log[G:N]$. Then:
\begin{itemize}
\item[(i)] $(\sC(G),\cap,v, \supseteq)$ is a normed semilattice with zero $G$;
\item[(ii)] $v$ is subadditive, arithmetic, monotone and d-monotone.
\end{itemize}

For every group homomorphism $\f: G \to H$, the map $\sC(\f): \sC(H) \to \sC(G)$, defined by $N\mapsto \f^{-1}(N)$, is a morphism in $\SL^\dag$. 
Then the assignments $G\mapsto\sC(G)$ and $\phi\mapsto\sC(\phi)$ define a contravariant functor 
$$\boxed{\newsym{functor $\mathfrak{sub}^\star:\mathbf{Grp}\to \SL^\dag$}{\mathfrak{sub}^\star}:\mathbf{Grp}\to \SL^\dag.}$$

\begin{Theorem}\label{aent}
Let $G$ be a group and $\phi:G\to G$ an endomorphism. Then $H^\star(\phi,N)=\H_{\mathfrak{sub}^\star}(\phi,N)$ for every $N\in\sC(G)$, and so $$\ent^\star(\phi)=\h_{\mathfrak{sub}^\star}(\phi).$$
\end{Theorem}
\begin{proof}
Let $N\in\sC(G)$. Since, for every $n\in\N_+$, $$C_n(\phi,N)=T_n(\sC(\phi),N),$$ the definitions yield the equality $\log[G:C_n(\phi,N)]=c_n(\sC(\phi),N)$. So, 
$$H^\star(\phi,N)=h_\Se(\sC(\phi),N)=\H_{\mathfrak{sub}^\star}(\phi,N),$$
and this concludes the proof.
\end{proof}

The contravariant functor $\mathfrak{sub}^\star$ need not send the subgroup inclusion $j: H \hookrightarrow G$ to a 
surjective  map $\sC(j):\sC(G) \to \sC(H)$; indeed, take an abelian group $H$ with $\sC(H)\neq\{H\}$ and $G$ its divisible hull (so $\sC(G)=\{G\}$). This fact suggests that the adjoint algebraic entropy is not monotone under taking restrictions to invariant subgroups; indeed, this is the case as shown by a counterexample in \cite[Example 4.8]{DGS}.

Moreover, the adjoint algebraic entropy is not continuous for inverse limits (see \cite{DGS}).

On the other hand, since the contravariant functor $\mathfrak{sub}^\star$ takes factors in $\mathbf{Grp}$ to embeddings in $\mathfrak L$, by the properties proved in \S\ref{f-sec} we have automatically that $\ent^\star$ is invariant under conjugation and inversion, it is monotone for factors, it satisfies the Logarithmic Law and it vanishes on quasi-periodic endomorphisms.

As in the case of the algebraic entropy, the weak Addition Theorem holds for the adjoint algebraic entropy. Indeed, for every group $G$ and every endomorphism $\phi:G\to G$, the order given by the containment $\supseteq$ is compatible with $\sC(\phi):\sC(G)\to \sC(G)$.
So let $\psi:H\to H$ be another group endomorphism. Then $\mathcal C(G)\oplus\mathcal C(H)$ is cofinal in $\mathcal C(G\times H)$ with respect to $\supseteq$ and the norm of $\mathcal C(G\times H)$ restricted to $\mathcal C(G)\oplus\mathcal C(H)$ coincides with $v_\oplus$.
So Lemma~\ref{wATg} applies and together with Theorem~\ref{aent} gives $$\ent^\star(\phi\times\psi)=\ent^\star(\phi)+\ent^\star(\psi).$$

\medskip
There exists also a version of the adjoint algebraic entropy for modules, namely the \emph{adjoint algebraic $i$-entropy}\index{adjoint algebraic $i$-entropy} \newsym{adjoint algebraic $i$-entropy}{$\ent_i^\star$} (see \cite{Vi}), obtained by a preadditive invariant $i:\mathbf{Mod}_R\to \R_{\geq0}$ for a ring $R$.  We omit the detailed outline of this case that can be treated analogously.

\subsection{Topological entropy for linearly topologized precompact groups}\label{lintop-sec}

Let $(G,\tau)$ be a linearly topologized precompact group, i.e., $(G,\tau)$ has a local base $\V_G(1)$ at $1$ consisting of open normal subgroups of $G$ of finite index. Each $V\in \V_G(1)$ defines a finite open cover $\U_V=\{x\cdot V\}_{x\in G}\in\fc(G)$ with $N(\U_V)=[G:V]$. 
Moreover, let $$\newsym{family of the finite open covers $\{\U_V:V \in \V_G(1)\}$ of a linearly topologized precompact group $G$}{\fc_s(G)}=\{\U_V:V \in \V_G(1)\}\subseteq \fc(G).$$
We first prove the following useful claim.

\begin{claim}\label{Claim18March} 
Let $G$ be a linearly topologized precompact group, $\phi:G\to G$ a continuous endomorphism. Then
$\fc_s(G)$ is a $\fc(\phi)$-invariant cofinal subsemigroup of $\fc(G)$.
\end{claim}
\begin{proof}
Take $\mathcal U\in\fc(G)$. Let $K$ be the compact completion of $G$. There exists $\mathcal W\in\fc(K)$ such that $$\mathcal U= \{G \cap W: W\in \mathcal W\}.$$ Since $\mathcal W$ is a finite open cover of $K$ and since $K$ is zero dimensional (being linearly topologized), there is a finite refinement $\mathcal V$ of $\mathcal W$ that is an open partition of $K$. Then each $V\in \mathcal V$ is a clopen set of $K$. Hence, there exists a finite set  $\{N_{i,V}: i\in I_V\}$ of open subgroups of $K$ and 
(finitely many) cosets $x_{i,V}  N_{i,V}$ such that $$V= \bigcup_{i\in I_V}x_{i,V}  N_{i,V}$$ (note that the enumeration $i\mapsto N_{i,V}$ need not be one-to-one). 
Then $$N=\bigcap_{V \in \mathcal V}\bigcap_{i \in I_V} N_{i,V}$$ is still an open subgroup of $K$ as the intersection is finite. 
The finite cover $$\mathcal U_N^K = \{y\cdot N\}_{y\in K}\in\fc(K)$$ of $K$ refines $\mathcal V$. Therefore, for $N\cap G \in \V_G(1)$, 
$\mathcal U_{N\cap G}$ is a refinement of $\mathcal U$.
\end{proof}

The following result was proved in \cite{DG-islam} under the stronger hypothesis that $G$ is a totally disconnected compact group (so, $\hf$ coincides with $h_{top}$). 
  
\begin{Proposition}\label{no-mu}
Let $G$ be a linearly topologized precompact group, $\phi:G\to G$ a continuous endomorphism. Then $$\hf(\phi)=\sup\{H^\star(\phi,U):U\in\V_G(1)\}.$$
\end{Proposition}
\begin{proof}
Each $V\in \V_G(1)$ defines a finite cover $\U_V=\{x\cdot V\}_{x\in G}\in\fc(G)$ with $N(\U_V)=[G:V]$. Clearly, 
\begin{equation}\label{eqApril15}
\f^{-i}(\U_{V})= \U_{\f^{-i}(V)}\ \text{and}\ \U_{V_1}\vee \U_{V_2}=  \U_{V_1\cap V_2}.
\end{equation} 
for $i\in \N$ and $V_1, V_2\in \V_G(1)$. Therefore, 
$$\U_V\vee \phi^{-1}(\U_V) \vee \ldots \vee \phi^{-n+1}(\U_V)=\U_{C_n(\phi,V)}.$$
for every $n\in\N_+$.
This gives
$$N(\U_V\vee \phi^{-1}(\U_V) \vee \ldots \vee \phi^{-n+1}(\U_V))=\log[G:C_n(\phi,V)].$$
Then, for every $V\in\V_G(1)$, $$\Hf(\phi,\U_V)=H^\star(\phi,V).$$
Now it suffices to note that, by Claim~\ref{Claim18March}, Lemma~\ref{mons} and Theorem~\ref{realization:fintop:ent}, we have that 
$$\hf(\phi)=\sup\{\Hf(\phi,\mathcal U_V):V\in \V_G(1)\},$$
and this concludes the proof.
\end{proof}

For a linearly topologized precompact group $G$, let $v(V)=\log [G:V]$ for every $V\in\V_G(1)$. Then:
\begin{itemize}
\item[(i)] $(\V_G(1), \cap,v,\supseteq)$ is a normed semilattice with zero $G$;
\item[(ii)] $v$ is subadditive, arithmetic, monotone and d-monotone.
\end{itemize}
For a continuous homomorphism $\f: G\to H$ between linearly topologized precompact groups, the map $\V_H(1)\to \V_G(1)$, defined by $V \mapsto \f^{-1}(V)$, is a morphism in $\SL^\dag$. Letting also $\sub_o^\star (G)= \V_G(1)$, this defines a contravariant functor 
$$\boxed{\newsym{functor $\mathfrak{sub}_o^\star:\mathbf{LPG}\to\SL^\dag$}{\mathfrak{sub}_o^\star}:\mathbf{LPG}\to\SL^\dag,}$$
which satisfies $\mathfrak{sub}_o^\star= \mathfrak{sub}^\star\circ U$, where $U: \mathbf{LPG}\to \mathbf{Grp}$ is the standard forgetful functor. Therefore the functor $\mathfrak{sub}_o^\star$ has the properties of the functor $\mathfrak{sub}^\star$ mentioned in \S\ref{adj-sec} (see Remark~\ref{Laaast:Remark} for the properties of the entropy $\hf$).

\medskip
It is easy to deduce the following result from Proposition~\ref{no-mu}.

\begin{Theorem}\label{htopsubo}
Let $G$ be a linearly topologized precompact group and $\phi:G\to G$ a continuous endomorphism. Then $$\hf(\phi)=\h_{\mathfrak{sub}_o^\star}(\phi).$$
\end{Theorem}

\begin{Remark}\label{Rem18March}
One can apply Proposition~\ref{no-mu}, Claim~\ref{Claim18March} and Theorem~\ref{htopsubo} to totally disconnected compact groups, since they are known to be linearly topologized by a theorem of van Dantzig. Consider the restriction of the functor $\sub_o^\star$ to the category $\mathbf{TdCG}$. As noticed above, $h_{top}$ coincides with $\hf$ on $\mathbf{TdCG}$.
Therefore, $h_{top}(\phi)=\h_{\sub_o^\star}(\phi)$ for any  continuous endomorphism $\phi:G\to G$ of a totally disconnected compact group $G$.  
\end{Remark}

\begin{Remark}\label{Laaast:Remark}
As mentioned in \S\ref{fr-sec}, $\hf$ is invariant under conjugation and invariant under inversion, moreover it satisfies the Logarithmic Law and it vanishes on quasi-periodic continuous selfmaps when considered on the category $ \Top$. Moreover, due to Theorem~\ref{realization:fintop:ent} monotonicity under taking closed invariant subspaces and factors is also available.
It is not hard to see that the passage to dense invariant subgroups in $ \mathbf{LPG}$ actually preserves the entropy $\hf$, so in conjunction with the above mentioned properties in $ \Top$, 
one has monotonicity under taking arbitrary invariant subgroups in $\mathbf{LPG}$.
\end{Remark}

\subsection{Algebraic and topological entropy for vector spaces}\label{V-sec}

Fix a discrete field $\mathbb K$. The algebraic dimension entropy for discrete $\mathbb K$-vector spaces (see Definition~\ref{Def:vector_space}) 
is a particular case of the $i$-entropy recalled in \S\ref{i-sec} (for the only invariant $i$ available for $\mathbb K$-vector spaces, namely the dimension), and was studied in detail in \cite{GBS}. For every $\mathbb K$-vector space $V$ let $\sF_d(V)$ be the family of all finite-dimensional linear subspaces $N$ of $V$. 

\begin{Definition}\label{Def:vector_space}
Let $V$ be a vector space over $\mathbb K$ and $\phi:V\to V$ an endomorphism. The \emph{algebraic dimension entropy of $\phi$ with respect to $N\in\mathcal F_d(V)$} is 
\begin{equation*}
H_{\dim}(\phi,N)=\lim_{n\to \infty}\frac{\dim T_n(\phi,N)}{n};
\end{equation*}
the \emph{algebraic dimension entropy}\index{algebraic dimension entropy} of $\phi$ is
$$ \newsym{algebraic dimension entropy}{\ent_{\dim}}(\phi)=\sup\{H_{\dim}(\phi,N): N\in \mathcal F_d(V)\}.$$
\end{Definition}

For a $\mathbb K$-vector space $V$ and $v=\dim$ we have that:
\begin{itemize}
 \item[(i)] $(\sF_d(V),+,\dim,\subseteq)$ is a normed semilattice with zero $\{0\}$;
 \item[(ii)] $v$ is subadditive, arithmetic, monotone and d-monotone.
\end{itemize}

For every morphism $\f: V \to W$ in $\mathbf{Mod}_\mathbb K$, the map $\sF_d(\f): \sF_d(V) \to \sF_d(W)$, defined by $N\mapsto\f(N)$, is a morphism in $\SL^\dag$.
Therefore, as a particular case of what is described in \S\ref{i-sec}, the assignments $V \mapsto \sF_d(V)$ and $\phi\mapsto \sF_d(\phi)$ define a covariant functor 
$$\boxed{\newsym{functor $\mathfrak{sub}_{\dim}:\mathbf{Mod}_\mathbb K\to \SL^\dag$}{\mathfrak{sub}_{\dim}}:\mathbf{Mod}_\mathbb K\to \SL^\dag.}$$
Then, as a particular case of Theorem~\ref{subi}, we have that 
\begin{equation}\label{entdimeq}
\h_{\mathfrak{sub}_{\dim}}=\ent_{\dim}.
\end{equation}

 As described in detail in \S\ref{i-sec}, it can be deduced by the results in \S\ref{f-sec} that $\ent_{\dim}$ is invariant under conjugation and under inversion, monotone under taking quotients and invariant linear subspaces, satisfies the Logarithmic Law and the weak Addition Theorem, moreover it is continuous for direct limits. This entropy can be computed also as follows.

\begin{Remark}
Every flow $\phi: V \to V$ of $\mathbf{Mod}_\mathbb K$ can be considered as a $\mathbb K[X]$-module $V_\phi$ letting $X$ act on $V$ as $\phi$. Then $\ent_{\dim}(\phi)$ coincides with the rank of the $\mathbb K[X]$-module $V_\phi$. 
\end{Remark}

In Lefschetz duality (see \cite{Lef}), the dual vector space of a discrete vector space is linearly compact. 
Recall that, given a linearly topologized vector space $V$ over the discrete field $\mathbb K$ (i.e., $V$ has a local base at $0$ consisting of linear subspaces), a \emph{linear variety} $M$ of $V$ is a coset $v+W$, where $v\in V$ and $W$ is a linear subspace of $V$. A linear variety $M=v+W$ is said to be \emph{open} (respectively, \emph{closed}) in $V$ if $W$ is open (respectively, closed) in $V$. 
A linearly topologized vector space $V$ over $\mathbb K$ is \emph{linearly compact} if any collection of closed linear varieties of $V$ with the finite intersection property has non-empty intersection (equivalently, any collection of open linear varieties of $V$ with the finite intersection property has non-empty intersection) (see \cite{Lef}).

In \cite{CGBtop} the topological counterpart of the algebraic dimension entropy is introduced and studied for linearly compact vector spaces.
For a linearly compact vector space $V$ denote by $\sF_o(V)$ the family of all open linear subspaces $U$ of $V$ (since $V/U$ is discrete and linearly compact, we have that $V/U$ has finite dimension).  

\begin{Definition}\label{LAAAAst_def}
Let $V$ be a linearly compact vector space over $\mathbb K$ and $\phi:V\to V$ a continuous endomorphism.
The \emph{topological dimension entropy of $\phi$ with respect to $N\in\sF_o(G)$} is
$$H_{\dim}^\star(\phi,N)={\lim_{n\to \infty}\frac{\dim(V/C_n(\phi,N))}{n}}.$$
The \emph{topological dimension entropy of $\phi$}\index{topological dimension entropy} is $$\newsym{topological dimension entropy}{\ent_{\dim}^\star}(\phi)=\sup\{H_{\dim}^\star(\phi,N):N\in\sF_o(G)\}.$$
\end{Definition}

The pair $(\sF_o(V),\cap)$ is a semilattice. For $N\in\sF_o(V)$, let $v(N)=\dim(V/N)$. Then:
\begin{itemize}
\item[(i)] $(\sF(V),\cap,v, \supseteq)$ is a normed semilattice with zero $V$;
\item[(ii)] $v$ is subadditive, arithmetic, monotone and d-monotone.
\end{itemize}

 Let $\mathbf{LCVect}_\mathbb K$ be the category of all linearly compact vector spaces over the discrete field $\mathbb K$.
For every continuous homomorphism $\f: V \to W$, the map $\sF_o(\f): \sF_o(W) \to \sF_o(V)$, defined by $N\mapsto \f^{-1}(N)$, is a morphism in $\SL^\dag$. 
Then the assignments $V\mapsto\sF_o(V)$ and $\phi\mapsto\sF_o(\phi)$ define a contravariant functor 
$$\boxed{\newsym{functor $\mathfrak{sub}_{\dim}^\star:\mathbf{LCVect_\mathbb K}\to \SL^\dag$}{\mathfrak{sub}_{\dim}^\star}:\mathbf{LCVect_\mathbb K}\to \SL^\dag.}$$

Analogously to Theorem~\ref{aent}, it is possible to prove that, if $V$ is a linearly compact vector spaces over $\mathbb K$ and $\phi:V\to V$ is a continuous endomorphism, then 
$$H_{\dim}^\star(\phi,N)=\H_{\mathfrak{sub}_{\dim}^\star}(\phi,N)$$ for every $N\in\sF_o(V)$, and so 
\begin{equation}\label{entdim*eq}
\ent_{\dim}^\star(\phi)=\h_{\mathfrak{sub}_{\dim}^\star}(\phi).
\end{equation}

As described in \S\ref{adj-sec}, it can be deduced by the results in \S\ref{f-sec}, that $\ent^\star_{\dim}$ is invariant under conjugation and under inversion, monotone under taking quotients, satisfies the Logarithmic Law and the weak Addition Theorem, and it vanishes on quasi-periodic continuous endomorphisms.  
Since the functor $\mathfrak{sub}_{\dim}^\star$ sends inclusions to quotient maps, it can be deduced from the results in \S\ref{f-sec} that $\ent^\star_{\dim}$ is also monotone under taking closed invariant linear subspaces and it is continuous for inverse limits (for the latter property one can use the same argument as in the case of 
compact topological space, see the text preceeding Proposition~\ref{New:corollary2}). 

\begin{Remark}\label{llc}
In \cite{CGBalg} and \cite{CGBtop}, the algebraic and the topological dimension entropy are studied in the more general case of continuous endomorphisms $\phi:V\to V$ of locally linearly compact vector spaces $V$; a linearly topologized vector space $V$ over a discrete field $\mathbb K$ is \emph{locally linearly compact} if it admits a local base at $0$ consisting of linearly compact open linear subspaces (see \cite{Lef}).

It is easy to see that for a flow $(V,\phi)$ of $\mathbf{Mod}_\mathbb K$ and $N\in\mathcal F_d(V)$, one has $$H_{\dim}(\phi,N)=\lim_{n\to \infty}\frac{\dim(T_n(\phi,N)/N)}{n}$$ 
for the algebraic dimension entropy recalled in Definition~\ref{Def:vector_space}.
In \cite{CGBalg} this limit is used as a definition of the algebraic dimension entropy of a flow $(V,\phi)$ with $V$ a locally linearly compact vector space and $N$ an open linearly compact linear subspace of $V$.

Analogously, for a flow $(V,\phi)$ of $\mathbf{LCVect_\mathbb K}$ and $N\in\sF_o(V)$, one can obtain the topological dimension entropy recalled in Definition \ref{LAAAAst_def} as 
$$H_{\dim}^\star(\phi,N)=\lim_{n\to \infty}\frac{\dim(N/C_n(\phi,N))}{n}.$$ 
In \cite{CGBtop} this limit is used as a definition of the topological dimension entropy of a flow $(V,\phi)$ with $V$ a locally linearly compact vector space $V$ and $N$ an open linearly compact linear subspace of $V$.

As it occurs for the intrinsic algebraic entropy, the algebraic and the topological dimension entropy for locally linearly compact vector spaces do not fit our general scheme. 
\end{Remark}

\section{Bridge Theorems}\label{BT}

In this section we study in detail the notion of Bridge Theorem for functorial entropies, pointing out various situations where such a phenomenon arises.

We start by introducing appropriate weaker forms of isomorphisms in $\Se^*$, in order to have a general scheme that covers all possible Bridge Theorems, including the inspiring one (namely, Theorem~\ref{BT1} above for torsion abelian groups and totally disconnected compact abelian groups; see also Theorems~\ref{WeissBT} and~\ref{WBT} below).

\subsection{Uniform and weak isomorphisms in $\Se^*$}\label{Se*}

The property of Invariance under conjugation in $\Se$ ensures the invariance of the  semigroup entropy for isomorphic flows of $\Se$, but isomorphisms in $\Se$ that are not easy to come by. Our aim in this subsection is to replace isomorphisms in $\Se$ by (appropriate) isomorphisms in $\Se^*$. 

The isomorphisms $\phi: S \to S$ in $\Se$ must be norm-preserving, i.e., $v(\phi(x)) = v(x)$ for all $x\in S$. Since this property is not available in  $\Se^*$, we introduce the following notion providing a partial remedy to this problem. 

\begin{Definition} 
A homomorphism $\alpha:(S,v)\to (S', v')$ in $\Se^*$ is a \emph{uniform isomorphism with coefficient}\index{uniform isomorphism}  $0<r\in\R$ if $\alpha$ is a semigroup isomorphism and $v'(\alpha(x)) = rv(x)$ for every $x\in S$.
\end{Definition}

If $\alpha: S \to S'$ is a uniform isomorphism in $\Se^*$ with coefficient $r$, then $\alpha^{-1}: S' \to S$ is a uniform isomorphism as well and has coefficient $r^{-1}$. 

\begin{Lemma}\label{unif1}
A uniform isomorphism in $\Se^*$ with coefficient $r$ is an isomorphism in $\Se$ if and only if $r=1$.
\end{Lemma}

For a normed semigroup $S$ and a normed preordered semigroup $T$, we give a weaker notion than that of uniform isomorphism:

\begin{Definition}
A homomorphism $\alpha: S \to T$ in $\Se^*$, where $T$ is a normed preordered semigroup, is a \emph{weak isomorphism with coefficient}\index{weak isomorphism} $0<r\in\R$ if $\alpha:S\to \alpha(S)$ a uniform isomorphism with coefficient $r$ and $\alpha(S)$ is a cofinal subsemigroup of $T$.
\end{Definition}

Every semigroup admits the discrete preorder, in this case ``cofinal" in the above definition means equal.

\medskip
A uniform isomorphism is always a semigroup isomorphism, whereas a weak isomorphism is a semigroup isomorphism if and only if it is a uniform isomorphism (with the same coefficient).

\medskip
Easy examples show that the Invariance under conjugation, which holds in $\Se$, drastically fails in $\Se^*$. Nevertheless, one can obtain a sharp substitute in $\Se^*$ as follows.

\begin{Proposition}\label{uInvariance}
Let $\f: S \to S$ be an endomorphism in $\Se$. If $\alpha:S\to T$ is a uniform isomorphism in $\Se^*$ with coefficient $r>0$, then: 
\begin{enumerate}[(a)]
     \item $\psi = \alpha\circ\f\circ\alpha^{-1} : T \to T$ is an endomorphism in $\Se$;
     \item $h_{\Se}(\psi)= r h_{\Se}(\f)$. 
\end{enumerate}
\end{Proposition}
\begin{proof}
(a) Follows from the definition of uniform isomorphism and contractive endomorphism. 

(b) Fix $y\in T$ and find $x \in S$ with $y= \alpha(x)$. Then $c_n(\psi,y)=r c_n(\phi, x)$ for every $n\in\N_+$, and so $h_{\Se}(\psi,y) = r\cdot h_{\Se}(\phi,x)$. So $h_{\Se}(\psi,y)\leq r\cdot  h_{\Se}(\phi,x)$. Therefore, 
$$h_{\Se}(\psi) \leq r\cdot h_{\Se}(\phi).$$
Analogously, using $\alpha^{-1}$ in place of $\alpha$, one can prove that 
$$h_{\Se}(\phi) \leq r^{-1}\cdot h_{\Se}(\psi).$$ 
This concludes the proof.
\end{proof}

A weaker form of the Invariance under conjugation property remains true also considering weak isomorphisms in $\Se^*$:

\begin{Corollary}\label{mons*}
Let $\f: S \to S$ be an endomorphism in $\Se$. Let $T$ be a normed semigroup with a preorder compatible with the endomorphism $\psi:T\to T$ in $\Se$. 
If $\alpha: S \to T$ is a weak isomorphism with coefficient $0<r\in\R$ and such that $\psi\circ\alpha=\alpha\circ\phi$, then:
\begin{enumerate}[(a)]
\item $\alpha(S)$ is $\psi$-invariant;
\item $h_{\Se}(\psi)= r h_{\Se}(\phi).$
\end{enumerate}
\end{Corollary}
\begin{proof}
(a) Follows immediately from the hypothesis.

(b) By Lemma~\ref{mons}, $h_\Se(\psi)=h_\Se(\psi\restriction_{\alpha(S)})$. Moreover, $h_\Se(\psi\restriction_{\alpha(S)})=r\cdot h_\Se(\phi)$ by Proposition~\ref{uInvariance}.
\end{proof}

We will need the following property.

\begin{Lemma}\label{circwi}
Let $S\in\Se$ and $S',S''\in\Se_p$. If $\alpha:S\to S'$ is a weak isomorphism with coefficient $r>0$ and $\beta:S'\to S''$ is a monotone weak isomorphism with coefficient $s>0$, then $\beta\circ\alpha$ is a weak isomorphism with coefficient $rs$.
\end{Lemma}
\begin{proof}
Clearly, $\beta\circ\alpha:S\to\beta(\alpha(S))$ is a uniform isomorphism with coefficient $rs$. That $\beta(\alpha(S))$ is cofinal in $S''$ follows from the monotonicity of $\beta$, the cofinality of $\alpha(S)$ in $S'$ and the cofinality of $\beta(S')$ in $S''$.
\end{proof}

\subsection{General scheme}\label{BTsec}

The Bridge Theorem  from Definition~\ref{BTdef} is a property of a functor $\varepsilon: \mathfrak X_1 \to \mathfrak X_2$ (with respect to entropies $h_1:\mathfrak X_1 \to \R_+$ and $h_2:\mathfrak X_2 \to \R_+$). It is natural to expect that it is invariant under natural equivalence of functors, as we observe now.

\begin{Remark}\label{BTne}
Let $\varepsilon,\varepsilon': \mathfrak X_1 \to \mathfrak X_2$ be naturally equivalent functors and let $h_1:\mathfrak X_1 \to \R_+$ and $h_2:\mathfrak X_2 \to \R_+$ be entropies. For $C>0$, the pair $(h_1, h_2)$ satisfies $BT_{\varepsilon,C}$ if and only if $(h_1, h_2)$ satisfies $BT_{\varepsilon',C}$.
\end{Remark}

Throughout this section $\varepsilon: \mathfrak X_1\to \mathfrak X_2$ is a functor and $\h_{F_1}$, $\h_{F_2}$ are functorial entropies arising from functors $F_1:\mathfrak X_1\to \Se_p$ and $F_2:\mathfrak X_2\to \Se_p$ such that $F_1$ and $F_2\varepsilon$ are simultaneously covariant or contravariant. Our aim is to verify whether the pair $(\h_{F_1},\h_{F_2})$ satisfies the Bridge Theorem. To do this, it seems natural to consider the following notion which yields a higher level of connection between the entropies $\h_{F_1}$, $\h_{F_2}$ than the one given by the Bridge Theorem, as we shall see in Theorem~\ref{MainBT}.

\begin{Definition}\label{SBT}
The pair $(\h_{F_1}, \h_{F_2})$ satisfies the \emph{Strong Bridge Theorem   with respect to $\varepsilon$ with coefficient}\index{Strong Bridge Theorem} $0<C\in\R$ (briefly, \newsym{Strong Bridge Theorem with respect to the functor $\varepsilon$ with coefficient $C>0$}{$SBT_{\varepsilon,C}$}) if there exists a natural transformation $$\eta:  F_1 \to F_2 \varepsilon$$ such that $\eta_X: F_1(X) \to F_2 \varepsilon (X)$ is a uniform isomorphism with coefficient $C$.

If each $\eta_X: F_1(X) \to F_2 \varepsilon (X)$ is only a weak isomorphism with coefficient $C$, we say that $(\h_{F_1}, \h_{F_2})$ satisfies 
the \emph{Strong$^w$ Bridge Theorem with coefficient $C$} (briefly, \newsym{Strong$^w$ Bridge Theorem with respect to the functor $\varepsilon$ with coefficient $C>0$}{$SBT^w_{\varepsilon,C}$}).

In case $C=1$ we write simply \newsym{Strong Bridge Theorem with respect to the functor $\varepsilon$ with coefficient $1$}{$SBT_\varepsilon$} (respectively, \newsym{Strong$^w$ Bridge Theorem with respect to the functor $\varepsilon$ with coefficient $1$}{$SBT^w_\varepsilon$}) and say that  $(\h_{F_1}, \h_{F_2})$ satisfies the \emph{Strong Bridge Theorem}
(respectively,  the \emph{Strong$^w$ Bridge Theorem}).
\end{Definition}

As the next diagram shows, this definition provides two levels of a sort of ``functorial Bridge Theorem", that connects both functors $F_1, F_2$.
\begin{equation}\label{A}
\xymatrix@R=5pt@C=40pt
{\mathfrak{X}_1\ar[dddd]_{\varepsilon}\ar[ddrr]^{F_1}\ar@/^1pc/[rrrdd]^{h_{1}=\h_{F_1}} & & & \\
& \ar[dd]^\eta & & \\
			&	& {\Se_p}\ar[r]^{ {h_\Se}}   & \R_+\\
 & & &			\\
\mathfrak{X}_2\ar[uurr]_{F_2}\ar@/_1pc/[rrruu]_{h_{2}=\h_{F_2}} &	& &
}
\end{equation}

In the next theorem we show that, roughly speaking, 
$$SBT_{\varepsilon,C}\ \Rightarrow\ SBT^w_{\varepsilon,C}\ \Rightarrow\ BT_{\varepsilon,C},$$
in particular, both the Strong Bridge Theorem and the Strong$^w$ Bridge Theorem yield the Bridge Theorem.

\begin{Theorem}\label{MainBT} In the above notation, we have that:
\begin{enumerate}[(a)]
\item if $(\h_{F_1},\h_{F_2})$ satisfies $SBT_{\varepsilon,C}$ then $(\h_{F_1},\h_{F_2})$ satisfies $SBT^w_{\varepsilon,C}$;
\item if $(\h_{F_1},\h_{F_2})$ satisfies $SBT^w_{\varepsilon,C}$ then $(\h_{F_1},\h_{F_2})$ satisfies $BT_{\varepsilon,C}$.
\end{enumerate}
In particular, if the pair $(\h_{F_1},\h_{F_2})$ satisfies $SBT_{\varepsilon}$ then $(\h_{F_1},\h_{F_2})$ satisfies $SBT^w_{\varepsilon}$, and if $(\h_{F_1},\h_{F_2})$ satisfies $SBT^w_{\varepsilon}$ then $(\h_{F_1},\h_{F_2})$ satisfies $BT_{\varepsilon}$.
\end{Theorem}
\begin{proof} 
(a) Follows from the fact that a uniform isomorphism is also a weak isomorphism.

(b) If $(\h_{F_1},\h_{F_2})$ satisfies $SBT_{\varepsilon,C}$, in view of Corollary~\ref{mons*} we have that 
$$h_\Se (F_2 \varepsilon (\phi))=C\cdot  h_\Se (F_1(\phi))$$ 
for every endomorphism $\phi$ in $\mathfrak X_1$. This means that  $(\h_{F_1},\h_{F_2})$ satisfies $BT_{\varepsilon,C}$.
\end{proof}

\begin{Remark}
In the above notation, even when the target of the functors $F_1$ and $F_2$ is the category $\Se$, the existence of a natural transformation 
$\eta:  F_1 \to F_2 \varepsilon$ in $\Se^*$ need not ensure that this natural transformation is in $\Se$ (i.e., the morphism $\eta_X: F_1(X) \to F_2 \varepsilon (X)$ ensured by $SBT_{\varepsilon,C}$ is an isomorphism in $\Se^*$, but need not be an  isomorphism in $\Se$). According to Lemma~\ref{unif1},
$\eta_X: F_1(X) \to F_2 \varepsilon (X)$ is an isomorphism in $\Se$  if and only if $C=1$ (i.e., $SBT_{\varepsilon}$ holds).  
\end{Remark}

In the specific Bridge Theorems stated below, (at least) $SBT^w_\varepsilon$ is verified in many cases, but actually $SBT_\varepsilon$ is available in most of those cases. 

Our choice to use $\mathfrak S_p$ as a target is motivated by the fact that the definition of $SBT^w_{\varepsilon, C}$ makes recourse to a preorder on the semigroups (whereas $SBT^w_{\varepsilon, C}$ perfectly works also for functors $F_1:\mathfrak X_1\to \Se$ and $F_2:\mathfrak X_2\to \Se$).  
Let us point out that this choice is quite painless since in all cases considered in \S\ref{known-sec} the functors have $\Se_p$ as a target with the exception of $\str$; more precisely, with the exception of $\pet$, the targets are the subcategories $\PSL^\dag$ and $\SL^\dag$ of $\Se_p$.

\begin{Remark}\label{remark**}
Assume that $h_1:\mathfrak X_1\to\Se_p$ and $h_2:\mathfrak X_2\to \Se_p$ are entropies and that the functor $\varepsilon:\mathfrak X_1\to\mathfrak X_2$ is invertible.
\begin{itemize}
\item[(a)] Then $(h_1,h_2)$ satisfies $BT_{\varepsilon,C}$ with $0<C\in\R$ if and only if $(h_2,h_1)$ satisfies $BT_{\varepsilon^{-1},C^{-1}}$.
\item[(b)] Unlike the Strong$^w$ Bridge Theorem $SBT_{\varepsilon,C}^w$, the Strong Bridge Theorem $SBT_{\varepsilon,C}$ can be ``inverted" in the following more precise sense. 
If $h_1=\h_{F_1}$ and $h_2=\h_{F_2}$ for functors $F_1:  \mathfrak{X}_1 \to \Se$ and $F_2:\mathfrak X_2\to \Se$, then $(\h_{F_1},\h_{F_2})$ satisfies $SBT_{\varepsilon,C}$ if and only if $(\h_{F_2},\h_{F_1})$ satisfies $SBT_{\varepsilon^{-1},C^{-1}}$.
\end{itemize}
\end{Remark}

The following facts are trivial, yet it is worth  noting the preservation of entropy in these cases. Moreover, we see in \S\ref{forg} an interesting occurrence of the case of the forgetful functor and one of the identity functor.

\begin{Example}
Let $\varepsilon: \mathfrak X_1 \to \mathfrak X_2$ be a functor and $h:\mathfrak X_2 \to \R_+$ an entropy. Assume that $\varepsilon$ is either an inclusion functor or a forgetful functor. In the first case one can consider the restriction $h:\mathfrak X_1 \to \R_+$ keeping the same notation {\em par abus de language}. Clearly, the pair $(h,h)$ satisfies $BT_\varepsilon$. Moreover, if $h=\h_F$ for some functor $F:\mathfrak X_2\to \Se$, then $(\h_{F\varepsilon},\h_F)$ satisfies $SBT_\varepsilon$. All this applies also when $\varepsilon$ is a forgetful functor.
\end{Example}

\begin{Lemma}\label{SBTid}
In the above notation, let $\mathfrak X=\mathfrak X_1=\mathfrak X_2$ and $\varepsilon=id_{\mathfrak X}$. 
\begin{enumerate}[(a)]
\item If $F_1$ and $F_2$ are naturally equivalent, then $(\h_{F_1},\h_{F_2})$ satisfies $SBT_{id_\mathfrak X}$.
\item If $(\h_{F_1},\h_{F_2})$ satisfies $SBT^w_{id_\mathfrak X}$, then $\h_{F_1}=\h_{F_2}$. 
\end{enumerate}
\end{Lemma}
\begin{proof}
(a) is obvious and (b) follows directly from Corollary~\ref{mons*}.
\end{proof}

In particular, given an entropy function $h:\mathfrak X\to \R_+$, if $h=\h_{F_1}$ for some functor $F_1:\mathfrak X\to \Se$ and $(\h_{F_1},\h_{F_2})$ satisfies $SBT^w_{id_\mathfrak X}$ for some other functor $F_2:\mathfrak X\to\Se$, then $h=\h_{F_2}$, and so one has an alternative description of $h$ as a functorial entropy.
A relevant example to this effect is the following.

\begin{Example}\label{remex}
\begin{enumerate}[(a)]
\item Consider the functors $\fc:\CT\to\PSL^\dag$ and $\cov:\CT\to\PSL^\dag$. In view of Remark~\ref{fcrem}, the pair $(\h_{\fc},\h_{\cov})$ satisfies $SBT^w_{id_{\CT}}$. In particular, if $\phi:X\to X$ is a morphism in $\CT$, then $\hf(\phi)=h_{top}(\phi)$ by Lemma~\ref{SBTid}(b).
\item Consider $\mathbf{TdCG}$ and the functors 
$$\cov:\mathbf{TdCG}\to \PSL^\dag\quad\text{and}\quad\sub_o^\star:\mathbf{TdCG}\to \PSL^\dag.$$
Let $\phi:K\to K$ be a morphism in $\mathbf{TdCG}$. By Theorem~\ref{realization:top:ent}, we have that $h_{top}(\phi)=\h_{\cov}(\phi)$.  We see that
\begin{center}
the pair $(\h_{\sub_o^\star},\h_{\cov})$ satisfies $SBT^w_{id_{\mathbf{TdCG}}}$,
\end{center} 
and hence $h_{top}(\phi)=\h_{\sub^\star_o}(\phi)$ by Lemma~\ref{SBTid}(b).

To verify that $(\h_{\sub_o^\star},\h_{\cov})$ satisfies $SBT^w_{id_{\mathbf{TdCG}}}$, let $$\eta_K:\V_K(1)\to\cov(K),\ V\mapsto\U_V,$$ where we recall that $\U_V=\{xV:x\in K\}$. Hence, $\eta_K(\V_K(1))=\fc_s(K)$. By Claim~\ref{Claim18March}, $\fc_s(K)$ is cofinal in $\fc(K)$, while $\fc(K)$ is cofinal in $\cov(K)$ since $K$ is compact (see Remark~\ref{fcrem}); therefore, $\eta_K(\V_K(1))$ is cofinal in $\cov(K)$. Moreover, $\eta_K:\V_K(1)\to\fc_s(K)$ is an isomorphism in $\Se$ since $$v(V)=\log[K:V]=\log N(\U_V).$$
Finally, $\eta:\sub_o^\star\to \cov$ is a natural transformation, since $\phi^{-1}(\U_V)=\U_{\phi^{-1}(V)}$ for every $V\in\V_K(1)$.
\end{enumerate}
\end{Example}

 The next result extends the observation in Remark~\ref{BTne} to Strong Bridge Theorems.

\begin{Lemma}\label{neSBT} 
Let $\varepsilon,\varepsilon': \mathfrak X_1 \to \mathfrak X_2$ be naturally equivalent functors and $C>0$. Then:
\begin{enumerate}[(a)]
\item $(\h_{F_1},\h_{F_2})$ satisfies $SBT^w_{\varepsilon,C}$ if and only if $(\h_{F_1},\h_{F_2})$ satisfies $SBT^w_{\varepsilon',C}$;
\item $(\h_{F_1},\h_{F_2})$ satisfies $SBT_{\varepsilon,C}$ if and only if $(\h_{F_1},\h_{F_2})$ satisfies $SBT_{\varepsilon',C}$.
\end{enumerate}
\end{Lemma}
\begin{proof}
To prove (a) apply Lemma~\ref{circwi}, while (b) follows from the definitions.
\end{proof}

\subsection{Preservation of entropy along (co)reflections and forgetful functors}\label{forg}

In Lemma~\ref{SBTid} we have seen that two distinct realizations of an entropy $\mathfrak X\to \R$ as a functorial entropy can be seen as a Strong Bridge Theorem with respect to the identity functor. In this section we push further this line with respect to (co)reflections and forgetful functors. 

\medskip
In the next remark we briefly discuss the connection between measure entropy and topological entropy.

\begin{Remark}
If $X$ is a compact metric space and $\phi: X \to X$ is a continuous surjective selfmap, by Krylov-Bogolyubov Theorem \cite{BK} there exists some $\phi$-invariant Borel probability measure $\mu$ on $X$ (i.e., making $\phi:(X,\mu) \to (X,\mu)$ measure preserving). Denote by $h_\mu$ the measure entropy with respect to $\mu$.
The inequality $h_{\mu}(\phi)\leq h_{top}(\phi)$ for every $\mu$ is due to Goodwyn \cite{Goo}. Moreover, the Variational Principle (see \cite{Goodman} and \cite[Theorem 8.6]{Wa}) gives the ultimate connection between these two entropies:
$$h_{top}(\phi)=  \sup \{h_{\mu}(\phi):  \mu\ \text{$\phi$-invariant measure on $X$}\}.$$
\end{Remark}

In the case of a compact group $K$ and a continuous surjective endomorphism $\phi:K\to K$, the uniqueness of the Haar measure of $K$ implies that $\phi$ is measure preserving, as noted by Halmos \cite{Halmos}. In particular, both $h_{top}$ and $h_{mes}$ are available for surjective continuous endomorphisms of compact groups, and they coincide as proved by Stoyanov \cite{S} (see \cite{B} for metrizable compact groups).

Denote by $\newsym{category of all compact Hausdorff groups and continuous homomorphisms}{\mathbf{CG}}$ the category of all compact groups and continuous homomorphisms, and by $\newsym{subcategory of $\textbf{CG}$ with morphisms all continuous surjective homomorphisms}{\textbf{CG}_e}$ the non-full subcategory of $\textbf{CG}$ having as morphisms all  continuous surjective homomorphisms in $\textbf{CG}$. Then the above fact can be stated as a Bridge Theorem as follows: 

\begin{Theorem}
Consider the forgetful functor $V: \mathbf{CG}_e\to \mathbf{Mes}$. The pair $(h_{top},h_{mes})$ satisfies $BT_V$.
\end{Theorem}

We are not aware if this Bridge Theorem holds true at a ``higher level", namely as a Strong Bridge Theorem: 

\begin{question}
Does the pair $(\h_\cov,\h_{\mathfrak{mes}})$ satisfy $SBT_V$?
\end{question}

One can see the interaction between the topological entropy and the frame entropy as a Strong Bridge Theorem. Indeed, to every topological space $X$ corresponds the frame $(\ms{O}(X),\cup,\cap)$ of open sets of $X$, with top element $X$ and bottom element $\emptyset$. Every continuous selfmap $\f: X\to X$ gives rise to a frame endomorphism $\ms{O}(\f): \ms{O}(X) \to \ms{O}(X)$ defined by $\ms{O}(\f)(U) = \f^{-1}(U)$ for $U \in \ms{O}(X)$. This gives a contravariant functor $$\ms{O}: \Top \to \Fr.$$
Moreover, every (finite) open cover $\U$ of $X$ gives rise to a (finite) cover of $\mathcal O(X)$ and $\ms{O}(\f): \ms{O}(X) \to \ms{O}(X)$  takes (finite) covers of $\ms{O}(X)$ to (finite) covers of $\ms{O}(X)$. Together with Theorem~\ref{realization:fin-top:ent} and \eqref{realization:hfr}, this proves the following
 
\begin{Theorem}\label{BTframe}
Consider the functor $\ms{O}: \Top \to \Fr$. The pair $(\h_{\fc},\h_{\fc_{fr}})$ satisfies $SBT_\ms{O}$.
In particular, $(\hf,h_{fr})$ satisfies $BT_\ms{O}$.
\end{Theorem}

Consider the $T_0$-reflection $$r:\mathbf{Top}\to \mathbf{Top}_0,$$ where $\Top_0$ is the full subcategory of $\Top$ of $T_0$ spaces. 
If $rX$ is the $T_0$-reflection of a topological space $X$, then for every continuous selfmap $\f: X \to X$ the reflection $\overline\f: rX \to rX$ in $\mathbf{Top}_0$ has the same topological entropy as $\f$. We want to present this result from \cite{DK} as a Strong Bridge Theorem as follows.

\begin{Theorem}\label{reflection}
Consider $r:\mathbf{Top}\to \mathbf{Top}_0$. The pair $(\h_{\fc},\h_{\fc})$ satisfies $SBT_r$. In particular, $(\hf,\hf)$ satisfies $BT_r$.
\end{Theorem}
\begin{proof}
Let $\phi:X\to X$ in $\Top$. The reflection $r$ assigns to $X$ a surjective continuos map $r_X:X\to rX$. Then $\fc(r_X):\fc(rX)\to \fc(X)$ is an isomorphism in $\Se$ such that $\fc(r\phi)$ and $\fc(\phi)$ are conjugate by $\fc(r_X)^{-1}$. So it suffices to take $\eta_X=\fc(r_X)^{-1}$.
For the second part, $\hf=\h_{\fc}$ by Theorem~\ref{realization:fin-top:ent}.
\end{proof}

This theorem reduces the study of $\hf$ to the category of $T_0$-spaces. 
In contrast with Theorem~\ref{reflection}, one can show that the reflection $\Top \to \Top_1$, where $\newsym{category of $T_1$ topological spaces and continuous maps}{\Top_1}$ denotes the (full) subcategory of $\Top$ of $T_1$ topological spaces, strongly fails to preserve the topological entropy. 
For an example consider the topological space $X$ obtained from $\Z$, equipped with the discrete topology and an extra point $a$, so that $X = \Z \cup \{a\}$ has $a$ as an isolated dense point, i.e., $\{\{a,n\}: n  \in \Z\}$ is a base of the topology of $X$. Clearly, $X$ is a compact $T_0$-space whose $T_1$-reflection is a singleton, as $\{a\}$ is dense in $X$. Let $\f: X \to X$ be defined by $\f(a) = a$ and $\f(n) = n + 1$ for all $n \in \Z$. Then $\f$ is a homeomorphism having $\Z$ as a closed invariant subspace. Since $\hf(\phi\restriction_\Z)=\infty$ (see \cite{Ellis,Hof}), we have that $\hf(\f) =\infty$.

\smallskip
For convenience we give the following obvious corollary of the theorem about the restriction
\begin{equation}\label{TGG}
r':\mathbf{CTG}\to \mathbf{CG}
\end{equation}
of the $T_0$-reflection $r:\mathbf{Top}\to \mathbf{Top}_0$ to the subcategory $\newsym{category of compact topological groups and continuous homomorphisms}{\mathbf{CTG}}$ of  compact topological groups (recall that $T_0$-topological groups are Hausdorff): 

\begin{Corollary}\label{reflection:corollary} 
For the functor \eqref{TGG} the pair $(\h_{\fc},\h_{\fc})$ satisfies $SBT_{r'}$. In particular, $(\hf,\hf)$ satisfies $BT_{r'}$. Hence, $(h_{top},h_{top})$ satisfies $BT_{r'}$.
\end{Corollary}

The last assertion of the above corollary follows from Remark~\ref{fcrem} (see also Example~\ref{remex}(a)). It shows that it makes sense to consider the topological entropy for compact groups that are Hausdorff (indeed, in this paper compact groups are usually intended to be Hausdorff). 

\medskip
Denote by $\mathbf{TAG}$ the full subcategory of $\mathbf{AG}$ whose objects are all torsion abelian groups. The coreflection 
\begin{equation}\label{funtor:t}
t:\mathbf{AG}\to \mathbf{TAG}
\end{equation}
 that assigns to every abelian group $G$ its torsion subgroup $t(G)$ preserves the algebraic entropy $\ent$ in the sense of the first equality in \eqref{H---}, since $\sub(G)=\sub(t(G))$. 

\begin{Theorem} 
The pair $(\h_\sub,\h_\sub)$ satisfies $SBT_t$, where $t$ is as in \eqref{funtor:t}. In particular, $(\ent,\ent)$ satisfies $BT_t$.
\end{Theorem}
\begin{proof}
Let $\phi:G\to G$ in $\mathbf{AG}$. The coreflection $t$ assigns to $G$ the embedding $t_G:t(G)\to G$. Then $\sub(t_G):\sub(t(G))\to \sub(G)$ is an isomorphism in $\Se$ such that $\sub(t\phi)$ and $\sub(\phi)$ are conjugate by $\sub(t_G)^{-1}$. So it suffices to take $\eta_G=\sub(t_G)^{-1}$.
For the second part,  $\ent=\h_{\sub}$ by Theorem~\ref{halg-sub}.
\end{proof}

Let \newsym{category of residually finite abelian groups and group homomorphisms}{$\mathbf{RFAG}$} be the full subcategory of $\mathbf{AG}$ with objects all residually finite groups. For an abelian group $G$ the first Ulm subgroup is 
$$
G^1=\bigcap_{m>0} mG.
$$
 The assignment $G\mapsto G/G^1$ gives a  reflection $$r:\mathbf{AG}\to\mathbf{RFAG}.$$ Since for every abelian group $G$ the canonical homomorphism $q: G \to G/G^1$ induces an isomorphism $\mathfrak{sub}^\star(q): \mathfrak{sub}^\star(G/G^1) \to \mathfrak{sub}^\star(G)$, we have the following theorem (its proof is analogous to that of Theorem~\ref{reflection}).

\begin{Theorem}
Consider $r:\mathbf{AG}\to\mathbf{RFAG}$. The pair $(\h_{\sub^\star},\h_{\sub^\star})$ satisfies $SBT_r$. In particular, $(\ent^\star,\ent^\star)$ satisfies $BT_r$.
\end{Theorem}

\subsection{Bridge Theorems for the algebraic entropy and the topological entropy}\label{BTat}

For a locally compact abelian group $G$ the Pontryagin dual $\widehat G$ is the group of all continuous character $\chi:G\to \T$, endowed with the compact-open topology; moreover, for a continuous endomorphism $\phi:G\to G$, its dual endomorphism $\widehat\phi:\widehat G\to\widehat G$ is continuous (see \cite{Po,HR}). This gives the Pontryagin duality functor $$\ \widehat{}: \mathbf{LCA} \to \mathbf{LCA},$$ where $\mathbf{LCA}$ is the category of all locally compact abelian groups and all continuous homomorphisms. The Pontryagin duality functor is invertible and coincides with its inverse (up to natural equivalence), by Pontryagin-van Kampen duality theorem.

For a subset $A$ of $G$, the annihilator of $A$ in $\widehat G$ is $$A^\perp=\{\chi\in\widehat G:\chi(A)=0\},$$ while for a subset $B$ of $\widehat G$, the annihilator of $B$ in $G$ is $$B^\top=\{x\in G:\chi(x)=0\ \text{for every }\chi\in B\}.$$ Clearly, ${}^\perp$ reverses the inclusions.

\medskip
The following Bridge Theorem was proved in \cite{DGS}. We consider the restrictions of the Pontryagin duality functor 
$$\ \widehat{}: \AG \to \CAG\quad\text{and}\quad \widehat{}: \CAG \to \AG,$$
which are one inverse to each other.

\begin{Theorem}
\begin{itemize}
\item[(a)] Consider $\ \widehat{}:\AG \to \CAG$.
The pair $(\h_{\mathfrak{sub}^\star},\h_{\mathfrak{sub}})$ satisfies $SBT_{\;\widehat{}}$. In particular, the pair $(\ent^\star,\ent)$ satisfies $BT_{\;\widehat{}}$.
\item[(b)] Consider $\ \widehat{}:\CAG \to \AG$.
The pair $(\h_{\mathfrak{sub}},\h_{\mathfrak{sub}^\star})$ satisfies $SBT_{\;\widehat{}}$. In particular, the pair $(\ent,\ent^\star)$ satisfies $BT_{\;\widehat{}}$.
\end{itemize}
\end{Theorem}
\begin{proof}
(a) Let $G$ be an abelian group and $\f: G\to G$ an endomorphism.
The semilattice isomorphism $\mathcal C(G) \to \mathcal F(\widehat G) $ given by $N \mapsto N^\perp$ preserves the norms, so it is an isomorphism in $\Se$, and it suffices to take $\eta={}^\perp$. 
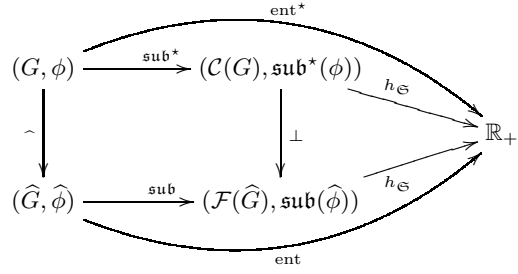
\begin{equation*}
\xymatrix@R=10pt@C=40pt{
(G,\phi)\ar[dd]_{\widehat{}\;} \ar[r]^{\mathfrak{sub}^\star}\ar@/^2.5pc/[rrd]^{\ent^\star}& (\mathcal C(G),\mathfrak{sub}^\star(\phi)) \ar[dr]^{h_\Se}\ar[dd]^{\perp} & \\
 & &  \mathbb R_+ \\
(\widehat G,\widehat \phi) \ar[r]^{\mathfrak{sub}}\ar@/_2.5pc/[rru]_{\ent} & (\mathcal F(\widehat G),\mathfrak{sub}(\widehat\phi))  \ar[ur]_{h_\Se}
}
\end{equation*}
For the second statement, note that $\ent^\star=\h_{\mathfrak{sub}^\star}$ and $\ent=\h_{\mathfrak{sub}}$ by Theorems~\ref{aent} and~\ref{halg-sub}.

(b) Follows from (a) and Remark~\ref{remark**}(b). 
\end{proof}

The following result, covering (and inspired by) Weiss' Bridge Theorem from \cite{W}, is our leading example. The proof follows the idea of the original one. The restrictions of the Pontryagin duality functor 
$$\widehat{}:\mathbf{TAG} \to \mathbf{TdCAG}\quad\text{and}\quad\widehat{}:\mathbf{TdCAG} \to \mathbf{TAG}$$
are inverse to each other.

\begin{Theorem}\label{WeissBT}
Consider $\ \widehat{}:\mathbf{TAG} \to \mathbf{TdCAG}$. The pair $(\h_{\sub},\h_{\cov})$ satisfies $SBT^w_{\;\widehat{}\;}$.
In particular, the pair $(\ent,h_{top})$ satisfies $BT_{\;\widehat{}\;}$, and for $\ \widehat{}:\mathbf{TdCAG} \to \mathbf{TAG}$, the pair $(\ent,h_{top})$ satisfies $BT_{\;\widehat{}\;}$.
\end{Theorem}
\begin{proof}
Let $\f:G\to G$ be an endomorphism of a torsion abelian group $G$. Then $\widehat G$ is a totally disconnected compact abelian group and $\widehat \f: \widehat G \to \widehat G$ a continuous endomorphism. 
The semilattice isomorphism $\mathcal F(G) \to \mathfrak{sub}_o^\star(\widehat G) $ given by $N \mapsto N^\perp$ preserves the norms, so it is an isomorphism in $\Se$. 
By Claim~\ref{Claim18March}, there exists a weak isomorphism $\iota_G:\mathfrak{sub}^\star_o(\widehat G)\to \mathfrak{cov}(\widehat G)$ with coefficient $1$.
Therefore, $\eta_G=\iota_G \circ \bot : \mathcal F(G) \to \mathfrak{cov}(\widehat G)$ is a weak isomorphism with coefficient $1$.
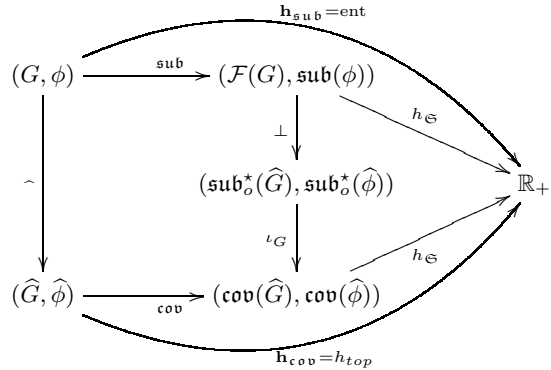
\begin{equation*}\label{W}
\xymatrix@R=25pt@C=40pt
{(G,\phi)\ar[dd]_{\widehat{}\;}\ar[r]^{\mathfrak{sub}}\ar@/^3.3pc/[rrd]^{\mathbf{h}_{\mathfrak{sub}}= \ent}&(\sF(G),\mathfrak{sub}(\phi))\ar[d]_{\bot}\ar[dr]^{h_\Se}&\\
& (\mathfrak{sub}^\star_o(\widehat G),\mathfrak{sub}^\star_o(\widehat \phi)) \ar[d]_{\iota_G}&\R_+	\\
(\widehat G, \widehat \phi)\ar[r]_{\mathfrak{cov}}\ar@/_3.3pc/[rru]_{\mathbf{h}_{\mathfrak{cov}} = h_{top}}&(\cov(\widehat  G),\mathfrak{cov}(\widehat \phi)) \ar[ur]_{h_\Se} &}
\end{equation*}
The second statement follows from the fact that $\ent=\h_{\sub}$ and $h_{top}=\h_{\cov}$ by Theorems~\ref{halg-sub} and~\ref{realization:top:ent}, and from Remark~\ref{remark**}(b).
\end{proof}

\begin{question}
Does the pair $(\h_\cov,\h_\sub)$ satisfy $SBT_{\;\widehat{}\;}$?
\end{question}

The conclusion of Theorem~\ref{WeissBT} that the pair $(\ent,h_{top})$ satisfies $BT_{\;\widehat{}\;}$ can be obtained in an easier way by chosing a different pair of functors.
This option, exploited in Theorem~\ref{WBT}, has also the advantage of producing a Strong$^w$ Bridge Theorem (consequently, also a Bridge Theorem) in both directions $\ \widehat{}:\mathbf{TAG} \to \mathbf{TdCAG}$ and $\ \widehat{}:\mathbf{TdCAG} \to \mathbf{TAG}$, whereas Theorem~\ref{WeissBT} provides a Strong$^w$ Bridge Theorem only for the first functor. 

\begin{Theorem}\label{WBT}
\begin{itemize}
\item[(a)] Consider $\ \widehat{}:\mathbf{TAG} \to \mathbf{TdCAG}$. The pair $(\h_{\sub},\h_{\mathfrak{sub}_o^\star})$ satisfies $SBT_{\;\widehat{}\;}$.
In particular, the pair $(\ent,h_{top})$ satisfies $BT_{\;\widehat{}\;}$.
\item[(b)] Consider $\ \widehat{}:\mathbf{TdCAG} \to \mathbf{TAG}$. The pair $(\h_{\mathfrak{sub}_o^\star},\h_{\sub})$ satisfies $SBT_{\;\widehat{}\;}$.
In particular, the pair $(h_{top},\ent)$ satisfies $BT_{\;\widehat{}\;}$ with $C=1$.
\end{itemize}
\end{Theorem} 
\begin{proof}
(a) Let $G$ be a torsion abelian group and $\f: G\to G$ an endomorphism. The semilattice isomorphism $\mathcal F(G)\to \mathfrak{sub}_o^\star(\widehat G)$ given by $N \mapsto N^\perp$ preserves the norms, so it is an isomorphism in $\Se$. Hence it suffices to take $\eta={}^\perp$.
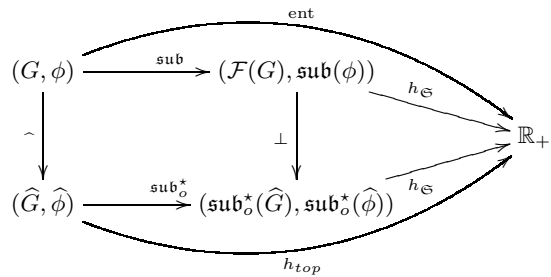
\begin{equation*}
\xymatrix@R=10pt@C=40pt{
(G,\phi)\ar[dd]_{\widehat{}\;} \ar[r]^{\mathfrak{sub}}\ar@/^2.5pc/[rrd]^{\ent}& (\mathcal F(G),\mathfrak{sub}(\phi)) \ar[dd]_{\perp}\ar[dr]^{h_\Se} & \\
 & &  \mathbb R_+ \\
(\widehat G,\widehat \phi) \ar[r]^{\mathfrak{sub}_o^\star}\ar@/_2.5pc/[rru]_{h_{top}} & (\mathfrak{sub}^\star_o(\widehat G),\mathfrak{sub}^\star_o(\widehat\phi)) \ar[ur]_{h_\Se}
}
\end{equation*}
For the second statement, note that $\ent=\h_{\mathfrak{sub}}$ and $h_{top}=\h_{\mathfrak{sub}^\star_o}$ by Theorems~\ref{halg-sub} and~\ref{htopsubo}. 

(b) Follows from (a) in view of Remark~\ref{remark**}.
\end{proof}

The general Bridge Theorem recalled in Theorem~\ref{BT1} can be stated as follows.

\begin{Theorem}\label{BTgen} 
For the functor $\ \widehat{}: \AG \to \CAG$ and its inverse $\ \widehat{}: \CAG \to \AG$, the pairs $(h_{alg},h_{top})$ and $(h_{top},h_{alg})$ satisfy $BT_{\;\widehat{}\;}$.
\end{Theorem}

On the other hand, it is not known whether the respective functorial entropies satisfy the Strong Bridge Theorem.

\begin{question}
Do the pairs $(\h_\pet,\h_\cov)$ and $(\h_\cov,\h_\pet)$ satisfy $SBT_{\;\widehat{}\;}$ or $SBT^w_{\;\widehat{}\;}$?
\end{question}

The following Bridge Theorem was proved in \cite{CGBtop} in the more general case of locally linearly compact vector spaces over a discrete field $\mathbb K$. We consider the restrictions of the Lefschetz duality functor from the category of locally linearly compact vector spaces to the cases 
\begin{equation}\label{Aug2}
\ \widetilde{}: \mathbf{LCVect_\mathbb K} \to \mathbf{Mod}_\mathbb K\quad\text{and}\quad \widetilde{}:  \mathbf{Mod}_\mathbb K \to \mathbf{LCVect_\mathbb K},
\end{equation}
which are inverse to each other. 
 
\begin{Theorem}
For the functors \eqref{Aug2}:
\begin{itemize}
\item[(a)] the pair $(\h_{\mathfrak{sub}_o^\star},\h_{\mathfrak{sub}_d})$ satisfies $SBT_{\;\widetilde{}}$, so in particular, $(\ent_{\dim}^\star,\ent_{\dim})$ satisfies $BT_{\;\widetilde{}}$;
\item[(b)] the pair $(\h_{\mathfrak{sub}_d},\h_{\mathfrak{sub}^\star_o})$ satisfies $SBT_{\;\widetilde{}}$, so in particular,  $(\ent_{\dim},\ent^\star_{\dim})$ satisfies $BT_{\;\widetilde{}}$.
\end{itemize}
\end{Theorem}
\begin{proof}
(a) Let $V$ be a linearly compact vector space over $\mathbb K$ and let $\f: V\to V$ be a continuous endomorphism. The semilattice isomorphism 
$\sF_o(V) \to \sF_d(\widetilde V)$ given by 
$$N \mapsto N^\perp=\{\chi\in\mathrm{CHom(V,\mathbb K)}:\chi(N)=0\}$$
 preserves the norms, by Lefschetz duality theory, so it is an isomorphism in $\Se$ (see \cite{CGBtop} for the details). Then it suffices to take $\eta={}^\perp$. 
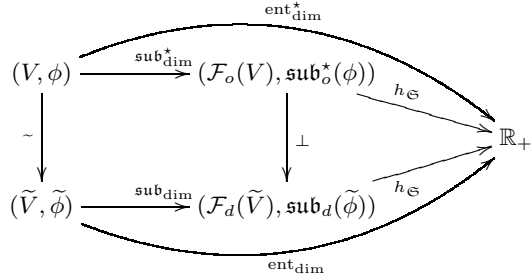
\begin{equation*}
\xymatrix@R=10pt@C=40pt{
(V,\phi)\ar[dd]_{\widetilde{}\;} \ar[r]^{\mathfrak{sub}_{\dim}^\star}\ar@/^2.5pc/[rrd]^{\ent_{\dim}^\star}& (\sF_o(V),\mathfrak{sub}_o^\star(\phi)) \ar[dr]^{h_\Se}\ar[dd]^{\perp} & \\
 & &  \mathbb R_+ \\
(\widetilde V,\widetilde \phi) \ar[r]^{\mathfrak{sub}_{\dim}}\ar@/_2.5pc/[rru]_{\ent_{\dim}} & (\sF_d(\widetilde V),\mathfrak{sub}_d(\widetilde\phi))  \ar[ur]_{h_\Se}
}
\end{equation*}
For the second statement, note that $\ent_{\dim}^\star=\h_{\mathfrak{sub}_{\dim}^\star}$ and $\ent_{\dim}=\h_{\mathfrak{sub}_{\dim}}$ by \eqref{entdim*eq} and \eqref{entdimeq}.

(b) Follows from (a) and Remark~\ref{remark**}(b). 
\end{proof}

\subsection{Bridge Theorems for the set-theoretic entropy}\label{BTset}

The Bridge Theorems that we analyze until the end of this section connect the topological entropy $h_{top}$ and the algebraic entropy $h_{alg}$ with the covariant set-theoretic entropy $\mathfrak h$ by means of the backward and the forward generalized shifts, respectively.

\subsubsection{Topological entropy and backward generalized shifts}

We start proving the following (Strong) Bridge Theorem between the set-theoretic and the topological entropy, then we give its counterpart for $\mathbf{CTop}$.

\begin{Theorem}\label{SetBT2*}
Let $\varepsilon: \mathbf{Set}\to \mathbf{CG}$ be a contravariant functor sending coproducts to products, and let $K = \varepsilon(*)$. Then: 
\begin{itemize}
\item[(a)] the pair $(\mathfrak h, h_{top})$ satisfies $BT_{\varepsilon,\infty}$, if $K$ is infinite; 
\item[(b)] the pair $(\h_{\atr},\h_{\mathfrak{sub}^\star_o})$ satisfies $SBT_{\varepsilon,\log|K|}^w$, if $K$ is finite.
\end{itemize}
In particular, the pair $(\mathfrak h, h_{top})$ satisfies $BT_{\varepsilon,\log|K|}$ (with the convention that $\log|K|=\infty$ if $K$ is infinite).
\end{Theorem}

In Lemma~\ref{SetBT2} we provide a proof of the theorem in the particular case of the functor \eqref{represenatble} of the backward generalized shift,
then Theorem~\ref{SetBT2*} follows immediately from Lemma~\ref{SetBT2}, Theorem~\ref{New_Lemma}, Lemma~\ref{neSBT}(a), and Remark~\ref{BTne}.

\begin{Lemma}\label{SetBT2}
Theorem~\ref{SetBT2*} holds with $\varepsilon=\mathcal B_K$, for a compact group $K$.
\end{Lemma}
\begin{proof}
Let $X$ be a non-empty set. For $F \in \mathcal{S}(X)$ let $\eta_X(F)$ be the subgroup $K^{X \setminus F} \times \{1\}^F$ of $K^X$. It is $\sigma_\lambda$-invariant, whenever $F$ is $\lambda$-invariant in $X$. 

\medskip
(a) Assume that $K$ is infinite, and let $\lambda:X\to Y$ be a map. We have to check that:
\begin{itemize}
  \item[(i)] if $\mathfrak{h}(\lambda) >0$, then $h_{top}(\sigma_\lambda)  =\infty $;
  \item[(ii)] if $\mathfrak{h}(\lambda) = 0$, then $h_{top}(\sigma_\lambda)  =0$ as well.
\end{itemize}
To check (i) assume that  $\mathfrak{h}(\lambda) >0$. Then there exists a $\lambda$-invariant infinite subset $N =\{x_0, \ldots, x_n, \ldots\}$ of $X$, where $\lambda$ acts as a right shift, i.e., $\lambda(x_n) = x_{n+1}$ for every $n\in\N$. Then the functor $ \mathcal B_K$ transforms the inclusion $j: N\to X$ into the projection $K^X \to K^N$ and the restriction $\lambda\restriction_N : N \to N$ to $\sigma_{\lambda\restriction_N}: K^N \to K^N$, which is conjugate to the left Bernoulli shift ${}_K\beta:K^\N\to K^\N$. Since $K$ is infinite, by the Monotonicity for factors, the Invariance under conjugation and \eqref{htopbeta}, we conclude that $h_{top}(\sigma_\lambda)  \geq h_{top}({}_K\beta) =\infty$. 

\smallskip
To verify (ii) assume that $\mathfrak{h}(\lambda) = 0$, i.e., every $F \in \mathcal{S}(X)$ is contained in a $\lambda$-invariant finite set $F' \in \mathcal{S}(X)$. This implies that every subgroup of the form $\eta_X(F)$ of $K^X$ contains a subgroup (of the same form) $\eta_X(F')$, with $F' \in \mathcal{S}(X)$, that is $\sigma_\lambda$-invariant. 

The group $K^X$ is the inverse limit of the inverse system of groups $$\mathcal K = \{(K^F,\pi^F_{F'}):F,F'\in \mathcal{S}(X)\}.$$ where $\pi^F_{F'}: K^F\to K^{F'}$, for 
$F\supseteq F'$ in $\mathcal{S}(X)$, is the standard projection. By what we observed above, $K^X$ is also the inverse limit of the inverse system 
$$\mathcal K_{inv}= \{(K^F,\pi^F_{F'}): F,F'\in \mathcal{S}(X),\ \lambda\text{-invariant}\}.$$ 
Obviously, $K^X/\eta_X(F)$ is topologically isomorphic to $K^F$ for every $F \in \mathcal{S}(X)$. Moreover, when $F \in \mathcal{S}(X)$ is $\lambda$-invariant, this isomorphism extends to an isomorphism of the flows $(K^X/\eta_X(F),\sigma_\lambda^F)$ and $(K^F, \sigma_{\lambda\restriction_{F}})$, where  
$$\sigma_\lambda^F:K^X/\eta_X(F)\to K^X/\eta_X(F)$$
 is the endomorphism induced by $\sigma_\lambda$; equivalently, $\sigma_\lambda^F$ and $\sigma_{\lambda\restriction_F}$ are conjugate and so $h_{top}(\sigma_\lambda^F)=h_{top}(\sigma_{\lambda\restriction_F})$.

Let $\zeta=\lambda\restriction_F: F \to F$ for the sake of brevity. 
Since $F$ is finite, we deduce there exist natural numbers $m>k$ such that $\zeta^m =\zeta^k$. By \eqref{powers}, this yields 
$$(\sigma_\zeta)^m = \sigma_{\zeta^m} = \sigma_{\zeta^k}  = (\sigma_\zeta)^k.$$ 
Therefore, $\sigma_\zeta$ is quasi-periodic. By Lemma~\ref{teo:provvisorio}, we conclude that $h_{top}(\sigma_\zeta)=0$. Therefore, $h_{top}(\sigma_\lambda^F)=0$.

This proves that $h_{top}(\sigma_\lambda)  =0$ in view of the Continuity for inverse limits of $h_{top}$.

\medskip
(b) Assume that $K$ is finite. Then $\eta_X(F) \in \mathcal V_{K^X}(1)$ for every $F \in \mathcal{S}(X)$.
\begin{equation*}
\xymatrix@R=10pt@C=40pt{
(X,\lambda) \ar[dd]_{\mathcal B_K} \ar[r]^{\atr}\ar@/^2.5pc/[rrd]^{\mathfrak{h}}& (\mathcal{S}(X),\atr({\lambda})) \ar[dr]^{h_\Se}\ar[dd]^{\eta_X} & \\
 & &  \mathbb R_+ \\
(K^X,\sigma_\lambda) \ar[r]^{\mathfrak{sub^\star_o}}\ar@/_2.5pc/[rru]_{h_{top}} & (\mathcal V_{K^X}(1),\mathfrak{sub^\star_o}(\sigma_\lambda))  \ar[ur]_{h_\Se}
}
\end{equation*}
Since every $U \in \mathcal V_{K^X}(1)$ contains $\eta_X(F)$ for some $F \in \mathcal{S}(X)$, we deduce that $\eta_X$ has cofinal image in $\mathcal V_{K^X}(1)$. Moreover, the norm of $F$ in $\mathcal{S}(X)$ is $|F|$, while the norm of $\eta_X(F)$ is
$$\log [K^X: \eta_X(F)] = \log |K|^{|F|} =  |F| \cdot \log |K|.$$
This proves that $\eta_X:\mathcal{S}(X)\to\mathcal V_{K^X}(1)$ is a weak isomorphism with coefficient $\log |K|$. Therefore, the pair $(\h_{\atr},\h_{\mathfrak{sub}^\star_o})$ satisfies $SBT_{\mathcal B_K,\log|K|}^w$. 
\end{proof}

The following is the counterpart of Theorem~\ref{SetBT2*} for the category $\mathbf{CTop}_2$ of all compact Hausdorff spaces. It covers the known formula recalled in \eqref{htoph}.

\begin{Theorem}\label{SetBT2star}
Let $\varepsilon: \mathbf{Set}\to \mathbf{CTop}_2$ be a contravariant functor sending coproducts to products, and let $K = \varepsilon(*)$. Then: 
\begin{itemize}
\item[(a)] the pair $(\mathfrak h, h_{top})$ satisfies $BT_{\varepsilon,\infty}$, if $K$ is infinite; 
\item[(b)] the pair $(\h_{\atr},\h_{\cov})$ satisfies $SBT_{\varepsilon,\log|K|}^w$, if $K$ is finite.
\end{itemize}
In particular, the pair $(\mathfrak h, h_{top})$ satisfies $BT_{\varepsilon,\log|K|}$ (with the convention that $\log|K|=\infty$ if $K$ is infinite).
\end{Theorem}

Theorem~\ref{SetBT2star} can be proved following the line of the proof of Theorem~\ref{SetBT2*}. Indeed, Theorem~\ref{SetBT2star} follows from Theorem~\ref{New_Lemma}, Lemma~\ref{neSBT}(a), Remark~\ref{BTne}, and the counterpart for $\mathbf{CTop}_2$ of Lemma~\ref{SetBT2}. This can be obtained as follows. Item (a) follows from a standard variation of the proof of item (a) of Lemma~\ref{SetBT2}. For item (b) 
consider for every finite $F$ the projection $p_F: K^X \to K^F$. As $K^F$ is finite, so discrete, the set 
$p_F^{-1}(x)$ is open in $K^X$ for every $x\in K^F$, hence such that
$$\mathcal U_{F}=\{p_F^{-1}(x) :x\in K^F\}$$
is a finite open cover of $K^X$ with $$N(\mathcal U_{F})=|K|^F.$$ 
Moreover, following the line of the proof of Claim~\ref{Claim18March}, one can prove that the family $$\{\mathcal U_{F}:F\in\mathcal S(X)\}$$ 
is cofinal in $\fc(X)$, and $\fc(X)$ is cofinal in $\cov(X)$ by Remark~\ref{fcrem}.
Hence, letting 
$$\nu_X(F)=\mathcal U_{F}$$
for every $F\in\mathcal S(X)$, one finds that $\nu$ is a natural transformation between $\atr$ and $\cov\circ\mathcal B_K$ such that, for every set $X$, $\nu_X:\mathcal S(X)\to \cov(X)$ is a weak isomorphism with coefficient $\log|K|$. Therefore, the pair $(\h_{\atr},\h_{\cov})$ satisfies $SBT_{\mathcal B_K,\log|K|}^w$.
\begin{equation*}
\xymatrix@R=10pt@C=40pt{
(X,\lambda) \ar[dd]_{\mathcal B_K} \ar[r]^{\atr}\ar@/^2.5pc/[rrd]^{\mathfrak{h}}& (\mathcal{S}(X),\atr({\lambda})) \ar[dr]^{h_\Se}\ar[dd]^{\nu_X} & \\
 & &  \mathbb R_+ \\
(K^X,\sigma_\lambda) \ar[r]^{\mathfrak{cov}}\ar@/_2.5pc/[rru]_{h_{top}} & (\cov(K^X),\cov(\sigma_\lambda))  \ar[ur]_{h_\Se}
}
\end{equation*}

\begin{Remark}\label{TSUNAMI-h*}
Let $K$ be a group and $$\mathcal B'_K: \mathbf{Set}_\mathrm{fin}\to \mathbf{Grp}$$ be the contravariant functor defined on $\mathbf{Set}_\mathrm{fin}$, sending $\emptyset$ to the terminal object $\{1\}$ of $\mathbf{Grp}$ and a non-empty set $X$ to $$\mathcal B'_K(X)=K^{(X)}.$$ For a finite-to-one map $\lambda:X\to Y$ let, as recalled in \eqref{^oplus}, $$\mathcal B'_K(\lambda)=\sigma_\lambda^\oplus:K^{(Y)}\to K^{(X)}$$ when $X$ and $Y$ are non-empty, and let $\mathcal B'_K(\lambda)=1$ when either $X$ or $Y$ is empty.

The result from \cite{DG-islam} recalled in \eqref{halgh} can be stated as follows as a Bridge Theorem, with the convention that $\log|K|=\infty$ if $K$ is infinite: 
\begin{quote}
the pair $(\mathfrak h_p^*, h_{alg})$ satisfies $BT_{\mathcal B'_K,\log|K|}$.
\end{quote}
Nevertheless, we cannot find it as a Strong Bridge Theorem, since we can find $\mathfrak h^*$ as a functorial entropy but not $\mathfrak h_p^*$.

To overcome this problem one could proceed as follows.
First, note that every functor $$F : \mathfrak X\to \mathfrak Y$$
induces a functor $$F_2 : \mathbf{Flow}_\mathfrak X\to \mathbf{Flow}_\mathfrak Y$$ assigning to each object $f: X \to X$ of 
$\mathbf{Flow}_\mathfrak X$ the object $F(f)$ of $\mathbf{Flow}_\mathfrak Y$. 
A careful analysis of our categorical approach to entropy 
shows that we need substantially, rather than the functor $$F : \mathfrak X \to \Se^*,$$ 
the functor $$F_2 : \mathbf{Flow}_\mathfrak X \to \mathbf{Flow}_{\Se^*},$$
so that a ``functorial entropy'' of $\mathfrak X$ can be obtained as a composition of $F_2$ with the entropy function $$h_{\Se^*}:\mathbf{Flow}_{\Se^*} \to \R_+\cup \{\infty\}.$$

This modified approach applies in cases when one has a convenient functor $G:  \mathbf{Flow}_\mathfrak X \to \mathbf{Flow}_{\Se^*}$
that is not necessarily of the form $G =F_2$. For example, it perfectly fits the case of the functor 
$$\mathrm{sc}:\mathbf{Flow}_{\mathbf{Set}_{\mathrm{fin}}}\to \mathbf{Flow}_{\mathbf{Set}_{\mathrm{fin}}},$$
sending the flow $(X,\lambda)$ to its subflow $(\mathrm{sc}(\lambda),\lambda\restriction_{\mathrm{sc}(\lambda)})$ (it cannot be obtained as 
$F_2$ for any functor $F: \mathbf{Set}_{\mathrm{fin}} \to \mathbf{Set}_{\mathrm{fin}}$).
Then, for the functor 
$$\str_2:\mathbf{Flow}_{\mathbf{Set}_{\mathrm{fin}}}\to\mathbf{Flow}_{\Se^*}$$ induced by the functor $\str$, let
$$G=\str_2\circ\mathrm{sc} :\mathbf{Flow}_{\mathbf{Set}_{\mathrm{fin}}}\to\mathbf{Flow}_{\Se^*}.$$
This allows us to obtain, for every $(X,\lambda)\in\mathbf{Flow}_{\mathbf{Set}_{\mathrm{fin}}}$,
$$\mathfrak h_p^*(X,\lambda)=h_{\Se^*}(G(X,\lambda)),$$
and so we can find $\mathfrak h_p^*$ as a ``functorial entropy'' as well, although in a slightly different way compared to the functorial entropy we studied so far.
This alternative approach will not be adopted in the present paper. 
\end{Remark}

\subsubsection{Algebraic entropy and forward generalized shifts}

We start proving the relation between the backward generalized shift $\sigma_\lambda:K^Y\to K^X$ induced by the map $\lambda:X\to Y$ (see Definition~\ref{bgs}  and \eqref{bgseq}), and the forward generalized shift $\tau_\lambda:\widehat K^{(X)}\to\widehat K^{(Y)}$ (see Definition~\ref{fgs} and \eqref{fgseq}). To this end we identify 
\begin{equation}\label{PD1}
\widehat{K^X}=\widehat K^{(X)}
\end{equation}
as follows. For $\chi\in \widehat{K^X}$ there exists a finite subset $F$ of $X$ such that $$\chi=\sum_{x\in F}\chi\restriction_{K_x},$$ where we denote by $K_x$ the copy of $K$ in $K^X$ corresponding to $x\in X$. Then we identify $\chi\in\widehat{K^X}$ with $(\chi\restriction_{K_x})_{x\in X}\in\widehat K^{(X)}$, noting that $\chi\restriction_{K_x}=0$ for every $x\in X\setminus F$. When $K$ is finite, $K\cong \widehat K$.

\begin{Proposition}\label{sigmatau}
Let $K$ be a compact abelian group and $\lambda:X\to Y$ be a map. Then the homomorphisms 
$$\sigma_\lambda:K^Y\to K^X\quad \text{and}\quad \tau_\lambda:\widehat K^{(X)}\to\widehat K^{(Y)}$$
satisfy 
$$\widehat{\sigma_\lambda}=\tau_\lambda.$$
\end{Proposition}
\begin{proof} 
For a character $\xi\in\widehat K$ and $x\in X$, let $\xi^x\in\widehat K^{(X)}$ be defined by $\xi^x(y) = 0$ if $y\in X\setminus \{x\}$ and $\xi^x(x)=\xi$.  In view of  the identification \eqref{PD1}, for $x_0\in X$ and $f\in K^X$ one has 
$$\xi^{x_0}(f) = \xi(f(x_0)).$$

Since $\{\xi^x:\xi\in\widehat K, x\in X\}$ is a set of generators of $\widehat K^{(X)}$, it suffices to prove that, for every $\xi\in\widehat K$ and every $x\in X$,
\begin{equation}\label{xix}
\widehat{\sigma_\lambda}(\xi^x)=\tau_\lambda(\xi^x).
\end{equation}
So, let $\xi\in\widehat K$ and $x\in X$. Then, by the definition,
$$\tau_\lambda(\xi^x)=\xi^{\lambda(x)}.$$ 
Moreover, $\widehat{\sigma_\lambda}(\xi^x)=\xi^x\circ\sigma_\lambda$, and hence, for every $f \in K^Y$,
$$\widehat{\sigma_\lambda}(\xi^x)(f)=\xi^x(\sigma_\lambda(f))=\xi^x(f \circ \lambda) =\xi ((f \circ \lambda)(x) )=\xi(f(\lambda(x)))=\xi^{\lambda(x)}(f);$$
 therefore, $$\widehat{\sigma_\lambda}(\xi^x)=\xi^{\lambda(x)}.$$
This proves the equality in \eqref{xix}, and consequently the equality $\widehat{\sigma_\lambda}=\tau_\lambda$.
\end{proof}

Fix an abelian group $K$ and consider the functor $\mathcal F_K: \mathbf{Set}\to \mathbf{AG}$ defined in \eqref{represenatble2} by means of the forward generalized shift.
Proposition~\ref{sigmatau} implies directly the following nice connection between the functors $\mathcal B_{K}$ and $\mathcal F_{\widehat K}$.

\begin{Corollary}\label{aboveclaim}
For a  compact abelian group  $K$ and the functor $\;\;\widehat{}\;:\mathbf{CAG}\to\mathbf{AG}$ one has $$\mathcal F_{\widehat K}=\; \widehat{}\;\circ\mathcal B_K.$$
\end{Corollary}

The following is the counterpart of Theorem~\ref{SetBT2*} for the algebraic entropy and the covariant set-theoretic entropy.

\begin{Theorem}\label{SetBT3*}
Let $\gamma: \mathbf{Set}\to \mathbf{AG}$ be a covariant functor sending coproducts to coproducts, and let $K = \gamma(*)$. Then: 
\begin{itemize}
\item[(a)] the pair $(\mathfrak h, h_{alg})$ satisfies $BT_{\gamma,\infty}$, if $K$ is infinite; 
\item[(b)] the pair $(\h_{\atr},\h_{\pet})$ satisfies $SBT_{\gamma,\log|K|}^w$, if $K$ is finite.
\end{itemize}
In particular, the pair $(\mathfrak h, h_{alg})$ satisfies $BT_{\gamma,\log|K|}$ (with the convention that $\log|K|=\infty$ if $K$ is infinite).
\end{Theorem}

Theorem~\ref{SetBT3*} follows directly from Lemma~\ref{SetBT3}, Theorem~\ref{New_Lemma2}, Lemma~\ref{neSBT}(a), and Remark~\ref{BTne}. 
We deduce Lemma~\ref{SetBT3} from Lemma~\ref{SetBT2} using the nice properties of the Pontryagin duality functor.

\begin{Lemma}\label{SetBT3}
Theorem~\ref{SetBT3*} holds with $\gamma=\mathcal F_K$, for an abelian group $K$.
\end{Lemma}
\begin{proof}
If $K$ is finite, then the pair $(\mathfrak h,h_{top})$ satisfies $SBT^w_{\mathcal B_K,\log|K|}$ by Lemma~\ref{SetBT2}, while $(h_{top},h_{alg})$ satisfies $SBT^w_{\;\widehat{}\;}$ (where $\;\widehat{}:\mathbf{CAG}\to\mathbf{AG}$ is the Pontryagin duality functor), by Theorem~\ref{WeissBT}. Then the pair  $(\mathfrak h, h_{alg})$ satisfies $SBT^w_{\mathcal B_K,\log|K|}$ by Corollary~\ref{aboveclaim}.
Analogously, the case when $K$ is infinite follows from Lemma~\ref{SetBT2}, Theorem~\ref{BTgen} and Corollary~\ref{aboveclaim}.
\end{proof}

\section{Topological and algebraic entropy in locally compact groups}\label{NewSec2}

In this section, in order to find the topological entropy and the algebraic entropy for locally compact groups as functorial entropies, we need to use the larger category $\Se^*$ instead of $\Se$. In fact, as noted in the Introduction, these entropies are not invariant under inversion: if $G$ is a totally disconnected locally compact group and $\phi:G\to G$ is a topological automorphism, then $h_{top}(\phi^{-1})=h_{top}(\phi)-\log\Delta(\phi)$, where $\Delta(\phi)$ is the modulus of $\phi$ (see \cite{DG-tdlc}), and the same formula holds for the algebraic entropy $h_{alg}$ when $G$ is strongly compactly covered (see \cite{GBST}). 

\medskip
Let $G$ be a locally compact group, let $\mathcal N(G)$ be the family of all compact neighborhoods of $1$ and $\mu$ be a right Haar measure on $G$. For a continuous endomorphism $\phi: G \to G$, $U\in\mathcal N(G)$ and $n\in\N_+$, the $n$-th $\phi$-cotrajectory $C_n(\phi,U)=U\cap \phi^{-1}(U)\cap\ldots\cap\phi^{-n+1}(U)$ is still in $\mathcal N(G)$. 
It can be shown that the value
$$H_{top}(\phi,U)=\limsup _{n\to \infty} - \frac{\log \mu (C_n(\phi,U))}{n},$$
is independent of the choice of the Haar measure $\mu$. The \emph{topological entropy} of $\phi$ is 
$$\newsym{topological entropy}{h_{top}}(\phi)=\sup\{H_{top}(\phi,U):U\in\mathcal N(G)\}.$$
We are using here the same notation $H_{top}$ and $h_{top}$ that we have already used in \S\ref{htop-sec} for the topological entropy of continuous selfmaps of compact spaces. 
This is safe since there is no possibility of confusion, moreover the two topological entropies coincide for continuous endomorphisms of compact groups.

\smallskip
If $G$ is discrete, then $\mathcal N(G)$ is the family of all finite subsets of $G$ containing $1$, and $\mu(A) = |A|$ for subsets $A$ of $G$. So $H_{top}(\phi,U)= 0$ for every $U \in \mathcal N(G)$, hence $h_{top}(\phi)=0$. 

\smallskip
To obtain the topological entropy $h_{top}(\phi)$ of a continuous endomorphism $\phi: G \to G$ of a locally compact group via the semigroup entropy, let $\mathcal V(G)$ be the family of all closed neighborhoods of $1$ in $G$.  Then 
$$\mathcal N(G)\subseteq \mathcal V(G).$$
For $U\in\mathcal V(G)$, let $$v_T(U)=\begin{cases}-\log\mu(U) & \text{if}\  \mu(U)\leq 1\\ 0 &\text{otherwise};\end{cases}$$
note that the second case includes also the possibility $\mu(U)=\infty$. Then $(\mathcal V(G),\cap, v_T)$ is a normed semilattice.  
We order $\mathcal V(G)$ by the containment, that is, $(\mathcal V(G),\supseteq)$. Then the norm $v_T$ is monotone with respect to this order and $\mathcal N(G)$ is cofinal in $(\mathcal V(G),\supseteq)$.

Moreover, we associate to a continuous endomorphism $\phi:G\to G$ the semigroup endomorphism $\phi_T:\mathcal V(G)\to\mathcal V(G)$ defined by $\phi_T(U)=\phi^{-1}(U)$ for every $U\in\mathcal V(G)$. The assignments $G \mapsto \mathcal V(G)$ and $\phi\mapsto\phi_T$ define a contravariant functor $$\boxed{\newsym{functor $\mathfrak{lctop}:\mathbf{LCG}\to \mathfrak{S}^*$}{\mathfrak{lctop}}:\mathbf{LCG}\to \mathfrak{S}^*.}$$

\begin{Theorem}
Let $G$ be a locally compact group and $\phi:G\to G$ a continuous endomorphism. Then $H_{top}(\phi,U)=\H_{\mathfrak{lctop}}(\phi,U)$ for every $U\in\mathcal N(G)$, and so $$h_{top}(\phi)=\h_{\mathfrak{lctop}}(\phi).$$
\end{Theorem}
\begin{proof} 
Let $U\in \mathcal N(G)$. For every $n\in\N_+$ we have that $$T_n(\phi_T,U)=C_n(\phi,U);$$
applying the definitions we can conclude that $-\log\mu(C_n(\phi,U))=c_n(\phi_T, U)$. So, 
\begin{equation}\label{UU}
H_{top}(\phi,U)=h_{\Se^*}(\phi_T,U)=\H_{\mathfrak{lctop}}(\phi,U).
\end{equation}
 As $\mathcal N(G)$ is cofinal in $(\mathcal V(G),\supseteq)$, now \eqref{UU} implies that $h_{top}(\phi)=\h_{\mathfrak{lctop}}(\phi)$. 
\end{proof}

\medskip
Let $G$ be a locally compact group, $\mu$ a right Haar measure on $G$ and $\phi:G\to G$ a continuous endomorphism. To define the algebraic entropy of $\phi$ with respect to $U\in\mathcal N(G)$ one uses the {$n$-th $\phi$-trajectory} $T_n(\phi,U)=U\cdot \phi(U)\cdot \ldots\cdot \phi^{n-1}(U)$ of $U$, that still belongs to $\mathcal N(G)$; the term ``algebraic'' is motivated by the fact that the definition of $T_n(\phi,U)$ (unlike $C_n(\phi,U)$) makes use of the group operation.

It turns out that the value 
\begin{equation}\label{**}
H_{alg}(\phi,U)=\limsup_{n\to \infty} \frac{\log \mu (T_n(\phi,U))}{n}
\end{equation}
does not depend on the choice of $\mu$. It can be shown that  $H_{alg}(\phi,U)< \infty$.
The \emph{algebraic entropy} of $\phi$ is $$\newsym{algebraic entropy}{h_{alg}}(\phi)=\sup\{H_{alg}(\phi,U):U\in\mathcal N(G)\}.$$
(What Peters defined actually was, in our notation, $h_{alg}(\phi^{-1})$ for a topological automorphism $\phi$; it is denoted by $h_{\infty}(\phi)$ in \cite{Pet1}.) 

As we saw above \eqref{**} is a limit when $G$ is discrete. Moreover, if $G$ is compact, then $h_{alg}(\phi)=H_{alg}(\phi,G)=0$.

\smallskip
Let $\mathcal K(G)$ be the family of all compact subsets of $G$ containing $1$. Then $$\mathcal N(G)=\mathcal K(G)\cap\mathcal V(G).$$
For $U\in\mathcal K(G)$, let $$v_A(U)=\begin{cases}\log\mu(U) & \text{if}\  \mu(U)\geq 1\\ 0 &\text{otherwise}.\end{cases}$$
Then $(\mathcal K(G),\cdot, v_A)$ is a normed semigroup containing $\mathcal N(G)$ as a subsemigroup. 
We consider on $\mathcal K(G)$ the order given by the containment, that is, $(\mathcal K(G),\subseteq)$.
Then the norm $v_A$ is monotone with respect this order and $\mathcal N(G)$ is cofinal in $(\mathcal K(G),\subseteq)$. 

Moreover, we associate to a continuous endomorphism $\phi:G\to G$ the semigroup endomorphism $\phi_A:\mathcal K(G)\to\mathcal K(G)$ defined by $\phi_A(U)=\phi(U)$ for every $U\in\mathcal K(G)$.
The assignments $G \mapsto \mathcal K(G)$ and $\phi\mapsto\phi_A$ define a contravariant functor 
$$\boxed{\newsym{functor $\mathfrak{lcalg}:\mathbf{LCG}\to \mathfrak{S}^*$}{\mathfrak{lcalg}}:\mathbf{LCG}\to \mathfrak{S}^*.}$$

\begin{Theorem}
Let $G$ be a locally compact group and $\phi:G\to G$ a continuous endomorphism. Then $H_{alg}(\phi,U)=\H_{\mathfrak{lcalg}}(\phi,U)$ for every $U\in\mathcal N(G)$, and so $$h_{top}(\phi)=\h_{\mathfrak{lcalg}}(\phi).$$
\end{Theorem}

\begin{proof} Let  $U\in \mathcal N(G)$. For every $n\in\N_+$ we have that
$$T_n(\phi_A,U)=T_n(\phi,U),$$
so applying the definitions we can conclude that $\log\mu(T_n(\phi,U))=c_n(\phi_A, U)$. Therefore, 
\begin{equation}\label{lalala}
H_{alg}(\phi,U)=h_{\Se^*}(\phi_A,U)=\H_{\mathfrak{lcalg}}(\phi,U).
\end{equation} 
As $\mathcal N(G)$ is cofinal in $(\mathcal K(G),\subseteq)$, now \eqref{lalala} implies that $h_{alg}(\phi)=\h_{\mathfrak{lcalg}}(\phi)$. 
\end{proof}

\medskip
Hereinafter we speak about the (Strong) Bridge Theorem, even if the functors considered in this section have $\Se^*$, and no more $\Se$, as a target. The above definitions can be generalized in the obvious way, so we are not going to do that, although we shall use $SBT_{\;\widehat{}\;}$ in the obvious sense. 

\smallskip
Considering the restriction
\begin{equation}\label{CCLCA}
\widehat{}: \mathbf{TdLCA} \to \mathbf{CcLCA}
\end{equation}
of the Pontryagin duality functor, where the objects of \newsym{category of totally disconnected locally compact abelian groups and continuous homomorphisms}{$\mathbf{TdLCA}$} are all totally disconnected locally compact abelian groups and the objects of \newsym{category of compactly covered locally compact abelian groups and continuous homomorphisms}{$\mathbf{CcLCA}$} are all compactly covered locally compact abelian groups, Theorem~\ref{BT2} reads as follows:

\begin{Theorem}
For the functor \eqref{CCLCA} the pair $(h_{top},h_{alg})$ satisfies $BT_{\;\widehat{}\;}$. 
Similarly, the pair $(h_{alg},h_{top})$ satisfies $BT_{\;\widehat{}\;}$ for the inverse functor $\ \widehat{}: \mathbf{CcLCA} \to \mathbf{TdLCA}$.
\end{Theorem}

We leave open the following problem.

\begin{question} 
Does the pair $(\h_\mathfrak{lctop},\h_\mathfrak{lcalg})$ satisfy $SBT_{\;\widehat{}\;}$ for the functor \eqref{CCLCA}? Does the pair $(\h_\mathfrak{lctop},\h_\mathfrak{lcalg})$ satisfy $SBT_{\;\widehat{}\;}$?
\end{question}

The validity of the Bridge Theorem for locally compact abelian group is open in general:

\begin{question}
Consider $\ \ \widehat{}: \mathbf{LCA} \to \mathbf{LCA}$. Does the pair $(h_{top},h_{alg})$ satisfy $BT_{\;\widehat{}\;}$?
\end{question}

We leave open also the following more general question.

\begin{question}
Consider $\ \ \widehat{}: \mathbf{LCA} \to \mathbf{LCA}$. Does the pair $(\h_\mathfrak{lctop},\h_\mathfrak{lcalg})$ satisfy $SBT_{\;\widehat{}\;}$? Does the pair $(\h_\mathfrak{lctop},\h_\mathfrak{lcalg})$ satisfy $SBT_{\;\widehat{}\;}$?
\end{question}

Indeed, we hope that considering $h_{alg}$ and $h_{top}$ as functorial entropies may help in proving the most general version of the Bridge Theorem between these two entropies.

\printindex 

\newpage

\listofsymbols

\end{document}